\documentclass{amsart}
\usepackage{amssymb}
\usepackage{latexsym}
\usepackage{amsmath}
\usepackage{euscript}

\def\cS{{\mathcal S}}   \def\cN{{\mathcal N}}  \def\cK{{\mathcal K}}        \def\cH{{\mathcal N}}   \def\sN{{\mathfrak N}}
\def\sH{{\mathfrak H}}  \def\dC{{\mathbb C}}   \def\dN{{\mathbb N}}         \def\dR{{\mathbb R}}
\def\dom{{\rm dom\,}}   \def\mul{{\rm mul\,}}  \def\ker{{\rm ker\,}}        \def\ran{{\rm ran\,}}
\def\cdom{{\rm \overline{dom}\,}}              \def\cspan{{\rm \overline{span}\, }}     \def\cran{{\rm \overline{ran}\,}}
\def\clos{{\rm clos\,}} \def\sgn{{\rm sgn\,}}
     \def\IM{{\rm Im\,}}
\def\Log{{\rm Log\,}}
\def\supp{{\rm supp\,}}
\def\wt#1{{{\widetilde #1} }}
\def\wh#1{{{\widehat #1} }}
\def\cmr{{\dC \setminus \dR}}
\def\uphar{{\upharpoonright\,}}
\DeclareMathOperator{\hplus}{\, \widehat + \,}

\newtheorem{theorem}{Theorem}[section]
\newtheorem{proposition}[theorem]{Proposition}
\newtheorem{corollary}[theorem]{Corollary}
\newtheorem{lemma}[theorem]{Lemma}
\theoremstyle{definition}
\newtheorem{definition}[theorem]{Definition}
\newtheorem{example}[theorem]{Example}
\newtheorem{remark}[theorem]{Remark}
\numberwithin{equation}{section}

\title[Products of generalized Nevanlinna functions]{Products of generalized Nevanlinna functions with symmetric rational functions}

\author{S. Hassi}
\author{H. L. Wietsma}
\date{}
\subjclass{Primary 30E20, 47A45, 47B25; Secondary 47A06, 47A10, 47B50}
\keywords{Generalized Nevanlinna function, Hilbert space, Pontryagin space,
selfadjoint operator, boundary triplet, Weyl function, realization}

\address{Department of Mathematics and Statistics\\
University of Vaasa\\
P.O. Box 700, FI-65101 Vaasa\\
Finland}
\email{sha@uwasa.fi; rwietsma@uwasa.fi}

\begin{document}

\begin{abstract}
New classes of generalized Nevanlinna functions, which under
multiplication with an arbitrary fixed symmetric rational function
remain generalized Nevanlinna functions, are introduced.
Characterizations for these classes of functions are established by
connecting the canonical factorizations of the product function and
the original generalized Nevanlinna function in a constructive
manner. Also a detailed functional analytic treatment of these
classes of functions is carried out by investigating the connection
between the realizations of the product function and the original
function. The operator theoretic treatment of these realizations is
based on the notions of rigged spaces, boundary triplets, and
associated Weyl functions.
\end{abstract}

\maketitle


\section{Introduction}
In the fifties M.G.~Kre\u{\i}n introduced the classes of Stieltjes
and inverse Stieltjes functions, denoted here by $\cS$ and
$\cS^{-1}$, as subclasses of the class of Nevanlinna functions.
These functions were introduced in connection with investigations on
the theory of the generalized resolvents and the theory of spectral
functions of strings. M.G.~Kre\u{\i}n showed that these classes have
a simple characterization:
\[ f(z) \in \cS \quad \Longleftrightarrow \quad f(z) , zf(z) \in \cN_0,\]
with a similar characterization for the class of inverse Stieltjes
functions. Here $\cN$ denotes the class of (ordinary) Nevanlinna
functions. In the seventies the class of Nevanlinna functions was
generalized by M.G.~Kre\u{\i}n and H. Langer to a class of
generalized Nevanlinna functions with $\kappa$ negative squares,
denoted by $\cN_\kappa$; see \cite{KL73,KL75}. One specific subclass
of $\cN_\kappa$, which was initially introduced for solving
indefinite analogues of the Stieltjes moment problem, is the class
$\cN_\kappa^+$ defined as
\[
 f(z) \in \cN_\kappa^+ \quad \Longleftrightarrow \quad f(z) \in \cN_\kappa, zf(z) \in \cN_0,
\]
see \cite{KL75}. This class extends the class of Stieltjes functions; it is
known to be connected with the theory of spectral functions of a generalized
string, see \cite{LW98}.

Further generalizations of the classes of Stieltjes (and inverse
Stieltjes) classes are due to V.A. Derkach and M.M. Malamud:
\[ f(z) \in \cS_\sH^{+\kappa}(\alpha,\beta) \quad \Longleftrightarrow \quad f(z) \in \cN, \frac{z-\beta}{z-\alpha}f(z) \in \cN_\kappa,\]
where $-\infty \leq \alpha < \beta < \infty$ (if $\alpha =-\infty$,
then $z+ \infty$ should be interpreted as being one). These classes
were introduced to describe the selfadjoint extensions of a
symmetric operator which have spectral gaps in their spectra, see
\cite{DM91}. V.A. Derkach extended the
Stieltjes (and inverse Stieltjes) classes also in a different direction:
\[ f(z) \in \cN_k^{+\kappa} \Longleftrightarrow \quad f(z) \in \cN_k, zf(z) \in \cN_\kappa, \]
see \cite{Der95}. All the preceding classes of complex functions can be seen as
special cases of the classes $\cN_\kappa^{\tilde \kappa}(r)$, which
will be investigated in the present paper. The definition of the
class $\cN_\kappa^{\tilde \kappa}(r)$ involves an arbitrary
symmetric rational function $r$ and two indices $\kappa,{\tilde
\kappa}$, which are used to describe the number of negative squares
of the Nevanlinna kernels associated with the functions $Q$ and
$rQ$:
\begin{equation*}
\cN_\kappa^{\tilde \kappa}(r) = \{\, Q \in \cN_\kappa: rQ \in \cN_{\tilde{\kappa}} \,\}.
\end{equation*}
All the above mentioned special cases involve a simple symmetric
rational function $r$ of degree one; namely either $r$ or $1/r$ has
the form $\frac{z-b}{z-a}$ or $z-b$, $a,b \in \dR$.

The aim of this paper is to describe the characteristic properties
of functions in the classes $\cN_\kappa^{\tilde \kappa}(r)$ and to
investigate relevant functional analytic connections they have in
the area of operator and spectral theory. The applicability of these
functions is covering topics, which occur in the
literature for some elementary rational functions $r$. The results
extend various previously known facts, which has been
established for some special cases of the classes
$\cN_\kappa^{\tilde \kappa}(r)$.

The methods used in the present paper differ from those appearing in
the above mentioned papers in some special cases; the approach is
constructive and a basic tool here is the canonical factorization of
generalized Nevanlinna functions, see \cite{DLLS00} (cf.
also\cite{DHS99} or Proposition~\ref{CharNK} below). Although from a
general point of view the framework involves notions from indefinite
inner product spaces, the approach used here reduces the main
problems to analogous problems for the classes $\cN_0^{0}(\wt r)$,
where no indefiniteness occurs. Here $\wt r$ can be taken to be a
symmetric rational function with real zeros and poles of order at
most two (see Theorem~\ref{Newth}).

In this paper systematically track is kept off the connections
between the canonical factorization of the functions $Q$ and $rQ$
(see Theorems~\ref{ResM1E}~and~\ref{ResTot}). This gives a basis for
the investigations concerning the realizations for the functions $Q$
and $rQ$ and the functional analytic connections between these
realizations. The key observation in establishing realizations for
the functions $Q$ and $rQ$ is based on some local integrability
properties of these functions, involving so-called Kac-Donoghue
classes of Nevanlinna functions (see
Propositions~\ref{charclassspectral} and~\ref{vorige}). This
motivates a study of certain local versions of rigged Hilbert spaces
associated with selfadjoint relations (see in particular
Sections~\ref{sec52},~\ref{secext}). The functional analytic
connection is first made precise in the abstract case (see
Theorems~\ref{eerstestap},~\ref{laatstestap} and
Proposition~\ref{DHmoldel}) and thereafter it is made more
accessible by means of (the minimal) $L^2(d\sigma)$-models
(Theorem~\ref{l2finfin}).

The contents of the paper are now outlined. In Section~\ref{sec2}
some preliminary results on (generalized) Nevanlinna functions and
their realization are stated. In Section~\ref{sec3} Kac-Donoghue
classes of Nevanlinna functions are introduced and Nevanlinna
functions with spectral gaps are characterized. Using these results
the classes $\cN_\kappa^{\wt \kappa}(r)$ are characterized in
Section~\ref{sec4}. Local versions of rigged spaces are introduced
in Section~\ref{sec5}, and they are used to obtain realizations for
functions in the Kac-Donoghue classes. Finally, Section~\ref{sec6}
uses all the preparations to connect the realizations of the
functions $Q$ and $rQ$ when $Q \in \cN_\kappa^{\wt \kappa}(r)$.

\section{Preliminary results}\label{sec2}

Some facts on scalar (generalized) Nevanlinna functions and linear
relations are recalled, and a realization for Nevanlinna
functions by boundary triplets is given.

\subsection{Nevanlinna functions.}
Recall that a (scalar) complex function $f$ is called \textit{symmetric} if $f(\overline{z}) = \overline{f(z)}$. A symmetric function $Q$ belongs to the class of \textit{(ordinary) Nevanlinna functions}, i.e. $Q \in \cN$, if $Q$ is holomorphic on $\cmr$ and $\IM(Q(z))/\IM(z) \ge 0$ for all $z \in \cmr$.
Equivalently, $Q$ is a Nevanlinna function if and only if there
exists $\alpha \in \dR$, $\beta >0$, and a nonnegative Borel measure
$d\sigma$ such that $Q$ has the following integral representation:
\begin{equation}\label{Nevfullr}
Q(z) = \alpha + \beta z + \int_{\dR \setminus \rho(Q)}\left(\frac{1}{t-z} - \frac{t}{1+t^2}\right)d\sigma(t), \quad \int_{\dR}\frac{d\sigma(t)}{1+t^2} <\infty,
\end{equation}
where $\rho(Q)$ denoted the \textit{domain of holomorphy} of $Q$. The set $\sigma(Q)$ is
defined as
\[
\sigma(Q) = \dR \setminus \rho (Q), \quad \text{if } \beta =0; \quad
\sigma(Q) = (\dR \setminus \rho(Q)) \cup \{\infty\}, \quad \text{if } \beta >0.
\]
Here $\beta>0$ is to be interpreted as a point mass at $\infty$. In
particular, if $\sigma_p(Q)$ denotes the point spectrum of $Q$, then
$\infty \in \sigma_p(Q)$ if and only if $\beta>0$. If $Q$ is
holomorphic on an interval $(c,d) \subset \dR$, then it follows from
\eqref{Nevfullr} that for $\lambda\in(c,d)$,
\begin{equation}\label{increasing}
Q'(\lambda) = \beta + \int_{(-\infty,c]\cup[d,\infty)} \frac{d\sigma(t)}{(t-\lambda)^2}.
\end{equation}
In particular, if $Q$ is not a constant function, then it is
strictly increasing on the interval $(c,d)$; such an interval is
called a \textit{spectral gap} of $d\sigma$.

The characterizing objects in the integral representation of a
Nevanlinna function are its \textit{spectral measure}\footnote{The
spectral measures of Nevanlinna functions will always be assumed to
be normalized in the standard manner: $\sigma(t) =
(\sigma(t+)+\sigma(t-))/2$ and $\sigma(0)=0$.} $d\sigma$ and
$\beta$. The spectral measure $d\sigma$ of a Nevanlinna function $Q$
can be recovered from the integral representation by means of the
generalized Stieltjes inversion formula: if $\varphi$ is real-valued
on $[c,d]\subset \dR$ and holomorphic on an open neighborhood of
$[c,d]$ in $\dC$, then
\begin{equation}\label{genStieltjesinv}
\begin{split}
&\lim_{\epsilon \downarrow 0}
  \int_{[c,d]} \IM \left(\varphi(\lambda+i\epsilon)Q(\lambda+i\epsilon)\right)\frac{d\lambda}{\pi} =
 \int_{[c,d]} \varphi(t)\left(\mathbf{1}_{(c,d)}+\frac{\mathbf{1}_{\{c\}}+ \mathbf{1}_{\{d\}}}{2}\right)d\sigma(t).
\end{split}
\end{equation}
This formula for sign varying functions $\varphi$ is easily derived from \cite[(S1.2.6)]{KK74}. In particular, since the spectral measure of Nevanlinna functions is nonnegative, \eqref{genStieltjesinv} yields the following statement.

\begin{lemma}\label{genStieltjesinvApplied}
Let $Q_0$ and $Q_1$ be Nevanlinna functions and assume that there
exists functions $\varphi_0$ and $\varphi_1$ holomorphic on a neighborhood
$D$ of $[c,d]$ in $\dC$, which are real-valued, nonzero and have
opposite, but fixed, sign on $[c,d]$, such that
\begin{equation*}\label{genStieltjesinvApplied1}
\varphi_0(z)Q_1(z) = \varphi_1(z)Q_1(z), \quad z \in D\setminus \dR.
\end{equation*}
Then $(c,d) \subset \rho(Q_0)\cap\rho(Q_1)$.
\end{lemma}

The next lemma contains some further information about the local
behavior of Nevanlinna functions on the real line: these facts can
easily be deduced from the integral representations of Nevanlinna
functions; see e.g. \cite{Don74,KK74}.
\begin{lemma}\label{pointlimit}
Let $c \in \dR$ and let $Q \in \cN$ have the integral representation
\eqref{Nevfullr}. Then
\[
 \lim_{z \widehat{\to} c} (c-z)Q(z) = \int_\dR \mathbf{1}_{\{c\}}\, d\sigma(t)
 \quad \textrm{and} \quad \lim_{z \widehat{\to} \infty} \frac{Q(z)}{z} = \beta.
\]
In particular, both limits are finite and nonnegative. Moreover,
\[
 \lim_{z \widehat{\to} c} \frac{Q(z)}{(z-c)} =
 \left\{ \begin{array}{cl} \beta + \int_{\dR}\frac{d\sigma(t)}{(t-c)^2}<\infty,& \; \textrm{if}
 \quad \lim_{z \widehat{\to} c} Q(z)=0 \textrm{ and } \int_{\dR} \frac{d\sigma(t)}{(t-c)^2} < \infty;\\
  \infty,& \; \textrm{otherwise}; \end{array} \right.
\]
and
\[
 \lim_{z \widehat{\to} \infty} z Q(z) =
 \left\{ \begin{array}{cl} -\int_{\dR}d\sigma(t), & \; \textrm{if}
 \quad \beta=0, \textrm{ } \lim_{z \widehat{\to} \infty} Q(z)=0 \textrm{ and } \int_{\dR} d\sigma(t) < \infty;\\
  \infty,& \; \textrm{otherwise}. \end{array} \right.
\]
\end{lemma}

\begin{remark}\label{interpretationlimit}
The notation $z \widehat{\to} c$ stands for the non-tangential limit
from the upper (or lower) half-plane, if $c \in \dR$, or a sectorial
limit (with $|\arg (z)-\pi/2|\ge \alpha>0$) if $c=\infty$. In
particular, $\lim_{z\wh \to c} f(z) = \infty$ indicates that
$\lim_{z\wh \to c} |f(z)| = + \infty$. If the limit of $f(x)$ (with
$x\in\dR$) exists as an improper limit the notation $\lim_{x \to c}
f(x) = + \infty$ or $\lim_{x  \to c} f(x) = - \infty$ is used.
\end{remark}

\subsection{Generalized Nevanlinna functions}\label{secON}
A symmetric function $Q$ belongs to the class of \textit{generalized Nevanlinna functions}, denoted by $\cN_\kappa$ for $\kappa \in \dN$, if $Q$ is meromorphic on $\cmr$ and its kernel $(Q(z)-\overline{Q(w)})/(z-\bar{w})$ has $\kappa$ negative squares for $z,w \in \rho(Q)$; see
\cite{KL71,KL73}. In particular, $\cN=\cN_0$ is the class of
(ordinary) Nevanlinna functions. For a generalized Nevanlinna
function $Q\ne 0$ define $\pi_\alpha$ with $\alpha \in \dR
\cup\{\infty\}$ to be the largest nonnegative integer such that
\begin{equation}\label{GZNT}
 -\infty < \lim_{z\widehat{\to}\alpha}\frac{Q(z)}{(z-\alpha)^{2\pi_\alpha - 1}} \leq 0, \quad
0 \leq \lim_{z\widehat{\to}\infty}z^{2\pi_\infty - 1}Q(z) < \infty.
\end{equation}
Likewise, define $\kappa_\beta$ with $\beta \in \dR \cup\{\infty\}$ to
be the smallest nonnegative integer such that
\begin{equation}\label{GPNT}
-\infty < \lim_{z\widehat{\to} \beta} (z-\beta)^{2\kappa_\beta + 1}
Q(z) \leq 0, \quad
0 \leq \lim_{z\widehat{\to}\infty}\frac{Q(z)}{z^{2\kappa_\infty +1}} < \infty.
\end{equation}
If $\pi_\alpha>0$ ($\kappa_\beta>0$) then $\alpha$ ($\beta$) is said
to be a \textit{generalized zero (pole) of non-positive type} (GZNT)
(or (GPNT)) of $Q$ with multiplicity $\pi_{\alpha_i}$
($\kappa_\beta$, respectively).  Let $\{\alpha_i\}_{i=1}^m$ be the
collection of all zeros in $\dC_+$ and all generalized zeros of
non-positive type (GZNT) in $\dR$ of $Q$ with multiplicities
$\pi_{\alpha_i}$, $1\leq i\leq m$. Similarly, let
$\{\beta_i\}_{i=1}^n$ be the collection of all poles in $\dC_+$ and
generalized poles of non-positive type (GPNT) in $\dR$ of $Q$ with
multiplicities $\kappa_{\beta_i}$, $1\leq i \leq n$. The following
result characterizes the class of generalized Nevanlinna functions;
see \cite{DLLS00} (cf. also \cite{DHS99}).

\begin{proposition}\label{CharNK} $Q \in \cN_\kappa$ if and only if there exists a unique $Q_0 \in \cN$ such that
\[
 Q(z) = \frac{\prod_{i=1}^m (z-\alpha_i)^{\pi_{\alpha_i}}(z-\overline{\alpha}_i)^{\pi_{\alpha_i}}}
 {\prod_{i=1}^{n} (z-\beta_i)^{\kappa_{\beta_i}}(z-\overline{\beta}_i)^{\kappa_{\beta_i}}} \, Q_0(z),
\]
where $\kappa=\max\{\sum_{i=1}^{m}\pi_{\alpha_i},\sum_{i=1}^{n}\kappa_{\beta_i}\}$.
\end{proposition}

The unique factorization $\phi Q_0$, $Q_0 \in \cN$, of any generalized Nevanlinna function $Q$ provided by the above proposition is called the \textit{canonical factorization} of $Q$. Moreover, note that $\pi_\infty + \sum_{i=1}^{m}\pi_{\alpha_i} = \kappa_\infty + \sum_{i=1}^{n}\kappa_{\beta_i}$. In particular, a generalized Nevanlinna function is an ordinary Nevanlinna function if it does not have any GZNT's or GPNT's in $\dC\cup \{\infty\}$.

The canonical factorization allows an easy proof of the next
composition result.

\begin{lemma}\label{transqw}
Let $Q \in \cN_{\kappa}$ and let $\tau \in \cN$ be a rational
Nevanlinna function of degree $k$. Then $Q \circ \tau \in \cN_{k
\kappa}$ and, furthermore, the canonical factorizations of $Q$ and
$Q \circ \tau$ are connected by
\[
 Q=\phi Q_0,\quad Q \circ \tau =(\phi \circ \tau )(Q_0 \circ \tau).
\]
\end{lemma}
\begin{proof}
For rational functions $r_1$ and $r_2$ the degree of the composition
$r_1 \circ r_2$ is given by $\deg(r_1\circ r_2) = \deg(r_1)
\deg(r_2)$, see \cite{Wa64}. Let $Q=\phi Q_0$, $Q_0 \in \cN$, be the
canonical factorization of $Q$ in Proposition~\ref{CharNK}. Then
$\phi \circ \tau$ is a rational function of degree $k\kappa$. Since
$\tau$ and $\phi$ are symmetric, so is $\phi \circ \tau$. Moreover,
this composition is nonnegative on the real line, because $\phi$ is
nonnegative on the real line and $\tau$ maps $\dR \cup\{\infty\}$
into $\dR \cup\{\infty\}$. Since $Q_0 \circ \tau \in \cN$, the
statement follows from Proposition~\ref{CharNK}.
\end{proof}

\subsection{Symmetric and selfadjoint relations.}
Let $\sH_i$ be a Hilbert space with inner product $(\cdot,\cdot)_i$, $i=1,2$.
Then $H$ is called a \textit{(linear) relation} from $\sH_1$ to $\sH_2$ if its
graph is a subspace of $\sH_1 \times \sH_2$. In particular, $H$ is
closed if and only if its graph is closed (as a subset of $\sH_1
\times \sH_2$); in what follows $H$ is usually identified with its
graph. The symbols $\dom H$, $\ran H$, $\ker H$, and $\mul H$ stand
for the domain, range, kernel, and the multi-valued part of $A$,
respectively.

The adjoint $H^{*}$ of $H$ is defined by
\begin{equation}\label{A*}
 H^{*} = \{ \{f,f'\} \in \sH_2\times\sH_1: (f',g)_1=(f,g')_2, \;\; \forall\{g,g'\}\in H\},
\end{equation}
A relation $H$ in the Hilbert space $\sH$ (i.e., relation from $\sH$ to $\sH$) is
said to be \textit{symmetric} or \textit{selfadjoint} if $H\subset
H^{*}$ or $H = H^{*}$, respectively. For a relation $H$ in
$\sH$ the \textit{eigenspaces} of $H$ are denoted
by
\[
\sN_\lambda(H)=\ker (H-\lambda) \quad \textrm{and} \quad \widehat{\sN}_\lambda(H)=\{\{f_\lambda,\lambda f_\lambda\}: f_\lambda \in \sN_\lambda(H)\}, \quad \lambda \in \dC.
\]
Recall that for a symmetric relation $S$ in $\sH$, $n_+(S) = \dim \sN_{\overline{\lambda}}(S^*)$ and $n_-(S) = \dim \sN_{\lambda}(S^*)$, $\lambda \in \dC_+$ denote its \textit{deficiency indices}. If $S$ has equal defect numbers, then $S$ allows selfadjoint extensions. If $A$ is a selfadjoint extension of $S$, then
\begin{equation}\label{vonneu2}
S^* = A \hplus \wh \sN_\lambda(S^*), \quad \lambda \in \rho(A).
\end{equation}
Here $\hplus$ indicates the componentwise sum (linear span) of the subspaces.

The multi-valued part of a selfadjoint relation $A$ reduces the
relation:
\begin{equation}\label{oppartred}
A = A_o \oplus \{\{0\}\times \mul A\}.
\end{equation}
Here $A_o=P_\infty A\upharpoonright_{\ran P_\infty}$, where $P_\infty$ stands for the orthogonal projector onto $\cdom
A=\sH\ominus \mul A$, the so-called operator part of $A$ is a selfadjoint operator in $\cdom A$. Using the above decomposition define $|A|^\alpha$,
$\alpha > 0$, as
\begin{equation}\label{defabsval}
|A|^\alpha = |A_o|^\alpha\oplus \{\{0\}\times \mul A\},
\end{equation}
where $|A_o|$ is the modulus of $A_o$, and set $|A|^{\alpha} =
|A^{-1}|^{-\alpha}$ for $\alpha < 0$. Then $|A|^\alpha$ is a
selfadjoint relation in $\sH$.

\subsection{Realizations of Nevanlinna functions as Weyl functions.}
Boundary triplets for symmetric operators in Hilbert spaces, which were introduced in
\cite{Bruk,Kochu}, can be used to realize Nevanlinna functions. Here the definition of a boundary triplet is given for symmetric
relations with defect numbers $n_+(S)=n_-(S)=1$ together with a definition for its associated $\gamma$-field and Weyl function; see \cite{DM91,DM95}.

\begin{definition}\label{defbndscalles}
Let $S$ be a closed symmetric relation in a Hilbert space
$\{\sH,(\cdot,\cdot)\}$. Then $\{\dC, \Gamma_0, \Gamma_1\}$ is a
\emph{boundary triplet} for $S^*$ if
\begin{enumerate}
\def\labelenumi{\rm (\roman{enumi})}
\item  the mappings $\Gamma_0,\Gamma_1:\, S^*\rightarrow\dC$ satisfy the abstract Green's identity:
\[
(f',g)-(f,g') = ({\Gamma}_1\{f,f'\}, {\Gamma}_0\{g,g'\})- ({\Gamma}_0\{f,f'\},{\Gamma}_1\{g,g'\})
\]
for all $\{f,f'\},\{g,g'\}\in S^*$;
\item $\Gamma :\, S^* \to \dC\times\dC$, $\{f,f'\} \mapsto \{\Gamma_0 \{f,f'\}, \Gamma_1 \{f,f'\}\}$
is surjective.
\end{enumerate}
With $A_0 := \ker \Gamma$, the \textit{$\gamma$-field $\gamma_\lambda$ and the Weyl function $M(\lambda)$} associated with
$\{\dC, \Gamma_0, \Gamma_1\}$ are the vector function and
scalar function defined by
\begin{equation*}
\gamma_\lambda=\pi_1(\Gamma_0\!\upharpoonright\!\widehat\sN_\lambda(S^*))^{-1},\quad M(\lambda) =\Gamma_1(\Gamma_0\!\upharpoonright\!\widehat\sN_\lambda(S^*))^{-1},\quad \lambda\in\rho(A_0).
\end{equation*}
\end{definition}
Here $A_0 := \ker \Gamma_0$ and $A_1 :=\ker \Gamma_1$ are selfadjoint extension of $S$, and $\pi_1$ denotes the orthogonal projection in $\sH\oplus\sH$ onto $\sH\times \{0\}$. The $\gamma$-field and the Weyl function satisfy the formulas
\begin{equation}\label{chargamma}
\gamma_\lambda = (I + (\lambda-\mu)(A_0-\lambda)^{-1})\gamma_\mu, \quad \lambda, \mu\in \rho(A_0),
\end{equation}
and
\begin{equation}\label{charM}
M(\lambda) - M(\mu)^*  = (\lambda-\bar{\mu})\gamma_\mu^*\gamma_\lambda,\quad \lambda,\mu\in \rho(A_0);
\end{equation}
see \cite{DM91}. This shows that $\gamma$ is a holomorphic
vector-function on $\rho(A_0)$ and that $M(\lambda)$ is a Nevanlinna
function. In particular, if $M$ is not a constant function, then
$0\in\rho(\IM M(\lambda))$ for all $\lambda\in\cmr$. Nevanlinna
functions which have this additional property are called
\textit{uniformly strict}. Notice that $S$ can be recovered from any
of its selfadjoint extensions $A$ by means of the
$\gamma$-field via
\begin{equation}\label{recoversym}
S = \{ \{f,f'\} \in A: (f'-\overline{\lambda}f,\gamma_\lambda)=0 \}, \quad \lambda \in \rho(A);
\end{equation}
see \cite{HLS95}. In the terminology of boundary triplets one can
formulate the following realization result for uniformly strict
Nevanlinna functions; cf. \cite[Theorem~1.1]{DM91}.

\begin{theorem}\label{unifstreal}
Let $Q$ be a uniformly strict scalar (i.e., nonconstant) Nevanlinna
function. Then there exist a closed symmetric relation $S$ in a
Hilbert space and a boundary triplet $\{\dC,\Gamma_0,\Gamma_1\}$ for
$S^*$ such that $Q$ is the corresponding Weyl function.

Conversely, let $S$ be a closed symmetric relation in a Hilbert
space and let $\{\dC,\Gamma_0,\Gamma_1\}$ be a boundary triplet for
$S^*$, then the Weyl function associated with the boundary triplet
is a uniformly strict scalar Nevanlinna function.
\end{theorem}

The realization in Theorem~\ref{unifstreal} is unique up to unitary
equivalence under the minimality condition
\begin{equation}\label{minimal}
 \sH=\cspan\{\gamma_\lambda:\, \lambda\in \cmr \},
\end{equation}
in which case $\rho(A_0)=\rho(Q)$. If the realization is not minimal,
then $\rho(A_0)$ may be a proper subset of $\rho(Q)$; see
\eqref{chargamma}, \eqref{charM}, and e.g. \eqref{Qfd} below.

\begin{remark}\label{Qfunctions}
It is a consequence of \eqref{chargamma}, \eqref{charM}, and
Theorem~\ref{unifstreal} that for every Nevanlinna function $Q$
there exists a selfadjoint relation $A$ in a Hilbert space
$\{\sH,(\cdot,\cdot)\}$ and an element $v \in \sH$, such that
\begin{equation}\label{Qfd}
Q(\lambda) = Q(\lambda_0)^* + (\lambda-\overline{\lambda_0})((I+(\lambda-\lambda_0)(A-\lambda)^{-1})v,v),
 \quad \lambda,\lambda_0 \in \rho(A).
\end{equation}
This type of realization of Nevanlinna functions, as so-called
\textit{$Q$-functions}, was developed by M.G. Kre\u{\i}n and H.
Langer (see e.g. \cite{KL71,KL73}). In this case the $\gamma$-field
associated to the $Q$ function is defined as
\begin{equation}\label{defgamwithbt}
\gamma_\lambda = (I+(\lambda-\lambda_0)(A-\lambda)^{-1})v, \quad \lambda \in \rho(A).
\end{equation}
An application of the resolvent identity shows that $\gamma_\lambda$
satisfies \eqref{chargamma}. This discussion shows that each
Nevanlinna function $Q$ can be realized by means of a selfadjoint
relation $A$ in a Hilbert space $\{\sH,(\cdot,\cdot)\}$ and a
function $\gamma_\lambda$ which satisfies \eqref{defgamwithbt}; in
what follows, such a realization is called a
\textit{$\{A,\gamma_\lambda\}$-realization of $Q$}.
\end{remark}

For a Nevanlinna function $Q$ and $\xi \in \dR$, define the
transform
\begin{equation}\label{defQxi}
 Q_\xi(\lambda) = -Q((\lambda-\xi)^{-1}), \quad \xi \in \dR.
\end{equation}
Clearly, $Q_\xi$ is a Nevanlinna function (cf. Lemma~\ref{transqw}).
The following result gives a connection between the realizations of
$Q$ and $Q_\xi$; see \cite[Lemma~2.4]{HL06}.

\begin{lemma}\label{Lugerres}
Let $Q \in \cN$ have the representation
\[ Q(z) = Q(z_0)^* + (z - \overline{z}_0)\left(\left(I+(z-z_0)(A - z)^{-1}\right)v,v\right),\]
for $z,z_0 \in \rho(A)$, where $A$ is a selfadjoint relation in the
Hilbert space $\{\sH,(\cdot,\cdot)\}$ and $v \in \sH$. Then $Q_\xi$, see \eqref{defQxi},  has the representation
\[ Q_\xi(\lambda) = Q_\xi(\lambda_0)^* + (\lambda - \overline{\lambda}_0)\left(\left(I+(\lambda-\lambda_0)((A^{-1} + \xi) - \lambda)^{-1}\right)\wt v,\wt v\right),\]
where $z= (\lambda-\xi)^{-1}$, $z_0 = (\lambda_0 -\xi)^{-1}$, and
$\wt v = (\lambda_0 - \xi)^{-1} v$.
\end{lemma}
If the $\gamma$-fields associated to $Q$ and $Q_\xi$ are denoted by
$\gamma_\lambda$ and $\gamma_\lambda^\xi$, respectively, then it
follows from Lemma~\ref{Lugerres} that they are connected by
\[
(\lambda - \xi) \gamma^\xi_\lambda = \gamma_{(\lambda -\xi)^{-1}}.
\]
Furthermore, Lemma~\ref{Lugerres} shows that if
$\{A,\gamma_\lambda\}$ is a realization for $Q$, see
Remark~\ref{Qfunctions}, then $Q_\xi$ is realized by $\{A^{-1} +
\xi, \gamma^\xi_\lambda\}$.

\section{Nevanlinna functions having spectral gaps}\label{sec3}
Characteristic properties of Nevanlinna functions having gaps in their spectral measure are studied. This
involves use of local variants of so-called Kac-Donoghue subclasses and the limit values of $Q$ at the endpoints of the spectral gaps.

\subsection{Kac-Donoghue classes of Nevanlinna functions.}
A Nevanlinna function $Q$ is said to belong to the \textit{Kac class}
$\cN(\infty,1)$, see \cite{Ka56}, or to the \textit{Kac-Donogue class $\cN(\xi,1)$ with $\xi \in \dR$}, cf. \cite{HSSW08}, if
\[ 
 \int_{[1,\infty)} \,\frac{\IM Q(iy)}{y}\,dy < \infty \quad \textrm{or} \quad \int_{(0,1]}\frac{\IM Q(\xi+iy)}{y} dy < \infty,
\]
respectively. The classes $\cN(\xi,1)$, $\xi \in \dR$, are connected
to the Kac class $\cN(\infty,1)$ by means of transformation $Q_\xi$,
$\xi \in \dR$, in \eqref{defQxi}. In fact, the identities
\[ \begin{split}
 \int_0^1 \frac{\IM Q_\xi(\xi+iy)}{y} dy &=
 \int_0^1 \frac{-\IM  Q(\frac{1}{iy})}{y} dy =
 \int_1^\infty \frac{-\IM Q\left(-it\right)}{t} dt
  = \int_1^\infty \frac{\IM Q\left(it\right)}{t} dt, \end{split} \]
show that if $Q \in \cN$ and $Q_\xi$ is defined by \eqref{defQxi}, then
\begin{equation}\label{transsets}
Q\in \cN(\infty,1) \quad \textrm{if and only if} \quad Q_\xi \in \cN(\xi,1);
\end{equation}
cf. \cite[Proposition~3.3]{HSSW08}. It should be also noted that if $Q$ has the
integral representation \eqref{Nevfullr}, then the spectral measure $d\sigma_\xi$
of $Q_\xi$ is given by
\[ 
d\sigma_\xi (t) = -(t-\xi)^2d\sigma(1/(t-\xi)), \quad t\in \dR \setminus \{\xi\}, \quad \sigma_\xi(\xi) = \beta;
\]
see \cite[Proposition 3.2]{HSSW08}. Therefore the classes $\cN(\xi,1)$ can also be characterized via the spectral
measures; for the case $\xi=\infty$, see \cite[Theorem~S1.3.1]{KK74}, \cite[Proposition 2.2]{HLS95} and for
the case $\xi\in\dR$, see \cite[Lemma 3.5]{HSSW08}.
\begin{proposition}\label{charclassspectral}
Let $Q \in \cN$ with the integral representation \eqref{Nevfullr}.
Then $Q \in \cN(\xi,1)$ if and only if
\[ \int_{|t-\xi|<1} \frac{d\sigma(t)}{|t-\xi|} < \infty, \quad  \xi \in \dR, \quad \textrm{or} \quad
\int_{\dR} \frac{d\sigma(t)}{|t| +1} < \infty  \,\, \textrm{and} \,\, \beta=0, \quad \xi = \infty.\]
If either of the above equivalent conditions holds, then  $\lim_{z \wh \to \xi} Q(z) \in \dR$.
\end{proposition}

If $\xi$ is the endpoint of an interval contained in $\rho (Q)$,
then the class $\cN(\xi,1)$ can be characterized by means of the
limit of the function at the point $\xi$. For this the following lemma,
which will also be used in later subsections, is stated.

\begin{lemma}\label{limvalue}
Let $Q \in \cN$ have the integral representation \eqref{Nevfullr}.
Then the following statements hold:
\begin{enumerate}
\def\labelenumi {\rm (\roman{enumi})}
\item if $(-\infty,c)\subset \rho(Q)$, then with $x\in(-\infty,c)$ and $\Delta_c = [c,\infty)$
\[
\lim_{x \to -\infty} Q(x) = \left\{ \begin{array}{cl} \alpha - \int_{\Delta_c}\frac{td\sigma(t)}{1+t^2} \in \dR,& \; \textrm{if} \quad \int_{\Delta_c} \frac{d\sigma(t)}{1+|t|} < \infty \quad \textrm{and}\quad \beta=0  ;\\ -\infty,& \; \textrm{otherwise}; \end{array} \right. \]
\item if $(c_-,c_+) \subset \rho(Q)$, then with $x\in (c_-,c_+)$ and $\Delta_{c} = \dR\setminus (c_-,c_+)$
\[\lim_{x\to c_\pm} \; Q(x) = \left\{ \begin{array}{cl} \alpha + \beta c_\pm + \int_{\Delta_{c}}\frac{(1+tc_\pm)d\sigma(t)}{(t-c_\pm)(1+t^2)} \in \dR,& \; \textrm{if} \quad \int_{|t-c_\pm|\leq1}\frac{d\sigma(t)}{|t-c_\pm|} < \infty;\\ \pm\infty,& \; \textrm{otherwise}. \end{array} \right. \]
\end{enumerate}
\end{lemma}
\begin{proof}
By monotonicity, see \eqref{increasing}, the limit in (i) exists in
$\dR \cup \{-\infty\}$. Furthermore, it follows from
Lemma~\ref{pointlimit} that for the limit in (i) to be finite it is
necessary that $\beta =0$.  If $\beta=0$, then the monotone
convergence theorem implies that
\[ \lim_{x \to -\infty} Q(x) = \alpha - \int_{[c,\infty)}\frac{td\sigma(t)}{1+t^2} \]
for $x\in (-\infty,c)$.  This limit is finite if and only if
$d\sigma$ satisfies the integrability condition in (i). The other
statement can be proven with similar arguments.
\end{proof}
Combining Lemma~\ref{limvalue} with Proposition~\ref{charclassspectral} yields the following result.

\begin{corollary}\label{limitaschar}
Let $Q \in \cN$, let $\xi \in \dR \cup \{\infty\}$ and assume that
there exists $c \in \dR$ such that $(\xi,c)$ or $(c,\xi)$ belongs to
$\rho(Q)$, if $\xi \in \dR$, or that $(-\infty,c)$ or $(c,\infty)$ belongs to $\rho(Q)$ if $\xi =\infty$.
Then
\[ Q \in \cN(\xi,1) \quad \textrm{if and only if} \quad \lim_{z \wh \to \xi} Q(z) \in \dR.\]
\end{corollary}

\subsection{Nevanlinna functions holomorphic on (the complement of) a compact interval}\label{sec23}
Nevanlinna functions with a gap in their spectrum are
characterized by means of their limits at the endpoints of this gap.
Two cases are considered: the case where $\rho(Q)$
contains a bounded interval (finite spectral gap) and the case where
$\rho(Q)$ contains the complement of a compact interval.

\begin{proposition}\label{CharQGap1}
Let $Q \in \cN$ have the integral representation \eqref{Nevfullr}.
Then the following statements are equivalent for $c,d\in\dR$ with
$c<d$:
\begin{enumerate}
\def\labelenumi {\rm (\roman{enumi})}
\item $(c,d) \subset \rho(Q)$ and $\lim_{x \uparrow d} Q(x) \in \dR$;
\item the integral representation of $Q$ is given by
\[
Q(z) = \eta + \frac{z-d}{z-c}\left(\beta z - \gamma + \int_{\dR\setminus (c,d]}\left(\frac{1}{t-z}-\frac{t}{1+t^2}\right)d\hat{\sigma}(t)\right),
\]
where $\gamma, \eta \in \dR$,
$d\hat{\sigma}(t)=\frac{(t-c)d\sigma(t)}{(t-d)}$, and
$\int_{\dR\setminus (c,d]}\frac{d\hat{\sigma}(t)}{1+t^2}<
\infty$;
\item $\exists \widetilde{Q} \in \cN$ and $\tilde{\eta} \in \dR$ such that $Q(z)=\tilde{\eta} + \frac{z-d}{z-c}\widetilde{Q}(z)$.
\end{enumerate}
In particular, in (ii) $\eta = \lim_{x \uparrow d} Q(x)$ and in
(iii) $\tilde{\eta}\ge \eta$.
\end{proposition}

\begin{proof}
(i) $\Rightarrow$ (ii) If $(c,d) \subset \rho(Q)$, then $Q$ has the
integral representation \eqref{Nevfullr} with $\Delta=\dR\setminus
(c,d)$. Moreover, by Lemma~\ref{limvalue}~(ii) the assumption in
(i) implies that $\int_{|t-d|\leq 1} |t-d|^{-1} d\sigma(t) <
\infty.$ Therefore $Q$ can be rewritten as
\begin{equation}\label{rlrlrlr}
\begin{split}
Q(z) &= \eta + Q(z)- \eta =
 \eta + (z-d)\left( \beta + \int_{\dR\setminus (c,d]}\frac{d\sigma(t)}{(t-z)(t-d)}\right),
\end{split}
\end{equation}
where $\eta := \lim_{x \uparrow d} Q(x) \in \dR$, see
Lemma~\ref{limvalue}. Observe that
\begin{equation}\label{klklkl}
\begin{split} &(z-c)\int_{\dR\setminus(c,d]} \frac{d\sigma(t)}{(t-z)(t-d)} = \int_{\dR\setminus (c,d]} \frac{[(z-t)+ (t-c)]d\sigma(t)}{(t-z)(t-d)} \\ =& \int_{\dR\setminus (c,d]} \left(\frac{1}{t-z}-\frac{t}{1+t^2}\right)\frac{(t-c)d\sigma(t)}{(t-d)}-\int_{\dR\setminus (c,d]} \frac{(1+tc)d\sigma(t)}{(1+t^2)(t-d)}, \end{split}
\end{equation}
where the integral terms converge (for $z \in \rho(Q)$) as a
consequence of the stated integrability condition that $d\sigma$
satisfies. Substituting \eqref{klklkl} into \eqref{rlrlrlr} yields
the integral representation in (ii).

(ii) $\Rightarrow$ (iii) This is evident.

(iii) $\Rightarrow$ (i) This follows from
Lemma~\ref{genStieltjesinvApplied} and Lemma~\ref{pointlimit}.

As to the last statement, observe that the measure $d\hat\sigma(t)$
in (ii) does not involve a point mass at $x=d$, in which case $\eta
= \lim_{x \uparrow d} Q(x)$ by Lemma~\ref{pointlimit}. On the other
hand, if $Q$ is as in (iii) then by applying Lemma~\ref{pointlimit}
to the function $\wt Q$ one obtains $\lim_{x \uparrow d}
Q(x)=\tilde{\eta}+\lim_{z \wh\to d} (z-d)\wt Q(z)/(d-c) \le
\tilde\eta$.
\end{proof}

If (iii) in Proposition~\ref{CharQGap1} holds for some
$\tilde{\eta}\ge \eta$, then it holds also for $\eta = \lim_{x
\uparrow d} Q(x)$. Moreover, if (iii) holds for some
$\tilde{\eta}\ge \eta$, then it holds for every $\tilde{\eta}\ge
\eta$.

Next the case that $\rho(Q)\cap\dR$ contains the complement of a
compact interval is characterized. The proof for this statement is
similar to the proof of Proposition~\ref{CharQGap1} and is therefore
omitted.

\begin{proposition}\label{CharQBound1}
Let $Q \in \cN$ have the integral representation \eqref{Nevfullr}.
Then the following statements are equivalent for $c,d\in\dR$ with
$c<d$:
\begin{enumerate}
\def\labelenumi {\rm (\roman{enumi})}
\item $(-\infty,c)\cup(d,\infty) \subset {\rho}(Q)$, $\beta=0$, and $\lim_{x \downarrow d} Q(x)  \in \dR$;
\item the integral representation of $Q$ is given by
\[ Q(z) = \eta + \frac{d-z}{z-c}\left( -\gamma+\int_{[c,d)}\frac{d\hat{\sigma}(t)}{t-z} \right), \]
where $\gamma, \eta\in \dR$,
$d\hat{\sigma}(t)=\frac{(t-c)d\sigma(t)}{(d-t)}$, and
$\int_{[c,d)}d\hat{\sigma}(t)< \infty$;
\item $\exists \widetilde{Q} \in \cN$ and $\tilde{\eta} \in \dR$ such that
$Q(z) =\tilde{\eta} + \frac{d-z}{z-c}\widetilde{Q}(z)$.
\end{enumerate}
In particular, in (ii) $\eta = \lim_{x \downarrow d} Q(x)$ and in
(iii) $\tilde{\eta}\le \eta$.
\end{proposition}
Note that Proposition~\ref{CharQBound1} can be used to obtain a characterization of the class
$\cS[a,b]$ introduced in \cite{KNapp}. The particular integral representation \cite[Theorem~A.7]{KNapp} can obtained directly from
the integral representation in Proposition~\ref{CharQBound1} (cf. \eqref{rlrlrlr}):
\[
 Q(z) = (d-z)\left(\frac{\eta}{d-z}+\int_{[c,d)}\frac{d\sigma(t)}{(t-z)(d-t)}\right)
 =(d-z)\int_{[c,d]}\frac{d\tilde{\sigma}(t)}{t-z}.
\]
Propositions~\ref{CharQGap1} and \ref{CharQBound1} imply the
following useful result.

\begin{corollary}\label{Npres3456}
Let $Q \in \cN$ have the integral representation \eqref{Nevfullr}
and let $c,d\in\dR$ with $c<d$. Then
\begin{enumerate}
\def\labelenumi {\rm (\roman{enumi})}
\item $\frac{z-c}{z-d}Q(z) \in \cN$ if and only if $(c,d) \subset \rho(Q)$ and $-\infty<\lim_{x \uparrow d} Q(x) \leq 0$;
\item $\frac{z-d}{z-c}Q(z) \in \cN$ if and only if $(c,d) \subset \rho(Q)$ and $0\leq \lim_{x \downarrow c} Q(x) <\infty$;
\item $\frac{z-c}{d-z}Q(z) \in \cN$ if and only if $(-\infty,c)\cup(d,\infty)\subset \rho(Q)$, $\beta=0$, and $0 \leq \lim_{x \downarrow d} Q(x) < \infty$;
\item $\frac{z-d}{c-z}Q(z) \in \cN$ if and only if $(-\infty,c)\cup(d,\infty)\subset \rho(Q)$, $\beta=0$, and $-\infty<\lim_{x \uparrow c} Q(x) \leq 0$.
\end{enumerate}
\end{corollary}
\begin{proof}
(i) Assume that $(c,d) \subset \rho(Q)$ and $-\infty<\lim_{x
\uparrow d} Q(x) \leq 0$. Then it follows from
Proposition~\ref{CharQGap1} that
\[
 \frac{z-c}{z-d}Q(z)=\frac{z-c}{z-d} \eta + \widetilde{Q}(z),
\]
where $\wt Q \in \cN$ and $\eta = \lim_{x \uparrow d} Q(x)$. Since
$\frac{z-c}{z-d} \eta \in \cN$ if $\eta \le 0$, one concludes that
$\frac{z-c}{z-d}Q(z) \in \cN$. Conversely, if $\frac{z-c}{z-d}Q(z)
\in \cN$, then item (iii) in Proposition~\ref{CharQGap1} holds with
$\wt \eta =0$. Hence item (i) together with the last statement in
Proposition~\ref{CharQGap1} shows that $(c,d) \subset \rho(Q)$ and
$\eta=\lim_{x \uparrow d} Q(x) \in \dR$ with $\eta\le \wt \eta =0$.

The equivalence in (ii) is obtained easily from the equivalence in
(i) by passing to the inverses. Finally, (iii) and (iv) follow from
Proposition~\ref{CharQBound1} in the same manner as (i) and (ii)
follow from Proposition~\ref{CharQGap1}.
\end{proof}

\begin{remark}\label{HolCont}
If $Q_1(z):=\frac{z-c}{z-d}Q(z)\in\cN$, then by
Lemma~\ref{pointlimit} $\lim_{x \downarrow c} Q_1(x)>0$ if and
only if $c\in\sigma_p(Q)$. Furthermore, $Q_1$ admits a holomorphic
continuation to the point $c\in\dR$ if and only if either $c\in
\rho(Q)$, or $c\in\sigma_p(Q)$ is a separated pole of $Q$, i.e.,
$\sigma(Q)\cap(c-\varepsilon,c+\varepsilon)=\{c\}$ for some
$\varepsilon>0$.

Similarly, by Lemma~\ref{pointlimit} $\lim_{x \uparrow  d}
Q_1(x)<\infty$ if and only if $\lim_{x \uparrow d} Q(x)=0$ and
$\int_{\dR} d\sigma(t)/(t-d)^2 < \infty$. Moreover, $Q_1$
admits a holomorphic continuation to the point $d\in\dR$ if and only
if $d\in \rho(Q)$ and $\lim_{x \uparrow  d} Q(x)=0$.
\end{remark}

\subsection{Nevanlinna functions holomorphic on a semibounded interval}\label{sec24}
In this section Nevanlinna functions having a semibounded interval
as their spectral gap are investigated. The statements below are
formulated explicitly only in case that an interval of the form
$(-\infty,c)$ belongs to $\rho(Q)$. Using the equivalence $Q(z) \in
\cN$ if and only if $-Q(-z) \in \cN$, these results are easily
modified to the case $(c,\infty)\subset\rho(Q)$.

\begin{proposition}\label{CharQflr1}
Let $Q\in \cN$ have the integral representation \eqref{Nevfullr}.
Then with $c\in\dR$ the following statements are equivalent:
\begin{enumerate}
\def\labelenumi {\rm (\roman{enumi})}
\item $(-\infty,c) \subset \rho(Q)$ and $\lim_{x \downarrow -\infty} \left(Q(x) - \beta x\right) \in \dR $;
\item the integral representation of $Q$ is given by
    \[ Q(z) = \eta + \beta z + \int_{[c,\infty)} \frac{d\sigma(t)}{t-z}, \]
    where $\eta \in \dR$ and $\int_{[c,\infty)}\frac{d\sigma(t)}{1+|t|}<\infty$;
\item $\exists \widetilde{Q} \in \cN$ and $\tilde{\eta} \in \dR$ such that $Q(z) =  \tilde{\eta} + \beta z + \frac{\widetilde{Q}(z)}{z-c}$.
\end{enumerate}
In particular, in (ii) $\eta = \lim_{x \downarrow -\infty}
\left(Q(x) - \beta x\right)$ and in (iii) $\tilde{\eta}\le \eta$.
\end{proposition}

\begin{proof}
(i) $\Rightarrow$ (ii) If $(-\infty,c) \subset \rho(Q)$, then $Q$
has the representation \eqref{Nevfullr} with $\Delta=[c,\infty)$. If
the limit in (i) is finite, then Lemma~\ref{limvalue} implies that
$\int_{[c,\infty)}\frac{d\sigma(t)}{1+|t|}<\infty$. Therefore $Q$
has the representation given in (ii).

(ii) $\Rightarrow$ (iii) This implication follows from
\[ \begin{split}
&(z-c)\int_{[c,\infty)} \frac{d\sigma(t)}{t-z} = \int_{[c,\infty)} \frac{(z-t + t-c)d\sigma(t)}{t-z} \\
= & \int_{[c,\infty)} \left(\frac{1}{t-z} - \frac{t}{1+t^2}\right)(t-c)d\sigma(t) - \int_{[c,\infty)} \left(1-\frac{t(t-c)}{1+t^2}\right)d\sigma(t)
\end{split} \]
where the integrals converge as consequence of the assumptions; in
particular one has $\int_{[c,\infty)}(t-c)d\sigma(t)/(1+t^2)<
\infty$.

(iii) $\Rightarrow$ (i) If (iii) holds, then $(-\infty,c) \subset
\rho(Q)$ by Lemma~\ref{genStieltjesinvApplied} and by
Lemma~\ref{pointlimit}
\[
 \lim_{x \downarrow -\infty} \left(Q(x) - \beta x\right)
 = \lim_{x \downarrow -\infty} \left( \tilde{\eta}+\frac{\widetilde{Q}(x)}{x-c} \right)  \in \dR.
\]

The last statement is clear by monotone convergence on
$(-\infty,c)$, see \eqref{increasing}.
\end{proof}

Next consider the case that the limit of $Q$ at the finite endpoint
of the semibounded spectral gap is finite.

\begin{proposition}\label{CharQflr2}
Let $Q \in \cN$ have the integral representation \eqref{Nevfullr}.
Then with $c\in\dR$ the following statements are equivalent:
\begin{enumerate}
\def\labelenumi {\rm (\roman{enumi})}
\item $(-\infty,c) \subset \rho(Q)$ and $\lim_{x \uparrow c} Q(x) \in \dR$;
\item the integral representation of $Q$ is given by
    \[ Q(z) = \eta + (z-c)\left(\beta + \int_{(c,\infty)} \frac{d\hat{\sigma}(t)}{t-z}\right), \]
    where $\eta \in \dR$, $d\hat{\sigma}(t)=\frac{d\sigma(t)}{t-c}$, and
 $\int_{(c,\infty)}\frac{d\hat{\sigma}(t)}{1+|t|}<\infty$;
\item $\exists \widetilde{Q} \in \cN$ and $\tilde{\eta} \in \dR$ such that $Q(z) = \tilde{\eta} + (z-c)\widetilde{Q}(z)$.
\end{enumerate}
In particular, in (ii) $\eta = \lim_{x \uparrow c} Q(x)$ and in
(iii) $\tilde{\eta}\ge \eta$.
\end{proposition}

\begin{proof}
(i) $\Rightarrow$ (ii) For this implication see the proof of the
implication form (i) to (ii) in Proposition~\ref{CharQGap1},
especially \eqref{rlrlrlr}.

(ii) $\Rightarrow$ (iii) This implication is evident.

(iii) $\Rightarrow$ (i) If (iii) holds, then by
Lemma~\ref{genStieltjesinvApplied} $(-\infty,c) \subset \rho(Q)$ and
by Lemma~\ref{pointlimit} applied to the spectral measure
$d\tilde{\sigma}$ of $\widetilde{Q}$, one obtains
\[
 \lim_{x \uparrow c} Q(x) = \tilde{\eta} + \lim_{x \uparrow c} (x-c)\widetilde{Q}(x)
 = \tilde{\eta}-\int_{[c,\infty)} \mathbf{1}_{\{c\}} \, d\tilde{\sigma}(t) \in \dR.
\]

Finally, Lemma~\ref{pointlimit} applied to the integral representation in part (ii)
gives the stated limit expression for $\eta$.
\end{proof}

\begin{remark}
Recall that the class of Stieltjes (inverse Stieltjes) functions
consists of those Nevanlinna functions which are holomorphic and
nonnegative (nonpositive) on $\dR_-$. These functions were
introduced and characterized by Kre\u{\i}n, see \cite{KK74} and the
references therein. In fact, since Nevanlinna functions $Q$ are
nondecreasing on $\rho(Q)\cap \dR$, it follows that the class of
Stieltjes (inverse Stieltjes) functions consists of those Nevanlinna
functions which are holomorphic on $\dR_-$ and satisfy $0 \leq
\lim_{x \downarrow -\infty} Q(x)< \infty$ ($-\infty < \lim_{x
\uparrow c} Q(x) \leq 0$). Therefore, Proposition~\ref{CharQflr1}
and Proposition~\ref{CharQflr2} with $c=0$ and $\eta \geq 0$ or
$\eta \leq 0$, respectively, contain, in particular, a
characterization for the classes of Stieltjes and inverse Stieltjes
functions.
\end{remark}

The previous two propositions give rise an analogue of Corollary~\ref{Npres3456}.

\begin{corollary}\label{Npres12}
Let $Q\in \cN$ have the integral representation \eqref{Nevfullr}.
Then with $c\in\dR$ the following statements are equivalent:
\begin{enumerate}
\def\labelenumi {\rm (\roman{enumi})}
\item $(z-c)Q(z) \in \cN$ if and only if $(-\infty,c) \subset \rho(Q)$ and $0\le \lim_{x \downarrow -\infty} Q(x) <\infty$;
\item $\frac{1}{z-c}Q(z) \in \cN$ if and only if $(-\infty,c) \subset \rho(Q)$ and $-\infty < \lim_{x \uparrow c} Q(x) \leq 0 $.
\end{enumerate}
\end{corollary}

\section{Classes of generalized Nevanlinna functions associated with symmetric rational functions}\label{sec4}
The characterizations of Nevanlinna functions holomorphic on some
interval of the real line in Section~\ref{sec23} and \ref{sec24}
show that the product of these Nevanlinna functions with certain
(simple) symmetric rational functions are again Nevanlinna
functions. In this section results of that type are generalized to
the class of generalized Nevanlinna functions. For this purpose the
following definition is introduced.

\begin{definition}
Let $r$ be a nonconstant symmetric rational function ($\deg r>0$),
and let $\kappa$ and $\tilde\kappa$ be nonnegative integers. Then
the class $\cN_\kappa^{\tilde \kappa}(r)$ of generalized Nevanlinna
functions is defined by the formula
\[ 
 \cN_\kappa^{\tilde \kappa}(r) = \{\, Q \in \cN_\kappa: rQ \in \cN_{\tilde{\kappa}} \,\}.
\]
\end{definition}

If a generalized Nevanlinna function $Q$ belongs to
$\cN_\kappa^{\tilde \kappa}(r)$ for some symmetric rational function
$r$, whose range meets the negative real line, then $Q$ necessarily
has gaps in its spectrum.

\begin{lemma}\label{rcreatesgaps}
Let $Q \in \cN_\kappa^{\tilde \kappa}(r)$ for a symmetric rational
function $r$ and let $\phi_0 Q_0$, $Q_0 \in \cN$, be the canonical
factorization of $Q$. Then $0\leq r \leq \infty$ on $\sigma(Q_0)$
except for finitely many poles of $Q_0$.
\end{lemma}

\begin{proof}
Let $\phi_1 Q_1$, $Q_1 \in \cN$, be the canonical
factorization of $rQ$. Then $r \phi_0
Q_0=\phi_1 Q_1$. The rational factors $\phi_0$
and $\phi_1$ admit only finitely many zeros and poles, and
between them $\phi_0(x),\phi_1(x)>0$. Hence one concludes
from Lemma~\ref{genStieltjesinvApplied} that there can exist only
finitely many separated poles (hence also zeros) of $Q_0$ and
$Q_1$ on each open interval where $r <0$.
\end{proof}

The main results in this section give characterizations for
functions belonging to the classes $\cN_\kappa^{\tilde \kappa}(r)$
and also identify their canonical factorizations. In particular, in
Theorem~\ref{ResTot} an inverse statement to
Lemma~\ref{rcreatesgaps} will be proved. The approach is
constructive and, in particular, the connection between the
canonical factorizations of the functions $Q$ and $rQ$, $Q \in
\cN_\kappa^{\wt \kappa}(r)$, is made explicit. From these results
one immediately gets, for instance, factorized integral
representations for the functions $Q$ and $rQ$, from which also
Kre\u{\i}n-Langer type integral representations for functions in
$\cN_\kappa^{\tilde \kappa}(r)$ can be obtained along the lines of
\cite[Corollary 3.5]{DHS99}.

\subsection{Factorization of symmetric rational functions}
An arbitrary symmetric rational function $r$ is a generalized
Nevanlinna function. In fact, the Nevanlinna kernel for $r$ can
be expressed explicitly by means of a Bezoutian; for details see
e.g. \cite{DH03}. The aim of this subsection is to derive in simple
terms the canonical factorization in Proposition~\ref{CharNK} for
symmetric rational functions. This factorization is used to simplify
later considerations.

If $r$ is a symmetric rational function, then it admits a factorization of the form
\begin{equation}\label{symratsimrat}
r(z) = \gamma \frac{\prod_{i=1}^{n_1} (z-\alpha_i)^{\mu_i}(z-\overline{\alpha}_i)^{\mu_i}}{\prod_{i=1}^{n_2} (z-\beta_i)^{\nu_i}(z-\overline{\beta}_i)^{\nu_i}}\frac{\prod_{i=1}^{m_1} (z-a_i)}{\prod_{i=1}^{m_2} (z-b_i)} ,
\end{equation}
where $\alpha_i,\beta_i \in \dC$, $m_1,m_2,n_1,n_2,\mu_i,\nu_i \in
\dN$ and $a_i, b_i, \gamma \in \dR$. The factorization
\eqref{symratsimrat} is unique if no cancelations occur; this is
assumed in the rest of this subsection. The canonical factorization
for $r$ can be obtained from \eqref{symratsimrat} by describing the
canonical factorization for the \textit{simple symmetric rational
function}
\begin{equation}\label{simrat}
s(z)=\gamma \frac{\prod_{i=1}^{m_1}(z-a_i)}{\prod_{i=1}^{m_2}(z-b_i)}, \quad a_i, b_i, \gamma \in \dR;
\end{equation}
here a rational function is called simple if all its zeros and poles
are real and of order one. Clearly, only a zero or pole  of $s(z)$
can be its GZNT or GPNT, respectively. Next observe that for any
zero $a \in \dR$ of $s(z)$ the multiplicity $\pi_a(s)$ is given by
\begin{equation}\label{not1}
\pi_a(s) = \left\{ \begin{array}{rl} 0,& \textrm{if \;\;} 0 < \lim_{z\widehat{\to} a}\{\frac{s(z)}{z-a}\}\leq \infty
   \quad \Leftrightarrow \quad s^\prime(a)>0; \\
 1,&\textrm{if \;\;} -\infty < \lim_{z\widehat{\to} a}\{\frac{s(z)}{z-a}\}\leq 0 \quad \Leftrightarrow \quad s^\prime(a)<0, \end{array} \right.
\end{equation}
and, similarly, for any pole $b \in \dR$ of $s(z)$ the multiplicity $\kappa_b(s)$ is given by
\begin{equation}\label{not2}
 \kappa_b(s) = \left\{ \begin{array}{rl} 0,& \textrm{if \;\;} -\infty < \lim_{z\widehat{\to} b}\{(z-b)s(z)\}\leq 0
  \quad \Leftrightarrow \quad (-1/s)^\prime(b)>0;\\
 1,& \textrm{if \;\;} 0 < \lim_{z\widehat{\to} b}\{(z-b)s(z)\}\leq \infty
 \quad \Leftrightarrow \quad (-1/s)^\prime(b)<0, \end{array} \right.
\end{equation}
see Remark~\ref{interpretationlimit}. Note that, since no
cancelations occur in \eqref{simrat}, the limits
\begin{equation}\label{not03a}
\lim_{z\widehat{\to} a} \frac{s(z)}{z-a} \quad \textrm{and} \quad \lim_{z\widehat{\to} b}(z-b)s(z)
\end{equation}
are actually finite and nonzero for every zero $a \in \dR$ and every
pole $b \in \dR$ of $s$. The multiplicities $\pi_{a_i}$ and
$\kappa_{b_i}$ only depend on the sign of $\gamma$ and the location
of the poles and zeros of $s$ in \eqref{simrat}. Therefore,
associate with $s$ and $c \in \dR$ the integer $\eta_c(s)$ by
\begin{equation}\label{not03}
 \eta_c(s) = \{\,\textrm{number of poles and zeros of $s$ greater than $c$}\,\}.
\end{equation}
Then the canonical factorization of $s$ can be stated in an explicit
form as follows.

\begin{lemma}\label{rrfNev0}
Let $s$ be a simple symmetric rational function of the form
\eqref{simrat} with $\gamma\neq 0$ and $a_i\neq b_j$, and let
$\eta_c(s)$ be defined by \eqref{not03}. Then the canonical
factorization for $s \in \cN_\kappa$, $\kappa \in \dN$, is given by
$s=\psi s_0$, where $s_0 \in \cN$ and $\psi$ are given by
\[
 s_0(z) = \gamma \frac{\prod_{i=1}^{m_1} (z-a_i)^{1-2\pi_{a_i}(s)}}{\prod_{i=1}^{m_2} (z-b_i)^{1-2\kappa_{b_i}(s)}}
\quad \textrm{and} \quad \psi(z) =  \frac{\prod_{i=1}^{m_1}
(z-a_i)^{2\pi_{a_i}(s)}}{\prod_{i=1}^{m_2}
(z-b_i)^{2\kappa_{b_i}(s)}}.
\]
Here $\pi_{a_i}(s)$, $\kappa_{b_i}(s)$, and
$\kappa=\max\{\sum_{i=1}^{m_1}\pi_{a_i}(s),\sum_{i=1}^{m_2}\kappa_{b_i}(s)\}$
are determined by
\begin{equation}\label{not4}
 \pi_{a_i}(s) = \frac{1 - (-1)^{\eta_{a_i}(s)}\sgn(\gamma)}{2}\quad \textrm{and} \quad
 \kappa_{b_i}(s) = \frac{1 + (-1)^{\eta_{b_i}(s)}\sgn(\gamma)}{2}.
\end{equation}
\end{lemma}

\begin{proof}
Since $a_i\neq b_j$, the limits in \eqref{not03a} are finite and the
signs of these limits are given by
$(-1)^{\eta_{a_i}(s)}\sgn(\gamma)$ and
$(-1)^{\eta_{b_i}(s)}\sgn(\gamma)$, respectively. This gives the
expressions \eqref{not4} for $\pi_{a_i}(s)$ and $\kappa_{b_i}(s)$ in
\eqref{not1} and \eqref{not2}, respectively. If $\pi_{a_i}(s)=0$ and
$\kappa_{b_i}(s)=0$, then the corresponding factors $(z-a_i)$ and
$1/(z-b_i)$ are included in $s_0$. If $\pi_{a_i}(s)=1$
($\kappa_{b_i}(s)=1$), then $(z-a_i)^2$ (resp. $1/(z-b_i)^2$) is
included as a factor in $\psi$, while $1/(z-a_i)$ (resp. $(z-b_i)$)
becomes a factor for $s_0$. This gives the stated formulas for $s_0$
and $\psi$. Observe, that the new terms do not change $s$ and the
limits in \eqref{not03a}. In terms of $s_0$ this means that
\[ 
\lim_{z \widehat{\to} c}\frac{s_0(z)}{z-c} > 0, \quad
\lim_{z\widehat{\to} d} (z-d)s_0(z) < 0
\]
for every finite zero $c$ and every finite pole $d$ of $s_0$.
Clearly, the above conditions alone imply that $s_0\in
\cN$.
\end{proof}

The proof of Lemma~\ref{rrfNev0} can be used to obtain an
independent and constructive proof for the existence of the
canonical factorization for simple symmetric rational functions $s$
of the form \eqref{simrat}. Combining the factorization of $s$ in
Lemma~\ref{rrfNev0} with the factorization of $r$ in
\eqref{symratsimrat} the existence of the canonical factorization
for an arbitrary symmetric rational function $r$ is obtained. To
express that factorization along the lines of Lemma~\ref{rrfNev0},
define $\eta_c(r)$ for $c \in \dR$ by
\begin{equation}\label{not03c}
 \eta_c(r) = \{\,\textrm{number of poles and zeros of odd order of $r$ greater than $c$}\,\}.
\end{equation}

\begin{proposition}\label{rrfNev}
Let $r$ be a symmetric rational function with the (unique)
factorization $r=\phi s$ as in \eqref{symratsimrat} and let $s=\psi
s_0$, $s_0 \in \cN$, be the canonical factorization of $s$ given by
Lemma~\ref{rrfNev0}. Then $r \in \cN_\kappa$, $\kappa \in \dN$, and
the canonical factorization of $r$ is given by $r=\phi_0
r_0$, $r_0 \in \cN$, where
\[ \phi_0(z)=\phi(z)\psi(z)
 \quad \textrm{and} \quad
 r_0(z)=s_0(z).
\]
In particular, if $\mu(a)$ ($\nu(b)$) denotes the order of $a$ ($b$)
as a zero (pole) of $r$, then for every nonreal zero $\alpha_i$ and
nonreal pole $\beta_i$ of $r$ one has
$\pi_{\alpha_i}(r)=\mu(\alpha_i)$ and
$\kappa_{\beta_i}(r)=\nu(\beta_i)$. For every real zero $c_i$ and
pole $d_i$ of $r$ of even order one has $\pi_{c_i}(r)=\mu(c_i)/2$
and $\kappa_{d_i}(r)=\nu(d_i)/2$, while for real zeros $a_i$ and
poles $b_i$ of $r$ of odd order one has
\begin{equation}\label{not5}
 \pi_{a_i}(r) = \frac{\mu(a_i) - (-1)^{\eta_{a_i}(r)}\sgn(\gamma)}{2}
 \quad \textrm{and} \quad
 \kappa_{b_i}(r) = \frac{\nu(b_i) + (-1)^{\eta_{b_i}(r)}\sgn(\gamma)}{2}.
\end{equation}
\end{proposition}
\begin{proof}
The factorization for $r$ is obtained by combining
\eqref{symratsimrat} with Lemma~\ref{rrfNev0}. The formulas for
$\pi_{a_i}(r)$ and $\kappa_{b_i}(r)$ and obtained from
\eqref{symratsimrat} and \eqref{not4}, since clearly
$(-1)^{\eta_c(r)}=(-1)^{\eta_{c}(s)}$ if $c$ is a zero or pole of
$r$ of odd order.
\end{proof}

In view of \eqref{not5} and Lemma~\ref{rrfNev0} the poles and zeros
of $r_0=s_0$ originate only from odd order real zeros or odd order
real poles of $r$. Moreover, it is easy to decide from the canonical
factorization $s=\psi s_0$ in Lemma~\ref{rrfNev0} that the poles and
zeros of $s_0$ have the interlacing property, since in view of
\eqref{not03} the signs of the limits in \eqref{not03a} are
alternating for consecutive poles and zeros of $s$.

\subsection{The case of simple rational functions of degree one}
To describe the class $\cN_{\kappa}^{\wt \kappa}(r)$ for simple
symmetric rational functions $r$ of degree one, first some
observations are made on the product of a general symmetric rational
function with a Nevanlinna function.

Let $r$ be an arbitrary symmetric rational function and let $Q_0 \in
\cN$, $Q_0\ne 0$. Then it is clear from Sections~\ref{sec23}
and~\ref{sec24} that the product $rQ_0$ is not in general a
Nevanlinna function; the product $rQ_0$ need not
even belong to the class of generalized Nevanlinna functions.
However, if $rQ_0\in\cN_\wt\kappa$ for some $\wt\kappa\in \dN$, then
the multiplicities of generalized zeros and poles of nonpositive
type of $rQ_0$ are well defined and given by
\eqref{GZNT} and \eqref{GPNT}, respectively. In the case of a symmetric rational function $r$ it follows
from Lemma~\ref{pointlimit} that for any simple zero $a \in \dR$ of $r$ one
has
\begin{equation}\label{not1Q}
\pi_a(rQ_0) = \left\{
 \begin{array}{rl}
 0,& \textrm{if \;\;} 0 < \lim_{z\widehat{\to} a}\{\frac{r(z)Q_0(z)}{z-a}\}\leq \infty;\\
 1,& \textrm{if \;\;} -\infty < \lim_{z\widehat{\to}a}\{\frac{r(z)Q_0(z)}{z-a}\}\leq 0,
\end{array} \right.
\end{equation}
and similarly for any simple pole $b \in \dR$ of $r$ one has
\begin{equation}\label{not2Q}
\kappa_b(rQ_0) = \left\{ \begin{array}{rl}
 0,& \textrm{if \;\;} -\infty < \lim_{z\widehat{\to} b}\{(z-b)r(z)Q_0(z)\}\leq 0;\\
 1,&  \textrm{if \;\;} 0 < \lim_{z\widehat{\to} b}\{(z-b)r(z)Q_0(z)\}\leq \infty.
 \end{array} \right.
\end{equation}
In fact, if $r$ is a symmetric rational function and
$rQ_0\in\cN_{\wt\kappa}$ then \eqref{not1Q} and \eqref{not2Q} imply
that for every zero $a$ and pole $b$ of $r$ of order one the limits
$\lim_{z \wh\to a} Q_0(z)$ and $\lim_{z \wh\to b}Q_0(z)$ exist in
$\dR\cup\{\pm\infty\}$; see Section~\ref{secON}. Furthermore, in
this case $r$ changes its sign at $a$ and $b$, which implies the
existence of the usual (improper) limits $\lim_{x \to a} Q_0(x)$ and
$\lim_{x \to b}Q_0(x)$ for $x \in \rho(Q_0) \cap \{y \in \dR :
r(y)<0\}$; see Lemma~\ref{rcreatesgaps}. Observe, that the signs of
$\frac{r(x)}{x-a}$ and $(x-b)r(x)$ remain constant around $a$ and
$b$, respectively. Consequently, the multiplicities $\pi_a(rQ_0)$
and $\kappa_b(rQ_0)$ can be determined easily from the limit values
of $Q_0$ at $a$ and $b$.

The following theorem gives a characterization for the class
$\cN_\kappa^{\widetilde{\kappa}}(s)$ in the case of a simple
symmetric rational function $s$ of degree one.

\begin{theorem}\label{ResM1E}
Let $Q \in \cN_\kappa$ have the canonical factorization $Q=\phi
Q_0$, $Q_0 \in \cN$, and let $s$ be a symmetric rational function of
degree one as in \eqref{simrat}. Then the following statements are
equivalent:
\begin{enumerate}
\def\labelenumi {\rm (\roman{enumi})}
\item $Q \in \cN_\kappa^{\widetilde{\kappa}}(s)$ for some $\widetilde{\kappa}\in\dN$;
\item $0\leq s \leq \infty$ on $\sigma(Q)$ except for finitely many poles of $Q$;
\item  $0\leq s \leq \infty$ on $\sigma(Q_0)$ except for finitely many poles of $Q_0$.
\end{enumerate}
Furthermore, if any of the above equivalent statements holds, then
$sQ_0\in\cN_{\kappa_0}$ for some $\kappa_0\in\dN$ and the canonical
factorization of $sQ$ is given by $\widetilde{\phi}\widetilde{Q}_0$,
$\widetilde{Q}_0 \in \cN$, where
\[ 
 \widetilde{\phi} = \phi\psi,\quad
 \widetilde{Q}_0=\frac{s}{\psi}\,Q_0,\quad
 \psi(z)=\frac{(z-a)^{2\pi_a(sQ_0)}
 \prod_{i=1}^{n_1}(z-\alpha_i)^{2}}{(z-b)^{2\kappa_b(sQ_0)}\prod_{i=1}^{n_2}(z-\beta_i)^2}.
\]
Here the factors $(z-a)$ and $(z-b)$ with $\pi_a(sQ_0)$ and
$\kappa_b(sQ_0)$ as in \eqref{not1Q} and \eqref{not2Q} appear only
for the finite zero $a$ and pole $b$ of $s$. Moreover,
$\alpha_i\in\dR$, $1\leq i \leq n_1$, are the zeros of $Q_0$ which
satisfy $-\infty < s(\alpha_i) < 0$ and $\beta_i\in\dR$, $1\leq i
\leq n_2$, are the poles of $Q_0$ which satisfy $-\infty< s(\beta_i)
< 0$.
\end{theorem}

For the proof of Theorem~\ref{ResM1E} the following lemma will be
used. The lemma itself is proved by making use of
Lemma~\ref{transqw}, which allows to prove the statement via
rational functions $s$ of specific form.

\begin{lemma}\label{ResM1}
Let $Q=\phi Q_0$, $s$, $\pi_a(sQ_0)$, and $\kappa_b(sQ_0)$ be as in
Theorem~\ref{ResM1E}. Then $0\leq s\leq \infty$ on $\sigma(Q_0)$
implies that $Q \in \cN^{\widetilde{\kappa}}_\kappa(s)$ for some
$\widetilde{\kappa}\in\dN$.

In this case $sQ_0\in\cN_{\kappa_0}$ for some $\kappa_0\in\dN$ and
the canonical factorization of $sQ \in \cN_{\widetilde{\kappa}}$ is
given by $sQ=\widetilde{\phi}\widetilde{Q}_0$, $\widetilde{Q}_0 \in
\cN$, where
\[
 \widetilde{\phi} = \phi \psi, \quad \widetilde{Q}_0 = \frac{s}{\psi}\,Q_0,
 \quad
 \psi(z)=\frac{(z-a)^{2\pi_{a}(sQ_0)}(z-\alpha)^{2{j_\alpha}}}{(z-b)^{2\kappa_{b}(sQ_0)}}.
\]
Here $j_\alpha=1$ if $\alpha\in\dR$ is a zero of $Q_0$ for which
$-\infty<s(\alpha) < 0$ (there is at most one such zero $\alpha$ of
$Q_0$), and $j_\alpha=0$ otherwise.
\end{lemma}

\begin{proof}
Observe that if $sQ \in \cN_{\wt\kappa}$, then $-1/(sQ) =
(1/s)\cdot(-1/Q)\in \cN_{\wt\kappa}$. Hence by considering either
$s$ or $1/s$ we can restrict ourselves, without loss of generality,
to the case that $s \in \cN$. In this case $s$ takes the form
\[
 s(z) = \gamma \frac{z-a}{z-b}, \,\, \gamma(b-a)<0;
 \quad s(z) = \gamma (z-a), \,\, \gamma >0;
 \quad  \textrm{or}\quad s(z) = \frac{\gamma}{z-b}, \,\, \gamma <0.
\]

\textit{Case $\gamma >0$:} Then $s$ is given by
\[ s(z) = \gamma \frac{z-a}{z-b}, \,\, b<a; \quad  \textrm{or} \quad s(z) = \gamma (z-a).\]
Define $b =-\infty$ in the second case, so that in both cases $b < a$.

The assumption $0\leq s \leq \infty$ on $\sigma(Q_0)$ implies that
$(b,a)\subset \rho(Q_0)$. Moreover, since $Q_0 \in \cN$,
monotonicity of $Q_0$ (see \eqref{increasing}) implies that there
exists at most one zero $\alpha\in(b,a)$ of $Q_0$ and, furthermore,
the limits $\lim_{z\widehat{\to} a} Q_0(z)$ and
$\lim_{z\widehat{\to} b}Q_0(z)$ exist in $\dR\cup\{\pm\infty\}$, see
Lemma~\ref{limvalue}.

Next observe that by Corollary~\ref{Npres3456} ($b$ finite) and
\ref{Npres12} ($b$ infinite) $s Q_0 \in \cN$ if and only if $0 \leq
\gamma \lim_{z\widehat{\to} b} Q_0(z) < \infty$, in which case
$\kappa_b(sQ_0)=\pi_a(sQ_0)=j_\alpha=0$. Hence the statements hold
in this case with $\psi = 1$.

If $-\infty \leq \gamma \lim_{z\widehat{\to} b} Q_0(z) < 0$ and
$Q_0$ does not change its sign between $a$ and $b$, so that
$j_\alpha=0$ and $\gamma \lim_{z\widehat{\to} a} Q_0(z) \leq 0$,
then by Corollary~\ref{Npres3456} ($b$ finite) and \ref{Npres12}
($b$ infinite) $(1/s)Q_0\in \cN$. Since $sQ_0=\left(s\right)^2 (1/s)
Q_0$, Proposition~\ref{CharNK} shows that the statements hold in

this case with $\pi_a(sQ_0)=\kappa_b(sQ_0)=1$, $j_\alpha=0$,
$\widetilde{Q}_0= \gamma^2 (1/s) Q_0$ and
\begin{equation}\label{fac2}
\psi(z)=\frac{(z-a)^2}{(z-b)^2}, \quad b \in \dR, \quad \textrm{or} \quad \psi(z)=(z-a)^2, \quad b=\infty.
\end{equation}

If $-\infty \leq \gamma \lim_{z\widehat{\to} b} Q_0(z)<0$ and $Q_0$ changes its sign between $a$ and $b$, then there exists a (unique) zero $\alpha$ of $Q_0$ between $a$ and $b$ with $-\infty<s(\alpha)<0$. Hence, $0 < \gamma \lim_{z\widehat{\to} a} Q_0(z)\leq \infty$ and the sign of $Q_0$ is constant between $b$ and $\alpha$ and between $\alpha$ and $a$. Define $s_1(z) = \gamma \frac{z-b}{z-\alpha}$ if $b$ is finite and $s_1(z) = \gamma \frac{1}{z-\alpha}$ if $b$ is infinite and in both cases define $s_2(z)= \frac{z-a}{z-\alpha}$. Now Corollary~\ref{Npres3456} ($b$ finite) and \ref{Npres12} ($b$ infinite) imply that $s_1Q_0 \in \cN$. In addition, $s_1Q_0$ has a constant (positive) sign between $\alpha$ and $a$, in particular, $0 \leq \lim_{z\widehat{\to} \alpha} s_1(z)Q_0(z) < \infty$, see Lemma~\ref{pointlimit}. Thus, Corollary~\ref{Npres3456} shows that $s_2(s_1Q_0) \in \cN$. Define $\psi$ as
\begin{equation}\label{fac3}
\psi(z)= \frac{(z-\alpha)^2}{(z-b)^2}, \quad b \in \dR, \quad \textrm{or} \quad \psi(z)=(z-\alpha)^2, \quad b=\infty.
\end{equation}
Then $sQ_0 = \psi s_2s_1Q_0$ and, hence, Proposition~\ref{CharNK} shows that the
statements hold in this case with $\kappa_b(sQ_0)=j_\alpha=1$,
$\pi_a(sQ_0)=0$, $\wt Q_0 =s_2(s_1Q_0)$ and $\psi$ as in \eqref{fac3}.

\textit{Case $\gamma < 0$:} Using Lemma~\ref{transqw} the statement
can be reduced to the cases where $\gamma>0$. If $s(z) = \gamma
\frac{z-a}{z-b} \in \cN$, $\gamma <0$, then consider the functions
$s \circ \tau$ and $Q \circ \tau$ where $\tau(\lambda) = (a+b)/2 -
1/\lambda \in \cN$. Then
\[ (s \circ \tau)(\lambda) = -\gamma \frac{\lambda -
2/(b-a)}{\lambda - 2/(a-b)} \quad \textrm{and} \quad (Q\circ\tau)(\lambda)\in \cN_\kappa.\]
Clearly $\tau$ maps the interval with the endpoints $2/(b-a)$, $2/(a-b)$ to the complement of the interval between $a$, $b$ (in $\dR \cup\{\infty\}$).
Moreover, when transforming back to get the factorization for the initial function $rQ$ from
\eqref{fac2}, \eqref{fac3} note that
\[
 \frac{\left(\lambda- 2/(b-a)\right)^2}{\left(\lambda- 2/(a-b)\right)^2}
 =\frac{\left(z-a\right)^2}{\left(z-b\right)^2}
\]
and
\[
 \frac{\left(\lambda-\alpha\right)^2}{\left(\lambda- 2/(a-b)\right)^2}
=\left\{
 \begin{array}{ll}
  \dfrac{\alpha^2(a-b)^2}{4}\,\dfrac{\left(z-(\frac{a+b}{2}-\frac{1}{\alpha})\right)^2}{(z-b)^2}, & 0<|\alpha|<\dfrac{2}{|b-a|},\\
 \dfrac{(a-b)^2}{4}\,\dfrac{1}{2(z-b)^2}, & \alpha=0.
  \end{array}
 \right.
\]
Here $\alpha=0$ corresponds to the GZNT of $sQ$ at $\infty$ (also to the zero of $Q_0$ at $\infty)$.

Similarly, the case $\gamma/(z-b) \in \cN$ can be reduced to the
proven case by composing it with the transformation $\tau(\lambda) =
(b-1) - 1/\lambda \in \cN$.
\end{proof}

\begin{proof}[Proof of Theorem~\ref{ResM1E}]
(i) $\Rightarrow$ (iii) This holds by Lemma~\ref{rcreatesgaps}.

(ii) $\Leftrightarrow$ (iii) The factorization $Q=\phi Q_0$ implies
that the sets $\sigma(Q)$ and $\sigma(Q_0)$ can differ from each
other only at the zeros or poles of the rational
function $\phi$; see also \eqref{genStieltjesinv}.

(iii) $\Rightarrow$ (i) Without loss of generality assume that
$\gamma > 0$; compare the proof of Lemma~\ref{ResM1}. Then by the
assumption the open interval where $s$ is negative contains finitely
many ($n_2$) separated poles $\beta_i\in\sigma_p(Q_0)$. Since
$Q_0\in\cN$, there are also finitely many ($n_1$) zeros $\alpha_i$
of $Q_0$, and the zeros $\alpha_i$ and the poles $\beta_i$ of $Q_0$
between $a$ and $b$ are interlacing, in particular, $|n_1-n_2|\leq
1$. By considering either $Q_0$ or $-1/Q_0 \in \cN$ one can assume
that the zeros and poles of $Q_0 \in \cN$ on the open interval
between $a$ and $b$ are ordered as follows: $(\alpha_0\leq) \beta_1
\leq \alpha_1 \leq \ldots \leq \beta_{n_2} \leq \alpha_{n_2}$ (here
$\alpha_0$ is excluded if $n_1=n_2$). In particular, if $n_2=0$ then
(i) follows immediately from Lemma~\ref{ResM1}.

Now assume that $n_1 \geq n_2\geq 1$. Because of the disjointness of
the intervals $[\beta_i,\alpha_i]$, $1\leq i \leq n_2$, define
$s_{i}(z):=\frac{z-\alpha_i}{z-\beta_i}$, $1\leq i \leq n_2$. Then
Corollary~\ref{Npres3456} shows that $Q_0/s_{i}\in\cN$. Furthermore,
by Remark~\ref{HolCont} this function admits a holomorphic
continuation to the endpoints of the interval $(\beta_i,\alpha_i)$
and $Q_0(x)/s_{i}(x)>0$ for all $x\in[\beta_i,\alpha_i]$, while
outside the interval $[\beta_i,\alpha_i]$ the values of $Q_0$ and
$Q_0/s_i$ have the same sign. Consequently, $
Q_1=Q_0/(\prod_{i=1}^{n_2} s_{i})\in \cN$ admits a holomorphic
continuation to the open interval where $s$ is negative, and on this
interval $Q_1(x)>0$ (with $x>\alpha_0$ if $n_1>n_2$). Furthermore,
\begin{equation}\label{facP1}
 s(z)Q(z)=s(z)\phi(z) Q_0(z)
 =\left(\prod_{i=1}^{n_2} s_{i}(z)\right)s(z)\phi(z) {Q}_1(z).
\end{equation}
Now $s$ and $Q_1$ satisfy the condition of Lemma~\ref{ResM1} and therefore
\begin{equation}\label{facP2}
s(z) {Q}_1(z) = \frac{(z-a)^{2\pi_a(sQ_1)}(z-\alpha_0)^{2\pi_{\alpha_0}}}{(z-b)^{2\kappa_b(sQ_1)}}Q_2(z),
\end{equation}
where $Q_2 \in \cN$ and $\pi_{\alpha_0}=|n_1-n_2|$. Here $Q_2$ has no zeros where $s$ is negative, and $\alpha_0$ is a pole of $Q_2$ only if $\pi_{\alpha_0}=1$. In particular, $Q_1$ and $Q_2$ have different signs on the interval where $s$ is negative. Applying Lemma~\ref{ResM1} ($n_2$ times) shows that
\begin{equation}\label{facP3}
\prod_{i=1}^{n_2} s_{i}(z) Q_2(z) = \left(\prod_{i=1}^{n_2} s_{i}(z)\right)^2 \frac{Q_2(z)}{ \prod_{i=1}^{n_2} s_{i}(z)} = \prod_{i=1}^{n_2} \left(s_{i}(z)\right)^2 Q_3(z),
\end{equation}
where $Q_3 \in \cN$. Now \eqref{facP3} together with \eqref{facP1}
and \eqref{facP2} implies (i). Then, equivalently,
$sQ_0\in\cN_{\kappa_0}$ for some $\kappa_0\in\dN$; see
\eqref{facP1}. Moreover, it is clear that $\pi_a(sQ_1)=\pi_a(sQ_0)$
and $\kappa_b(sQ_1)=\kappa_b(sQ_0)$, since $\prod_{i=1}^{n_2}
s_{i}(z)$ is positive if $z$ tends to the finite zero or pole of
$s$. This proves the canonical factorization
$sQ=\widetilde{\phi}\widetilde{Q}_0$ with $\widetilde{\phi}$ given
in the statement.
\end{proof}

Note that Theorem~\ref{ResM1E} contains the analogous results
without factorizations on the classes $\cN_\kappa^\pm(a,b)$ from
\cite{KL75} as well as the results (in the scalar case) on the
classes $\cS^{\pm \wt\kappa}(a,b)$ from \cite{DM91,DM97} and on the
classes $\wt \cN_k^{\pm \kappa}(\dC)$ from \cite{Der95}.

\begin{remark}\label{discclass}
Theorem~\ref{ResM1E} gives in particular a characterization for the
classes $\cN_0^{\wt \kappa}(s)$, when $s$ is a simple rational
function of degree one. For instance, if $a \in \dR \cup\{\infty\}$
and $b \in \dR \cup\{\infty\}$, $a\neq b$, denote the zero and pole
of $s$, then the class $\cN_0^{\wt \kappa}(s)$ consists of all
Nevanlinna functions $Q$ such that either $Q$ has $\wt \kappa$ zeros
(poles) where $s$ is negative (on $\dR \cup\{\infty\}$) and
$\pi_a(sQ)=0$ ($\kappa_b(sQ)=0$) or $Q$ has $\wt \kappa-1$ zeros
(poles) where $s$ is negative (on $\dR \cup\{\infty\}$) and
$\pi_a(sQ)=1$ ($\kappa_b(sQ)=1$).
\end{remark}

Let $s$ be a simple symmetric rational function of degree one and
let $Q \in \cN_\kappa^{\wt \kappa}(s)$ have the canonical
factorization $\phi Q_0$, $Q_0 \in \cN$. Then by
Theorem~\ref{ResM1E} all the zeros $\alpha_i$ and poles $\beta_i$ of
$Q_0$ which satisfy $-\infty< s(\alpha_i)<0$ or
$-\infty<s(\beta_i)<0$ become GZNT's and GPNT's of $sQ_0$ (with
multiplicity one), respectively, and the only other points which can
become a GZNT and a GPNT of $sQ_0$ (with multiplicity one) are the
pole and zero of $s$ (including $\infty$), respectively.
Consequently, only the zeros of $\phi$, the zero of $s$, and
$\alpha_i$ can become GZNT's of $sQ$ and only the poles of $\phi$,
the pole of $s$, and $\beta_i$ can become GPNT's of $sQ$.

Furthermore, by means of \eqref{GZNT} or \eqref{not1Q} it can be
decided whether the (finite) zero of $s$ becomes a GZNT of $sQ_0$
and, similarly, by means of \eqref{GPNT} or \eqref{not2Q} it can be
decided whether the (finite) pole of $s$ becomes a $GPNT$ of $sQ_0$.
To decide whether $\infty$ is a GZNT or GPNT of $sQ_0$ (and hence
possible a GZNT or GPNT of $sQ$) one can use \eqref{GZNT} and
\eqref{GPNT}: $\infty$ is a GZNT or GPNT of $sQ_0$ if and only if
\[ 0 \leq \lim_{z \wh \to \infty} \left\{ zs(z)Q_0(z) \right\} < \infty \quad \textrm{or} \quad
-\infty \leq \lim_{z \wh \to \infty} \left\{ \frac{s(z)Q_0(z)}{z} \right\} < 0. \]

\subsection{Characterization of the classes $\cN_\kappa^{\widetilde{\kappa}}(r)$}
The results obtained in the previous subsection are extended to obtain a characterization for the classes $\cN_\kappa^{\wt
\kappa}(r)$, when $r$ is an arbitrary symmetric rational function.

The following lemma is used to prove the main result via Theorem~\ref{ResM1E}.

\begin{lemma}\label{sprod}
Let $s$ with $n=\deg s\ge 1$ be a simple symmetric rational
function, whose poles and zeros (all of order one) are real and
interlacing. Then $s$ admits a factorization $s = \prod_{i=1}^n
s_i$, where $s_i$ are rational functions of degree one, such that
\begin{equation*}\label{sprod0}
 D^-_i \cap D^-_j =\emptyset, \quad i \neq j, \quad \text{where}\quad
 D^-_i = \clos \{\,x \in \dR:\, s_i(x) \leq 0\,\}.
\end{equation*}
\end{lemma}

\begin{proof}
Let the numbers of real zeros and poles of $s$ be $l_1$ and $l_2$,
respectively. By the interlacing property one has $|l_1-l_2| \leq
1$. Moreover, by considering either $s$ or $1/s$ one can assume
without loss of generality that the zeros $a_i$ and poles $b_i$ of
$s$ have the ordering $({b}_0<)\, a_1 < b_1 < \ldots < a_{l_1} <
b_{l_1}$, where $b_0$ is excluded if $l_1=l_2$. Then $s$ can be
factorized as follows:
\[
 s(z) = \gamma \prod_{i=1}^{l_1} \frac{z- a_i}{z- b_i}, \,\, l_1 = l_2;
 \quad \text{or} \quad s(z) = \gamma \frac{1}{z- b_0}
 \prod_{i=1}^{l_1} \frac{z- a_i}{z- b_i}, \,\, l_2 = l_1+1;
\]
where $\gamma=\lim_{z\to\infty} s(z)$, if $l_1 = l_2$, and
$\gamma=\lim_{z\to\infty} zs(z)$, if $l_2 = l_1+1$. In the first
case, define
\[
 \left\{
 \begin{array}{l}
 s_1(z) = \gamma \dfrac{z- a_1}{z- b_{l_1}} \quad\text{and}\quad
 s_i(z) = \dfrac{z- a_i}{z- b_{i-1}}, \quad 2\leq i \leq l_1, \quad \text{if } \gamma<0;\\[4mm]
 s_1(z) = \gamma \dfrac{z- a_1}{z- b_1} \quad\text{and}\quad
 s_i(z) = \dfrac{z- a_i}{z- b_{i}}, \quad 2\leq i \leq l_1, \quad \text{if } \gamma>0.
 \end{array} \right.
 \]
In the second case, define
\[
 \left\{
 \begin{array}{l}
 s_i(z) = \dfrac{z- a_i}{z- b_{i-1}},\quad 1\leq i \leq l_1,
 \quad\text{and}\quad s_{l_2}(z) = \dfrac{\gamma}{z-b_{l_1}}, \quad \text{if } \gamma<0;\\[4mm]
 s_i(z) = \dfrac{z- a_i}{z- b_i},\quad 1\leq i \leq l_1,
 \quad\text{and}\quad s_{l_2}(z) = \dfrac{\gamma}{z-b_0}, \quad \text{if } \gamma>0.
 \end{array} \right.
\]
Then clearly $s = \prod_{i=1}^{n} s_i$ (with $n=l_2$) and it is easy
to check that in each case the factors $s_i$ satisfy the stated properties.
\end{proof}

\begin{theorem}\label{ResTot}
Let $Q \in \cN_\kappa$, $Q\neq 0$, have the canonical factorization
$Q=\phi Q_0$, $Q_0 \in \cN$, and let $r$ be a symmetric rational
function. Then the following statements are equivalent:
\begin{enumerate}
\def\labelenumi {\rm (\roman{enumi})}
\item $Q \in \cN_\kappa^{\widetilde{\kappa}}(r)$ for some $\widetilde{\kappa}\in\dN$;
\item $0\leq r \leq \infty$ on $\sigma(Q)$ except for finitely many poles of $Q$;
\item $0\leq r \leq \infty$ on $\sigma(Q_0)$ except for finitely many poles of $Q_0$.
\end{enumerate}
Furthermore, if $r=\psi s_0$ is the canonical factorization of $r$,
then $s_0Q_0\in\cN_{\kappa_0}$ for some $\kappa_0\in\dN$ and the
canonical factorization of $rQ$ is given by
$\widetilde{\phi}\widetilde{Q}_0$, $\widetilde{Q}_0 \in \cN$, where
\begin{equation*}\label{facPoU}
 \widetilde{\phi} = \phi\psi \wt \psi,\quad
 \widetilde{Q}_0=\frac{s_0}{\wt \psi}\,Q_0,\quad
 \wt \psi(z) =\frac{\prod_{i=1}^{m_1}(z-a_i)^{2\pi_{a_i}(s_0 Q_0)}\prod_{i=1}^{n_1} (z-\alpha_i)^2}
 {\prod_{i=1}^{m_2}(z-b_i)^{2\kappa_{b_i}(s_0 Q_0)}\prod_{i=1}^{n_2} (z-\beta_i)^2}.
\end{equation*}
Here $\pi_{a_i}(s_0Q_0)$ and $\kappa_{b_i}(s_0Q_0)$ are as in
\eqref{not1Q} and \eqref{not2Q} for the finite zeros $a_i$ and poles
$b_i$ of $s_0$, respectively. Moreover, $\alpha_i\in\dR$, $1\leq i
\leq n_1$, are the zeros of $Q_0$ which satisfy $-\infty<
s_0(\alpha_i) < 0$ and $\beta_i\in\dR$, $1\leq i \leq n_2$, are the
poles of $Q_0$ which satisfy $-\infty< s_0(\beta_i) < 0$.
\end{theorem}

\begin{proof}
(i) $\Rightarrow$ (iii) Again this holds by
Lemma~\ref{rcreatesgaps}.

(ii) $\Leftrightarrow$ (iii) As in the proof of
Theorem~\ref{ResM1E}, this follows from $Q=\phi Q_0$.

(ii) $\Rightarrow$ (i) Let $r=\psi s_0$, $s_0 \in \cN$, be the
canonical factorization or $r$ given by Proposition~\ref{rrfNev}.
Then $s_0$ satisfies the assumptions in Lemma~\ref{sprod}. Hence,
there exists a factorization $s_0 = \prod_{i=1}^n s_i$, where $s_i$,
$1\leq i\leq n$, are simple rational functions of degree one
satisfying the properties \eqref{sprod0} in Lemma~\ref{sprod}. Now,
as a consequence of the properties \eqref{sprod0},
Theorem~\ref{ResM1E} can be applied inductively ($n$ times) to show
that $s_0 Q \in \cN_{\kappa_1}$ for some $\kappa_1 \in \dN$.
Consequently, $rQ = \psi s_0 Q \in \cN_{\wt\kappa}$ for some
$\wt\kappa \in \dN$ by Proposition~\ref{CharNK}. This proves (i).

Finally, since the rational functions $s_1,\dots,s_n$ satisfy the
properties \eqref{sprod0}, one can produce the canonical
factorization for the product $rQ$ stepwise by applying the
canonical factorization in Theorem~\ref{ResM1E} to the products
$\prod_{i=1}^{k}s_iQ_0$, $1\le k\le n$.
\end{proof}

In the canonical factorization of $rQ$ in Theorem~\ref{ResTot} there
can occur many cancelations in the products $\wt\phi=\phi\psi \wt
\psi$ and $s_0/\wt\psi$. The canonical factorization $r=\psi s_0$ of
$r$ does not take  into account the structure of $Q_0$, and hence
there can occur already cancelations in the product $\psi \wt \psi$.
It is easy to see that the canonical factorization of $rQ$ in
Theorem~\ref{ResTot} can be reformulated via an arbitrary (not
necessarily canonical) factorization $r=\psi s$ of $r$, where $\psi$
is a nonnegative rational function and $s$ is a simple symmetric
rational function. In this case the formulas for $\wt\phi$ and $\wt
Q_0$ appear as in Theorem~\ref{ResTot}, where the multiplicities
$\pi_{a_i}(sQ_0)$ and $\kappa_{b_i}(sQ_0)$ can still be calculated
as in \eqref{not1Q} and \eqref{not2Q} by the simplicity of $s$. In
fact, the factorization of $rQ$ in Theorem~\ref{ResTot} shows that
it is possible to factorize $r$ with respect to $Q_0$ as
$r=\psi^\prime s^\prime$ such that $\psi^\prime$ is nonnegative and
$s^\prime Q_0\in\cN$: take $\psi^\prime=\psi\wt\psi$ and
$s^\prime=s_0/\wt \psi$.

Theorem~\ref{ResTot} together with the discussion after
Remark~\ref{discclass} yields the following result.

\begin{proposition}\label{enigecand}
Let $Q \in \cN_\kappa^{\wt\kappa}(r)$ and let $\phi Q_0$, $Q_0 \in \cN$, be its canonical factorization. Then the only points which can become GZNT's and GPNT's of $rQ$ are
\begin{enumerate}
\def\labelenumi {\rm (\roman{enumi})}
\item the zeros and poles of $r$ in $\dR\cup\{\infty\}$;
\item the GZNT's and GPNT's of $Q$;
\item all the zeros $\alpha_i$ and the poles $\beta_i$ of $Q_0$ in $\dR\cup\{\infty\}$
for which $-\infty< r(\alpha_i)<0$ or $-\infty< r(\beta_i) <0$.
\end{enumerate}
In particular, if $Q \in \cN_0^{\wt\kappa}(r)$ then all the points
$\alpha_i$ and $\beta_i$ in (iii) become GZNT's and GPNT's of $rQ$,
respectively, and the only other possible GZNT's and GPNT's of $rQ$
are the zeros and poles of $r$, respectively.
\end{proposition}

Here is an example of the type of results which is obtained if
$\cN_\kappa$-function are multiplied with non-simple symmetric
rational functions. This corollary contains as a special case
\cite[Corollary 4.9]{KWW06}.

\begin{corollary}
Let $Q \in \cN_\kappa$ have the canonical factorization $\phi Q_0$,
$Q_0 \in \cN$, and let $r(z)=\gamma (z-a)^{2m+1}$, $a,\gamma \in
\dR$ and $m \in\dN$. Then $rQ \in \cN_{\widetilde{\kappa}}$ if and
only if the interval $(-\infty,a)$, $\gamma>0$, or $(a,\infty)$,
$\gamma<0$, contains only finitely many poles of $Q_0$.

Furthermore, $rQ_0\in\cN_{\kappa_0}$ for some $\kappa_0\in\dN$ and
the canonical factorization of $rQ$ is given by
$\widetilde{\phi}\widetilde{Q}_0$, $\widetilde{Q}_0 \in \cN$, where
\begin{equation*}\label{facP02}
 \widetilde{\phi}(z) = \frac{\phi(z)(z-a)^{2(\pi_a((z-a)Q_0(z)) + m) }\prod_{i=1}^{n_1}(z-\alpha_i)^{2}}{\prod_{i=1}^{n_2}(z-\beta_i)^2} \quad \textrm{and} \quad
 \widetilde{Q}_0=\frac{r}{\wt \phi}\,Q.
\end{equation*}
Here $\pi_a((z-a)Q_0(z))$ is as in \eqref{not1Q}. Moreover,
$\alpha_i$, $1\leq i \leq n_1$, are the zeros of $Q_0$ satisfying $\alpha_i \in (-\infty,a)$, $\gamma>0$, or $\alpha_i \in (a,\infty)$,
$\gamma<0$, and $\beta_i$, $1\leq i \leq n_2$,
are the poles of $Q_0$ satisfying $\beta_i \in (-\infty,a)$, $\gamma>0$, or $\beta_i \in (a,\infty)$,
$\gamma<0$.
\end{corollary}

\subsection{Characterization of the class $\cN_0^{0}(r)$}
Assume that $Q \in \cN_\kappa^{\wt \kappa}(r)$ and let $Q=\phi Q_0$
and $\wt Q=\wt \phi \wt Q_0$ with $Q_0,\wt Q_0 \in \cN$ be the
(unique) canonical factorizations of $Q$ and $rQ$. Then
\begin{equation}\label{cnkkcn00}
\wt Q_0(z) = \frac{r(z)\phi(z)}{\wt \phi (z) }Q_0(z),
\end{equation}
i.e. $Q_0 \in \cN_0^0(r\phi/\wt \phi)$. This shows that of all the
classes $\cN_\kappa^{\wt \kappa}(r)$, the classes $\cN_0^0(r)$ play
a key role. The next theorem gives a characterization for these
classes.

\begin{theorem}\label{Newth}
Let $Q_0\in\cN$, $Q_0\neq 0$, and let $r$ ($\deg r>0$) be a
symmetric rational function. Then $Q_0 \in \cN_0^0(r)$ if and only
if
\begin{enumerate}
\def\labelenumi {\rm (\roman{enumi})}
\item $0\leq r \leq \infty$ on $\sigma(Q_0)$ and if $-\infty<r(x)<0$, $x\in\dR$, then $Q_0(x)\neq 0$;
\item every finite zero and pole $\gamma$ of $r$ is real and of order at most two:
\begin{enumerate}
\item[(a)] if $\gamma$ is a zero (pole) of $r$ of order two, then $\gamma$ is an isolated pole of $Q_0$
(an isolated zero of $Q_0$, i.e. $\gamma\in\rho(Q_0)$,
$Q_0(\gamma)=0$) and, respectively,
\[ -\infty< \lim_{z\wh \to \gamma} \frac{r(z)}{(z-\gamma)^2} <0 \quad \textrm{or} \quad -\infty< \lim_{z\wh \to \gamma} (z-\gamma)^2 r(z) <0;\]
\item[(b)] if $\gamma$ is a zero (pole) of $r$ of order one and
$\iota_\gamma = \lim_{z \wh \to \gamma}\sgn
\frac{r(z)}{z-\gamma}$ (or $\iota_\gamma = \lim_{z \wh \to
\gamma} \sgn((z-\gamma)r(z))$), then, respectively,
\[ 0< \lim_{z \wh \to \gamma} \iota_\gamma Q_0(z) \leq \infty \quad \textrm{or}\quad -\infty< \lim_{z \wh \to \gamma} \iota_\gamma Q_0(z)\leq 0. \]
\end{enumerate}
\end{enumerate}
\end{theorem}
\begin{proof}
If $Q_0 \in \cN_0^0(r)$, then (i) holds by Theorem~\ref{ResTot}; see
also Proposition~\ref{enigecand}.

Now let $r=\psi s$ be the canonical factorization of $r$ as in
Proposition~\ref{rrfNev}. Since $rQ_0\in\cN_0$, the factorization in
Theorem~\ref{ResTot} shows that $\psi\wt\psi \equiv 1$ and hence the
order of every finite zero and pole of $r$ is at most two.

If $\gamma$ is a zero (pole) of $r$ of order two, then in view of
$\psi\wt\psi \equiv 1$ the factorization in Theorem~\ref{ResTot}
shows that $Q_0$ should have a pole (zero) at $\gamma$ and that
$-\infty<s(\gamma)<0$; thus $\gamma$ is an isolated pole (zero) of
$Q_0$ by Lemma~\ref{rcreatesgaps}. This gives (ii)(a).

If $\gamma$ is a simple zero (or pole) of $r$, then the limit value
$\iota_\gamma$ as defined in the statement belongs to
$\dR\setminus\{0\}$. It follows from $rQ_0\in\cN_0$ that the
(improper) limit $\lim_{z \wh \to \gamma} \iota_\gamma Q_0(z)$
exists and satisfies the given inequalities (with $\gamma$ a zero or
pole of $r$, respectively); see Lemma~\ref{pointlimit}. Thus (ii)(b)
holds.

Conversely, if the condition (i) holds, then by Theorem~\ref{ResTot}
$rQ \in \cN_\kappa$ for some $\kappa \in \dN$. Moreover, if
$-\infty<r(x)<0$ then $x\in\rho(Q_0)$ by Lemma~\ref{rcreatesgaps}
and $Q_0(x)\ne 0$ by the assumption in (i). According to
Proposition~\ref{enigecand} only the zeros and poles of $r$ can
produce GZNT and GPNT for $rQ_0$. The assumptions in (ii)(a) and (ii)(b)
imply that if $\gamma$ is a zero (pole) of $r$, then
$\pi_\gamma(rQ_0)=0$ ($\kappa_\gamma(rQ_0)=0$); see \eqref{GZNT},
\eqref{GPNT}. Therefore, $rQ \in \cN$.
\end{proof}

Let $Q_0 \in \cN_0^0(r)$ for a symmetric rational function $r$. By
Theorem~\ref{Newth} $r$ can have only real zeros and poles of order
at most two. Clearly, such rational functions can be written as the
product of degree one symmetric rational functions. The following
results show how rational functions of degree one can be chosen in
such a way that the stepwise products with $Q_0$ also stays in the
class of Nevanlinna functions.

\begin{proposition}\label{CorN00a}
Let $Q_0 \in \cN_0^0(r)$. Then there exist rational functions $s_i$,
$1\leq i \leq n$, of degree one, such that
\[ \left(\prod_{i=1}^j s_i\right)Q_0 \in \cN, \quad 1\leq j \leq n,
 \quad \textrm{and} \quad r=\prod_{i=1}^n s_i,\]
and there are no cancelations between the factors $s_i$. Moreover,
the even order poles and zeros of $r$ are interlacing on intervals,
where $r$ is not positive.
\end{proposition}

\begin{proof}
Let $\phi s$ be a decomposition of $r$, where $s$ contains all the
simple zeros and poles of $r$ and, hence, $\phi$ is nonnegative, see
Theorem~\ref{Newth}. The rational factors $s_i$ of $r$ of degree one
will be constructed in two steps. First it is shown that on each
finite maximal interval $(a,b)$, where $s$ is negative a suitable
factorization can be defined involving all the poles and zeros of
$r$ contained in $[a,b]$. Then the desired factorization is obtained
inductively by considering all such negative intervals of $s$.

\emph{Step 1.} Let $(a,b)$ with $a,b\in\dR$ be a finite maximal
interval, where $s$ is negative. By the maximality assumption the
endpoints of the interval $(a,b)$ are zeros or poles of $s$, and
hence also of $r$, which by Theorem~\ref{Newth} must be of order
one. Let $r_{[a,b]}$ be the factor of $r$, which contains all the
finite zeros and poles of $r$ on the interval $[a,b]$. Decompose
$r_{[a,b]}=\phi_{[a,b]} s_{[a,b]}$, where $\phi_{[a,b]}$ is
nonnegative and $s_{[a,b]}$ is simple with $s_{[a,b]}(x)=s(x)<0$ for
$x\in(a,b)$; note that the only finite zeros and poles of
$s_{[a,b]}$ are $a$ and $b$. On the other hand, by
Theorem~\ref{Newth} (ii)(a) all the finite poles and zeros of
$\phi$, thus also of $\phi_{[a,b]}$, are of order two and they are
necessarily zeros and poles of $Q_0$ on the interval $(a,b)$.
Furthermore, by part (i) of Theorem~\ref{Newth} $Q_0$ cannot have
any other poles or zeros on the interval $(a,b)$, since if
$-\infty<r(x)<0$, then $x\in\rho(Q_0)$ and $Q_0(x)\neq 0$. It
follows that the ordered zeros $\alpha_j$ (poles of $\phi_{[a,b]}$)
and poles $\beta_j$ (zeros of $\phi_{[a,b]}$) of $Q_0\in\cN$ in
$(a,b)$ are interlacing:
\[ (\alpha_0 \leq) \beta_1 \leq \alpha_1 \leq \ldots \leq \beta_{n} \leq \alpha_{n} (\leq \beta_{n+1})\]
where $\alpha_0$ or $\beta_{n+1}$ need not occur (due to ordering).
Consequently, $\phi_{[a,b]}$ is given by
\[
 \phi_{[a,b]}(z)= \left(\frac{1}{(z-\alpha_0)^2}\right)\cdot\frac{\prod_{i=1}^n
  (z-\beta_i)^2}{\prod_{i=1}^n(z-\alpha_i)^2}\cdot \biggl((z-\beta_{n+1})^2 \biggr),
\]
where the first and last term in the brackets is excluded if,
respectively, $\alpha_0$ or $\beta_{n+1}$ does not exist.
Furthermore, by Theorem~\ref{Newth} (ii)(b) $a$ is a zero or pole of
$s_{[a,b]}$ if
\begin{equation}\label{coreq1}
-\infty \leq \lim_{z \wh \to a} Q_0(z) <0 \quad \textrm{or}\quad  0
\leq \lim_{z \wh \to a} Q_0(z) <\infty,
\end{equation}
respectively, and $b$ is a zero or pole of $s_{[a,b]}$ if
\begin{equation}\label{coreq2}
0< \lim_{z \wh \to b}Q_0(z) \leq \infty \quad \textrm{or}\quad
-\infty < \lim_{z \wh \to b} Q_0(z)\leq 0,
\end{equation}
respectively. Now, let $\wt s_i$ be defined by
\[ \wt s_i(z) = \frac{z-\beta_i}{z-\alpha_i}, \quad i=1,\ldots n, \quad \wt s_{n+1}(z)=1;
 \textrm{ or }  \wt s_i(z) = \frac{z-\beta_{i}}{z-\alpha_{i-1}}, \quad i=1,\ldots n+1,\]
if $\alpha_0$ or $\beta_{n+1}$ does not exist, respectively, if
$\alpha_0$ and $\beta_{n+1}$ both exist. Moreover, let $s_1$ and
$s_2$ be defined by
\[ s_1(z) =\frac{z-a}{z-b}, \,\, s_2(z) =1; \quad s_1(z) =\frac{z-b}{z-a}, \,\, s_2(z)=1;\]
\[ s_1(z) = \frac{z-a}{z-\alpha_0}, \,\, s_2(z)=\frac{z-b}{z-\alpha_0};
 \quad \textrm{or} \quad s_1(z) = \frac{z-\beta_{n+1}}{z-a}, \,\, s_2(z) = \frac{z-\beta_{n+1}}{z-b},\]
if $\alpha_0$ and $\beta_{n+1}$ do exist, if $\alpha_0$ and
$\beta_{n+1}$ do not exist, if only $\alpha_0$ exists, or if only
$\beta_{n+1}$ exists, respectively. Then by applying
Theorem~\ref{Newth} (or Corollary~\ref{Npres3456}) it is seen that
$\left(\prod_{i=1}^j \wt s_i\right) Q_0 \in \cN$ for all $1\leq j
\leq n+1$. In particular, $Q_1 := \left(\prod_{i=1}^{n+1} \wt
s_i\right) Q_0 \in \cN$ has at most one zero ($\alpha_0$) or one
pole ($\beta_{n+1}$) in $(a,b)$ (by construction $Q_1$ cannot have
both a zero and pole on $(a,b)$). The monotonicity property (see
\eqref{increasing}) together with \eqref{coreq1} and \eqref{coreq2}
implies that, if $\alpha_0$ is a zero of $Q_1$, then $a$ and $b$
both are zeros of $s$, if $\beta_{n+1}$ is a pole of $Q_1$, then $a$
and $b$ both are poles of $s$. If $Q_1$ neither has a pole nor a
zero on $(a,b)$, then $Q_1$ is either positive ($\alpha_0$ and
$\beta_{n+1}$ do not exist) or negative ($\alpha_0$ and
$\beta_{n+1}$ both exist) on $(a,b)$; in the first case $a$ is a
pole and $b$ is a zero of $s$, and in the second case $a$ is a zero
and $b$ is a pole of $s$. In all cases
$\left(\prod_{i=1}^ks_i\right)Q_1 \in \cN$, $1\leq k \leq 2$, by
Theorem~\ref{Newth} (or Corollary~\ref{Npres3456}) and, hence, $Q_2
:= \left(\prod_{i=1}^2 s_i\right) Q_1 \in \cN$. Finally, by
construction, for $x\in (a,b)$ one has $Q_2(x)<0$ ($Q_2(x)>0$) if
and only if $Q_1(x)>0$ ($Q_1(x)<0$) and $\alpha_0$ is a pole and
$\beta_{n+1}$ is a zero of $Q_2$ (when they exist). Since $Q_1$ and
$Q_2$ have opposite signs on the interval, it follows again from
Theorem~\ref{Newth} (or Corollary~\ref{Npres3456}) that $\left(
\prod_{i=1}^j \wt s_i\right)Q_2 \in \cN$, $1\leq j \leq n+1$.

As a conclusion, the above construction shows that $r_{[a,b]} =
(\prod_{i=1}^{n+1}\wt s_i^2)(\prod_{i=1}^2 s_i)$ and, furthermore,
that for all $1\leq j \leq n+1$ and $1\leq k \leq 2$, the products
\begin{equation}\label{CorN00aprep}
 \left(\prod_{i=1}^j \wt s_i\right) Q_0, \,\, \left(\prod_{i=1}^k s_i\right)
 \left(\prod_{i=1}^{n+1} \wt s_i\right) Q_0, \, \textrm{ and } \,
 \left(\prod_{i=1}^j \wt s_i\right)\left(\prod_{i=1}^2 s_i\right)
 \left(\prod_{i=1}^{n+1} \wt s_i\right) Q_0
\end{equation}
are Nevanlinna functions. In particular, $r_{[a,b]}Q_0\in\cN_0$.

\emph{Step 2.} The result \eqref{CorN00aprep} in Step 1 can be
applied to every maximal bounded interval on which $s$ is negative,
and if $s$ is negative on an interval $(-\infty,b)$, $(a,\infty)$ or
on $(-\infty,b)\cup (a,\infty)$, then it can be transformed to an
interval of the type $(a,b)$ by means of a rational function of
degree one using Lemma~\ref{transqw} (cf. the proof of
Lemma~\ref{ResM1}). Now proceed inductively by applying Step 1 to
each maximal interval where $s$ is negative. Note that the closures
of these intervals do not overlap, since $s(x)$ changes its sign,
when $x$ passes a pole or zero of $s$. After one interval $(a,b)$,
where $s$ is negative, is considered, the new functions to which
\eqref{CorN00aprep} in Step 1 is applied are $\wt
Q_0=r_{[a,b]}Q_0\in\cN_0$ and $r/r_{[a,b]}$, whose maximal negative
intervals coincide with those of $r$, apart from the interval
$(a,b)$, since $0\le r_{[a,b]}(x)\le\infty$ if $x\not\in[a,b]$ and
$-\infty\le r_{[a,b]}(x)\le0$ if $x\in[a,b]$. This gives the
factorization of $r$ as the product of $s_i$, without cancelations.
This complete the proof.
\end{proof}

The order of the rational functions $s_i$ in
Proposition~\ref{CorN00a} is essential; they cannot be reordered in
an arbitrary manner, since then some of the products may produce
functions which are not Nevanlinna functions; see
Example~\ref{exlog} below. In particular, in \eqref{CorN00aprep} the
factors $s_i$ and $\wt s_i$ can be reordered within the products
$\prod_{i=1}^j \wt s_i$ and $\prod_{i=1}^k s_i$, but in general it
is not possible to interchange factors between the three product
terms that appear \eqref{CorN00aprep}.

\begin{example}\label{exlog}
The function $\Log z$ (the principle branch of the logarithm) is a
Nevanlinna function; its integral representation is given by
\[ \Log (z) = \int_{(-\infty,0]} \left(\frac{1}{t-z} - \frac{t}{1+t^2}\right)dt,\]
see \cite[p. 27]{Don74}. Now consider the function $Q$ defined as
\[ Q(z) = \Log(z) + (2-z)^{-1} -1.\]
Then $Q$ has a zero at $1$ and a pole at $2$. Moreover, $Q$ has pole
also at $0$ and one further zero at a point $c_0>2$. Applying
Theorem~\ref{Newth} (or Corollary~\ref{Npres3456}) one concludes
that $r_1Q,r_2r_1Q,r_1r_2r_1Q \in \cN$, where
\[
 r_1(z) = \frac{z-2}{z-1} \quad \textrm{and} \quad
 r_2(z)= \frac{z-a}{z-b}, \quad
  a \in [0,1), b \in (2,c_0].
\]
In particular, $rQ\in\cN_0$ if
$r(z)=\frac{(z-2)^2}{(z-1)^2}\frac{z-a}{z-b}$ with $a$ and $b$ as
indicated. Note that $r_2Q \in \cN_1$: the only GZNT is at $z=1$ and
the only GPNT is at $z=2$; see Theorem~\ref{ResM1E}. Similarly
$r_1^2Q \in \cN_1$ with a GZNT at $z=2$ and a GPNT at $z=1$; see
Proposition~\ref{CharNK}. Note that if \eqref{CorN00aprep} is
applied to $Q\in\cN_0^0(r)$, it produces the factors $r_1$, $r_2$,
and $r_1$ in this order.

Note that one cannot use simpler factors of $r$ in this connection.
For instance, the function $(z-2)Q(z)$ cannot belong to the class
$\cN_\wt\kappa$ for any $\wt\kappa\in\dN$, since $x-2<0$ for $x<2$;
see Theorem~\ref{ResTot}. However, since $Q(c_0)=0$ and $Q(x)>0$ for
$x>c_0$, one concludes from Theorem~\ref{Newth} that
\[
 Q_c(z) = \frac{Q(z)}{c-z}\in\cN_0, \, c\ge c_0, \quad \textrm{and} \quad
 Q_2(z)= (2-z)Q_{c_0}(z)=\frac{z-2}{z-c_0}\,Q_0(z)\in\cN_0.
\]
\end{example}

Proposition~\ref{CorN00a} shows that the class $\cN_0^0(r)$ is
in particular interesting when $r$ is a symmetric rational function
which has only simple zeros and poles (including $\infty$). In that
case the following theorem holds; it generalizes the corollaries
established in Sections~\ref{sec23} and \ref{sec24}.
\begin{theorem}\label{productinNg}
Let $Q_0\in\cN$, $Q_0\neq 0$, and let $s$ be a symmetric rational
function, whose zeros and poles are real and simple (including
$\infty$). Then the following statements are equivalent:
\begin{enumerate}
\def\labelenumi {\rm (\roman{enumi})}
\item $Q_0 \in \cN_0^0(s)$;
\item
\begin{enumerate}
\item if $-\infty < s(x)<0$ then $Q_0$ is holomorphic at $x$ and $Q_0(x)\ne 0$;
\item $\pi_{a_i}(sQ_0)=0$ and $\kappa_{b_i}(sQ_0)=0$,
see \eqref{not1Q} and \eqref{not2Q}, for every finite zero $a_i$ and every finite pole $b_i$ of $s$;
\end{enumerate}
\item
\begin{enumerate}
\item if $(\alpha,\beta)$, $\alpha,\beta \in \dR \cup \{\pm \infty\}$, is a maximal interval
such that $-\infty <s(x) <0$ for all $x \in (\alpha,\beta)$, then
$Q_0$ is holomorphic and either $Q_0(x)>0$ or $Q_0(x)<0$ on
$(\alpha,\beta)$.
\item if $(\alpha,\beta)$ is a maximal interval as in (a) and, moreover,
if $0 < Q_0(x) < \infty$ for $\alpha < x < \beta$ and $\alpha\in\dR$
(or $\beta\in\dR$), then
$\alpha$ is a pole of $s$ ($\beta$ is a zero of $s$); if
$-\infty < Q_0(x) < 0$ for $\alpha < x < \beta$ and
$\alpha\in\dR$ (or $\beta\in\dR$), then $\alpha$ is a zero
of $s$ ($\beta$ is a pole of $s$).
\end{enumerate}
\item
\begin{enumerate}
\item ${\sigma}(Q_0) \cap \{\, x \in \dR\cup \{\infty\} : -\infty< s(x)<0\} =
\emptyset$;
\item for every pole $b$ of $s$
\[ 
-\infty< \lim_{z \wh \to b\in \dR} \left\{ (z-b)s(z)Q_0(z)\right\} \leq
0, \quad 0 \leq \lim_{z \wh \to b =\infty} \left\{ \frac{s(z)Q_0(z)}{z}\right\}
< \infty.
\]
\end{enumerate}
\end{enumerate}
\end{theorem}
\begin{proof}
(i) $\Leftrightarrow$ (ii) This is clear by Theorem~\ref{Newth}.

(i),(ii) $\Rightarrow$ (iv) This is immediate, see also Lemma~\ref{pointlimit}.

(iv) $\Rightarrow$ (i) This implication follows from the fact that a generalized Nevanlinna
function is a Nevanlinna function if and only if it has no generalized poles.

(ii) $\Leftrightarrow$ (iii) The equivalence of (ii)(a) and (iii)(a)
is obvious. Clearly, every $a_i$ or $b_i$ as in (ii)(b) is the
endpoint of a unique maximal interval $(\alpha,\beta)$ where $s$ is
negative, and, conversely, every finite endpoint of a maximal
interval $(\alpha,\beta)$ is a zero $a_i$ or a pole $b_i$ of $s$.
Recall that by \eqref{not1Q}, \eqref{not2Q} $\pi_{a_i}(sQ_0)=0$ or
$\kappa_{b_i}(sQ_0)=0$ if and only if
\[ 0 < \lim_{z \wh \to a_i} \left\{\frac{s(z)Q_0(z)}{z-a_i}\right\}\le \infty
\quad \textrm{or} \quad -\infty < \lim_{z \wh \to b_i}
\left\{(z-b_i)s(z)Q_0(z)\right\} \leq 0,\] respectively. Hence, the
equivalence of (ii)(b) and (iii)(b) follows from the observation
that $\lim_{z \wh \to a_i}\frac{s(z)}{z-a_i} \in \dR$ ($\lim_{z \wh
\to b_i}(z-b_i)s(z) \in \dR$) is positive if and only if $s$ is
positive on a neighborhood to the right of $a_i$ ($b_i$) or,
equivalently, negative on a neighborhood to the left of $a_i$
($b_i$).
\end{proof}

Theorem~\ref{productinNg} (iii)(b) gives some further information
about the location of zeros and poles of $s$, because by simplicity
of $s$ the sign of $s(x)$ changes, when $x$ passes a pole or zero of
$s$; compare Proposition~\ref{CorN00a}.

\begin{remark}\label{interlacing}
Let $Q_0 \in \cN_0^0(s)$, $Q_0\neq 0$, with a simple symmetric
rational function $s$. Then the poles and zeros of $s$ satisfy the
following (interlacing type) property: if all the finite zeros and
poles are ordered on the real line, then next to (i.e. before or
after) each zero of $s$ there is at least one pole of $s$, and next
to each pole of $s$ there is at least one zero of $s$.
\end{remark}

The location property stated in Remark~\ref{interlacing} includes
the case where the poles and zeros of $s$ are interlacing. Hence,
Theorem~\ref{productinNg} generalizes in particular the class
$\cS(E_m)$ introduced in \cite[p. 396]{KNapp}, which in the present
notation would be $\cN_0^0(s)$ for a simple symmetric rational
function whose zeros and poles are interlacing (starting with a
finite pole).

If $-\infty<\lim_{x\to\pm\infty}s(x)<0$ then $Q_0$ is holomorphic
also at $\pm\infty$ and, moreover, $\lim_{x\to\pm\infty}Q_0(x)\neq
0$; see Lemmas~\ref{genStieltjesinvApplied},~\ref{pointlimit}.
Hence, in Theorem~\ref{productinNg} (ii) (a) one can also include
the point $x=\pm\infty$. Similarly if, for instance, $s(x)<0$ on
$(-\infty,\beta)$ and this interval is maximal in the sense that
$-\infty$ is either a pole or zero of $s$, then in
Theorem~\ref{productinNg} (iii) (b), $0 < Q_0(x) < \infty$ for
$-\infty < x < \beta$ implies that $-\infty$ is actually a pole of
$s$, while $-\infty < Q_0(x) < 0$ for $-\infty < x < \beta$ implies
that $-\infty$ is a zero of $s$.

Some further information on the class $\cN_0^0(r)$, now related to
the functions $Q_0$ and $rQ_0$, is given in the next result. This
information will be used for constructing models for functions
belonging to the class $\cN_\kappa^{\wt\kappa}(r)$.

\begin{proposition}\label{vorige}
Let $Q_0 \in \cN_0^0(r)$ and let $a_i \in \dR \cup\{\infty\}$,
$1\leq i \leq n_1$, and $b_i \in \dR \cup\{\infty\}$, $1\leq i \leq
n_2$, be some sets of zeros and poles of $r$. Then
\[
 Q_0 \in \bigcap_{i=1}^{n_2} \cN(b_i,1) \quad \textrm{and} \quad rQ \in \bigcap_{i=1}^{n_1} \cN(a_i,1).
\]
\end{proposition}

\begin{proof}
If $b$ is a simple pole of $r$, then $\lim_{z\wh \to b}
(z-b)r(z)\in\dR\setminus\{0\}$, $b \in \dR$, and $\lim_{z \wh \to b}
\frac{r(z)}{z}\in\dR\setminus\{0\}$, $b=\infty$. Hence, if $Q_0 \in \cN_0^0(r)$,
then Theorem~\ref{productinNg} (iv)(b) shows that $\lim_{z \wh \to
b}Q_0(z)\in\dR$. Moreover, since $r(x)$ changes its sign when $x$
passes a simple pole of $r$, $Q_0$ is holomorphic on an interval
with a pole of $r$ as its endpoint; see Theorem~\ref{productinNg}
(ii)(a). Now, according to Corollary~\ref{limitaschar} $Q_0\in
\cN(b,1)$.

If $b \in \dR$ is a pole of $r$ of order greater than one, then its
order is two and by Theorem~\ref{Newth}~(ii)(a) $\beta$ is an isolated
zero of $Q_0$, in particular, $\beta\in\rho(Q_0)$. Hence, again
$Q_0\in \cN(b,1)$ by Corollary~\ref{limitaschar}.

The statement concerning the Nevanlinna function $rQ_0$ is now
obtained e.g. by means of the equivalence $Q_0 \in \cN_0^0(r)$ if
and only if $rQ_0 \in \cN_0^0(\frac{1}{r})$.
\end{proof}

Proposition~\ref{vorige} implies that the function $Q_0 \in
\cN_0^0(r)$ does not have poles at the poles of $b_i$ of $r$; see
Proposition~\ref{charclassspectral}. However, the function $rQ_0$
may have poles at these points. In fact, there is a pole of $rQ_0$
at $b_i$ when $b_i$ is a pole of order $2$ of $r$ or when $b_i$ is a pole of order $1$
of $r$ and $\lim_{z\wh\to b_i} Q_0(z)\neq 0$; cf. Lemma~\ref{pointlimit}.

\section{Realizations of a subclass of Nevanlinna
functions}\label{sec5}
In this section local versions of rigged Hilbert spaces are associated to a selfadjoint
relation. These rigged space are used to construct realizations for
functions in the Kac-Donoghue classes, cf. Proposition~\ref{vorige}.

\subsection{Rigged spaces}
Let $A$ be a selfadjoint relation in the Hilbert space
$\{\sH,(\cdot,\cdot)\}$. Define the inner products
$(\cdot,\cdot)_{\infty,\pm 1}$ via
\begin{equation}\label{def+}
\begin{split}
(f,g)_{\infty,+1} &= (f,g) + (|A_o|^{\frac{1}{2}}P_\infty f,|A_o|^{\frac{1}{2}}P_\infty g),  \quad f,g \in \dom |A_o|^\frac{1}{2}\oplus \mul A;\\
 (f,g)_{\infty,-1} &= ((I-P_\infty) f,g) + ((I+|A_o|)^{-1}P_\infty f,P_\infty g),  \quad f,g \in \sH,
\end{split}
\end{equation}
see \eqref{defabsval}. Then
$\sH_{+1}(A,\infty):=\{\dom|A_o|^\frac{1}{2}\oplus \mul
A,(\cdot,\cdot)_{\infty,+1}\}$ is a Hilbert space and also the
completion of $\sH$ with respect to $(\cdot,\cdot)_{\infty,-1}$,
denoted by $\sH_{-1}(A,\infty)$, is a Hilbert space; it is the dual
space of $\sH_{+1}(A,\infty)$. \emph{The rigging of $\sH$ with
respect to $A$ at $\infty$} means the space triplet
$\sH_{+1}(A,\infty) \subset \sH \subset \sH_{-1}(A,\infty)$; for the
case of operators, see \cite{Berezanskii}.

In the rigging $\sH_{+1}(A,\infty) \subset \sH \subset
\sH_{-1}(A,\infty)$ the topology is changed on the part of the space $\sH$
connected with the behavior of $A$ close to $\infty$. The
topology on the other parts of the space is not changed (though the
norm in general is).

\begin{proposition}\label{iner1}
Let $A$ be a selfadjoint relation in the Hilbert space
$\{\sH,(\cdot,\cdot)\}$ and let $\sH_{+1}(A,\infty) \subset \sH
\subset \sH_{-1}(A,\infty)$ be the rigging of $\sH$ with respect to
$A$ at $\infty$. Then $\ker (A-\xi)$, $\xi \in \dR$, and $\mul A$
are closed in $\sH_{+1}(A,\infty)$ and $\sH_{-1}(A,\infty)$.
\end{proposition}

\begin{proof}
Clearly, $\mul A$ is closed in all the topologies by definition of the rigging. If $f \in \ker (A-\xi)$, $\xi \in \dR$, then by \eqref{def+}
\[
 ||f||_{\infty,+1} = (1+|\xi|)||f|| \quad \textrm{and} \quad ||f||_{\infty,-1} = (1+|\xi|)^{-1}||f||.
\]
This shows that the statement holds.
\end{proof}

By definition,
\[ 
(f,f)_{\infty,+1}\ge (f,f) \ge (f,f)_{\infty,-1}, \quad f\in  \sH_{+1}(A,\infty).
\]
It is clear from \eqref{def+} that $\|f\|_{\infty,+1}=\|f\|$ holds
if and only if $P_\infty f\in \ker |A_o|^{\frac{1}{2}}=\ker A_o$.
Similarly, \eqref{def+} implies that $\|f\|_{\infty,-1}=\|f\|$ holds
if and only if $P_\infty f\in \ker A_o$. Hence the
decomposition $\sH=(\cdom A\ominus\ker A)\oplus\ker A \oplus \mul A$
implies corresponding decompositions for the rigged spaces $\sH_{\pm
1}(A,\infty)$:
\begin{equation}\label{Rdec01}
 \sH_{\pm 1}(A,\infty)= \sH_{\pm 1}(A_o\uphar_{\cran A_o},\infty)\oplus \ker A\oplus
 \mul A.
\end{equation}
Furthermore, all the norms $\|f\|_{\infty,+1}$, $\|f\|$, and
$\|f\|_{\infty,-1}$ coincide on $\ker A\oplus \mul A$. If the
operator part $A_o$ of $A$ is bounded, then the topologies on
$\sH_{+1}(A,\infty)$ and $\sH_{-1}(A,\infty)$ are equal to the
original topology on $\sH$, implying that the rigging
$\sH_{+1}(A,\infty) \subset \sH \subset \sH_{-1}(A,\infty)$
collapses to a single (topological) space, equipped in general with
different but equivalent norms.

Observe that $A_o$ and $A_o+\lambda I$, $\lambda\in\dC$, when
treated as operators from $\sH_{+1}(A,\infty)\ominus \mul A$ to
$\sH_{-1}(A,\infty)\ominus \mul A$, are bounded, and by continuity can be uniquely
extended to everywhere defined mappings $(A_o)_\infty$ and
$(A_o+\lambda I)_\infty=(A_o)_\infty+\lambda I$ from
$\sH_{+1}(A,\infty)\ominus \mul A$ to $\sH_{-1}(A,\infty)\ominus
\mul A$. Consequently, define
\begin{equation}\label{Rdec01b}
A_\infty = (A_o)_\infty \oplus \{\{0\}\times \mul A\}.
\end{equation}
For $\lambda\in\rho(A)$, the resolvent operator $(A -\lambda
I)^{-1}$ is bounded as a mapping from $\sH_{-1}(A,\infty)$ to
$\sH_{+1}(A,\infty)$ and, hence, it also admits a unique extension
by continuity:
\begin{equation}\label{reseq1}
 ((A-\lambda I)^{-1})_\infty
 =((A-\lambda I)_\infty)^{-1}
 =(A_\infty-\lambda I)^{-1}, \quad \lambda\in\rho(A).
\end{equation}
The resolvent identity implies that
\begin{equation}\label{reseq2}
(A_\infty-\lambda)^{-1} - (A_\infty-\mu)^{-1} = (\lambda-\mu)(A-\lambda)^{-1}(A_\infty-\mu)^{-1},
\quad \lambda,\mu \in \rho(A).
\end{equation}

Since $(I+|A_o|)^{-1}\oplus I_{\mul A}$ is an isometric operator
from $\sH_{-1}(A,\infty)$ to $\sH_{+1}(A,\infty)$, its closure,
$V_\infty:\sH_{-1}(A,\infty)\to \sH_{+1}(A,\infty)$, is everywhere
defined and unitary. This mapping is called the \emph{Riesz
operator} associated with the rigging. Clearly, the Riesz operator
decomposes with respect to \eqref{Rdec01}:
\begin{equation}\label{Vdec01}
 V_{\infty}=\clos(V_{\infty}\uphar_{\cran A_o}) \oplus I_{\ker A}\oplus I_{\mul A}.
\end{equation}
The Riesz operator can be used to introduce the duality
$(\cdot,\cdot)_{\infty}$ between $\sH_{+1}(A,\infty)$ and
$\sH_{-1}(A,\infty)$ by
\begin{equation}\label{reseq3}
 (f,g)_{\infty} = (f,V_{\infty}g)_{\infty,+1}=(V^{-1}_{\infty} f,g)_{\infty,-1}, \quad (g,f)_\infty = \overline{(f,g)_\infty},
\end{equation}
for $f \in \sH_{+1}(A,\infty)$ and $g \in \sH_{-1}(A,\infty)$.

\begin{remark}\label{inprred}
If $f \in \sH_{+1}(A,\infty)$ and $g \in \sH$, then as a consequence
of \eqref{def+} the duality reduces to the inner
product on $\sH$; i.e. $(f,g)_\infty = (f,g)$. Moreover, note that
the form $(\cdot,\cdot)_{\infty}$ is continuous on
$\sH_{+1}(A,\infty)\times\sH_{-1}(A,\infty)$; see \eqref{reseq3}.
\end{remark}

Using the duality introduce for a relation $H$ from
$\sH_{+1}(A,\infty)$ to $\sH_{-1}(A,\infty)$ the dual
relation, denoted by $H^+$, along the lines of \eqref{A*}:
\begin{equation}\label{dmap1}
 H^+=\{\, \{f,f'\}\in \sH_{+1}(A,\infty)\times\sH_{-1}(A,\infty):\,  (f,g')_\infty =(f',g)_\infty,\, \forall\{g,g'\}\in H\,\}.
\end{equation}
Then $H^+$ is a closed subspace of
$\sH_{+1}(A,\infty)\times\sH_{-1}(A,\infty)$; see
Remark~\ref{inprred}. The dual $H^+$ is connected with the Hilbert
space adjoint $H^*$ from $\sH_{-1}(A,\infty)$ to
$\sH_{+1}(A,\infty)$, see \eqref{A*}, by means of the Riesz
operator:
\begin{equation}\label{B+*}
 H^*=V_\infty H^+ V_\infty=\{\, \{f,V_\infty f'\}:\, \{V_\infty f,f'\}\in H^+
 \,\}.
\end{equation}
In particular, $H$ is a bounded operator from $\sH_{+1}(A,\infty)$

to $\sH_{-1}(A,\infty)$ if and only if its dual mapping
$H^+:\,\sH_{+1}(A,\infty)\to\sH_{-1}(A,\infty)$ is bounded and,
moreover, $\|H^+\|=\|H^*\|=\|H\|$. Finally, the relation $H$ is said
to be \textit{symmetric (with respect to the duality)} or
\textit{self-dual} if $H \subset H^+$ or $H = H^+$, respectively.

If $A$ is considered as a relation from $\sH_{+1}(A,\infty)$ to
$\sH_{-1}(A,\infty)$, then, as in the case where $A$ is an operator,
the dual relation $A^+$ coincides with $A_\infty$.

\begin{proposition}\label{connect+*cor}
Let $A$ be a selfadjoint relation in the Hilbert space
$\{\sH,(\cdot,\cdot)\}$. Then the dual relation $A^+$ (of $A$ as a
relation from $\sH_{+1}(A,\infty)$ to $\sH_{-1}(A,\infty)$)
coincides with $A_\infty$. Moreover, $A_\infty$ is self-dual, i.e.,
$(A_\infty)^+=A_\infty$.
\end{proposition}
\begin{proof}
Since $A_o$ is a bounded operator from $\sH_{+1}(A,\infty)$ to
$\sH_{-1}(A,\infty)$ and by Remark~\ref{inprred} satisfies
$(f,A_og)_{\infty}=(A_of,g)_{\infty}$ for all $f,g\in\dom A_o$,
$(A_o)_\infty$ is self-dual as a mapping from
$\sH_{+1}(A,\infty)\ominus \mul A$ to $\sH_{-1}(A,\infty)\ominus
\mul A$. Now it follows from the decompositions \eqref{Rdec01},
\eqref{Rdec01b}, and \eqref{Vdec01} that $A^+=\clos A=A_\infty$.
This implies that $(A_\infty)^+=(A^+)^+=\clos A=A_\infty$.
\end{proof}

\subsection{Local versions of rigged spaces}\label{apprig2}
The rigging of the space $\sH$ presented in the previous subsection
is connected with the behavior of $A$ in the neighborhood of
$\infty$. Here riggings determined by the behavior of $A$
in a neighborhood of a point on the real line are considered.

\begin{definition}\label{rigdef}
Let $A$ be a selfadjoint relation in the Hilbert space $\{\sH,(\cdot,\cdot)\}$ and let
$\xi \in \dR$. Define the Hilbert spaces $\sH_{\pm 1}(A,\xi)$ by
$\sH_{\pm 1}(A,\xi)=\sH_{\pm1}((A-\xi)^{-1},\infty)$. Then
$\sH_{+1}(A,\xi) \subset \sH \subset \sH_{-1}(A,\xi)$ is called
\emph{the rigging of $\sH$ with respect to $A$ at $\xi$}.
\end{definition}

Note that the decompositions in \eqref{Rdec01} take for the rigged spaces $\sH_{\pm 1}(A,\xi)$, $\xi \in \dR$,
the following form:
\begin{equation}\label{Rdec02}
 \sH_{\pm 1}(A,\xi)= \sH_{\pm 1}(((A-\xi)^{-1})_o\uphar{ \cran (A_o-\xi)},\infty)
  \oplus \mul A \oplus  \ker (A-\xi).
\end{equation}
The rigging of $\sH$ with respect to $A$ at $\xi \in \dR$ gives rise
to the rigged space closure of the selfadjoint relation
$(A-\xi)^{-1}$:
\begin{equation}\label{Rdec02b}
 ((A-\xi)^{-1})_\xi:\, \sH_{+1}(A,\xi) \to \sH_{-1}(A,\xi);
\end{equation}
in what follows it will be denoted by $(A-\xi)^{-1}_\xi$ for short.
Note that its operator part is a bounded
everywhere defined operator on $\sH_{+1}(A,\xi)\ominus \ker
(A-\xi)$. Moreover, the resolvent $((A-\xi)^{-1}-\lambda)^{-1}$
extends to a bounded operator $R_{\xi}(\lambda): \sH_{-1}(A,\xi) \to
\sH_{+1}(A,\xi)$ with $\ker R_{\xi}(\lambda)=\ker (A-\xi)$ and it
satisfies the analog of \eqref{reseq2}:
\[ 
R_{\xi}(\lambda)-R_{\xi}(\mu) = (\lambda-\mu)((A-\xi)^{-1}-\lambda)^{-1}R_{\xi}(\mu), \quad \lambda,\mu \in \rho((A-\xi)^{-1}).
\]
Furthermore, the corresponding Riesz operator $V_\xi$ from $\sH_{-1}(A,\xi)$ onto
$\sH_{+1}(A,\xi)$ gives rise to a duality:
\[ 
 (f,g)_{\xi} =(f,V_{\xi} g)_{\xi,+1} = (V_{\xi}^{-1} f,g)_{\xi,-1}, \quad f \in \sH_{+1}(A,\xi),\, g \in \sH_{-1}(A,\xi).
\]

Let $\cK_\xi = \ker (A-\xi)$, $\xi \in \dR$, and let $\cK_\infty = \mul A$, then for $\xi \in \dR$ denote by
\begin{equation}\label{defnot}
 P_\xi:\, \sH\to \sH\ominus\cK_{\xi} \quad\text{and}\quad
 P_{\xi}^\pm:\,\sH_{\pm1}(A,\xi)\to \sH_{\pm1}(A,\xi)\ominus\cK_{\xi}
\end{equation}
the orthogonal projections from the rigged space members onto their
subspaces as indicated. Similarly, define the following orthogonal
projections for $\xi \in \dR \cup\{\infty\}$:
\begin{equation}\label{defnot2}
 P_{\infty}:\,\sH\to \sH\ominus\cK_{\infty}
  \quad\text{and}\quad
 P^\pm_{\infty}:\,\sH_{\pm1}(A,\xi) \to \sH_{\pm1}(A,\xi)\ominus\cK_{\infty}.
\end{equation}
Note that, $P_\xi = P_\xi^-\upharpoonright_{ \sH}$ and $P_\xi^+=P_\xi\upharpoonright_{ \sH_{+1}(A,\xi)}$.
Using the introduced projections, the Riesz operator $V_\xi$, and
the spectral family\footnote{For a selfadjoint relation $A$ the
spectral family $\{E_t\}_{t\in \dR}$ in $\sH$ is defined as
$\{E_{o,t}\oplus 0_{\mul A}\}_{t\in \dR}$, where $\{E_{o,t}\}_{t\in
\dR}$ is the spectral family of the operator part $A_o$ of $A$.}
$\{E_t\}_{t\in \dR}$ of $A$, the inner products
$(\cdot,\cdot)_{-1,\xi}$ and $(\cdot,\cdot)_{+1,\xi}$ can be made
explicit. From \eqref{def+} and \eqref{Rdec02} one obtains for $f,g
\in \sH_{+1}(A,\xi)$, $\xi \in \dR$,
\[ 
(f,g)_{+1,\xi} = \left((I-P_\infty)f,g\right) +\left((I-P_{\xi})f,g\right)
  +\int_{\dR\setminus\{\xi\}} \left( 1 + \frac{1}{|t-\xi|}\right)d\left(E_t P_\infty f,g \right),
\]
and a similar formula for $f,g \in \sH_{-1}(A,\xi)$ is obtained via
$(f,g)_{-1,\xi} = (V_\xi f,V_\xi g)_{+1,\xi}$. Finally, for
$R_{\xi}(\lambda)$ the following extended functional calculus result
holds:
\begin{equation}\label{spectralQ}
\begin{split}
 \left(R_{\xi}(\lambda)f,g\right)_\xi  &= -\frac{1}{\lambda}\left((I-P_\infty^-)f, (I-P_\infty^-)g\right)\\
 &\hspace{0cm}+ \int_{\dR\setminus\{\xi\}} \frac{\left(1+|t-\xi|\right)^2d\left(E_t P_\infty V_\xi f, V_\xi g\right)}{(1-\lambda(t-\xi))(t-\xi)},
\end{split}
\end{equation}
where $f,g \in \sH_{-1}(A,\xi)$ and $\lambda \in \rho((A-\xi)^{-1})$, $\xi \in \dR$.

\subsection{Realization of Kac-Donoghue classes}\label{sec52}
The following proposition gives some characterizations for functions in
the class $\cN(\xi,1)$; for $\xi =\infty$ the result has been
established in \cite[Theorem 3.1]{HKS97},
\cite[Proposition~2.2]{HLS95}; for $\xi =0$ see also \cite[Theorem
7.1]{HKS98}.

\begin{proposition}\label{charclass}
Let $Q$ be a Nevanlinna function and let $A$ be a selfadjoint
extension in the Hilbert space $\{\sH,(\cdot,\cdot)\}$ such that
\eqref{Qfd} holds for some $v \in \sH$. Moreover, let
$\gamma_\lambda$ and $S$ with $n_+(S)=n_-(S)=1$ as in
\eqref{defgamwithbt} and \eqref{recoversym}, respectively. Then the
following statements are equivalent for $\xi \in \dR$:
\begin{enumerate}
\def\labelenumi {\rm (\roman{enumi})}
\item $Q \in \cN(\xi,1)$;
\item $\gamma_\lambda \in \ran|A-\xi|^\frac{1}{2}$, $\lambda \in \rho(A)$;
\item there exists $\omega \in \sH_{-1}(A,\xi)$, $\omega \notin \ker (A-\xi)$, such that
\[ S = \{ \{f,f'\} \in A: (f'- \xi f,\omega)_\xi=0 \};\]
\item there exists $\omega \in \sH_{-1}(A,\xi)$, $\omega \notin \ker (A-\xi)$, and $\eta_\xi \in \dR$ such that $Q$ has the representation
\[ 
 Q(\lambda) =  \eta_\xi +
(\lambda-\xi)(\gamma_\lambda ,\omega)_\xi, \quad \lambda \in
\rho(A),
\]
where $\eta_\xi = \lim_{\lambda \widehat{\to} \xi} Q(\lambda)$ and
\[ 
\gamma_\lambda  = \left(I+(\xi-\lambda)(A-\xi)^{-1}_\xi \right)^{-1}\omega, \quad \lambda \in \rho(A).
\]
\end{enumerate}
The function $\gamma_\lambda$ in (iv) is the $\gamma$-field for $Q$.
\end{proposition}
\begin{proof}
If $Q \in \cN(\xi,1)$, then $Q_\infty(z) := -Q(\xi + 1/z) \in
\cN(\infty,1)$; see \eqref{transsets}. Let
$\{B,\gamma^\infty_\lambda\}$ be a realization for $Q_\infty$, see
Remark~\ref{Qfunctions}. Then by Lemma~\ref{Lugerres}, and the
discussion following it, $\{A,\gamma_\lambda\}$ is a realization for
$Q$, where $A = B^{-1}+\xi$ and
\begin{equation}\label{congam}
(\lambda-\xi)\gamma_\lambda = \gamma^\infty_{(\lambda-\xi)^{-1}}.
\end{equation}
Consequently, the equivalence of (i) and (ii) follows from the fact that $Q_\infty \in \cN(\infty,1)$ if and only if $\gamma^\infty_\lambda \in \dom |B|^\frac{1}{2}$, see \cite[Proposition 2.2]{HLS95}. Furthermore, according to the characterizations for $\xi=\infty$ in \cite[Theorem~3.1]{HKS97} (see also \cite[Section~4]{HLS95}) one has the representations
$Q_\infty(\lambda)=\eta+((B_\infty-\lambda)^{-1}\omega,\omega)$ and
$\gamma_\lambda^\infty=(B_\infty-\lambda)^{-1}\omega$ for some
$\omega\in\sH_{-1}(B,\infty)=\sH_{-1}(A,\xi)$, under the additional minimality
assumption on $S$ (cf. \eqref{minimal}) or the assumption that
$(S-\xi)^{-1}$ is an operator. Combining this with the formula
\eqref{congam} gives
\[
 \gamma^\xi_\lambda=\frac{((A-\xi)^{-1}_\xi-(\lambda-\xi)^{-1})^{-1}}{\lambda-\xi}\, \omega
 =-(I-(\lambda-\xi)(A-\xi)^{-1}_\xi)^{-1}\omega,
\]
and Lemma~\ref{Lugerres} then yields
\[
 Q(\lambda)=-Q_\infty\left((\lambda-\xi)^{-1}\right)=-\eta-\left(\gamma^\infty_{(\lambda-\xi)^{-1}},\omega\right)
 =-\eta-(\lambda-\xi)(\gamma_\lambda,\omega)
\]
with $\omega\in\sH_{-1}(A,\xi)$; cf. Definition~\ref{rigdef}. It is
easy to see that in (iii) and (iv) the minimality assumption on $S$,
or the assumption that $(S-\xi)^{-1}$ is an operator, can be
replaced by the condition $\omega\not\in\mul(A-\xi)^{-1}$; in the
opposite case $S=A$ in (iii) and $Q$ is constant in (iv).
This yields all the desired conclusions.
\end{proof}

\begin{remark}\label{rem5.7}
The element $\omega$ in Proposition~\ref{charclass} is not
unique: if $\omega' \in \sH_{-1}(A,\xi)$ is such that
$\omega' - \omega \in \ker (A-\xi)$, then the Nevanlinna
functions associated to them in part (iv) are the same. This
non-uniqueness of $\omega$ can be used to express $\eta_\xi$
in terms of $\omega$ e.g. in the form $\eta_\xi=c_\xi\|(I-P_\xi^-)
\omega\|^2_{\xi,-1}$, $c_\xi=\pm 1$; see Theorem~\ref{eerstestap}.

Furthermore, by the assumptions in Proposition~\ref{charclass} $S$
has defect numbers $(1,1)$. However, all the statements (i)--(iv) in
Proposition~\ref{charclass} remain equivalent also if $S=A$: in this
case the function $Q$ in (i) is constant, $\gamma_\lambda= 0$ in
(ii), and $\omega\in\ker(A-\xi)$ in (iii) and (iv). In what follows,
such constant functions may occur.
\end{remark}

Next Proposition~\ref{charclass} is augmented by giving a structured realization
for the class $\cN(\xi,1)$, $\xi \in \dR \cup\{\infty\}$, as Weyl
functions of boundary triplets. First observe the following
lemma which is a consequence of Proposition~\ref{charclass} and \cite[Theorem 6.2]{DHS99}.

\begin{lemma}\label{realizationprepinfty}
Let $A$ be a selfadjoint relation in the Hilbert space
$\{\sH,(\cdot,\cdot)\}$ and let $\omega \in \sH_{-1}(A,\infty)$,
$\omega \notin \mul A$. Then
\[ S_\omega = \{\{f,f'\} \in A: (f,\omega)_\infty=0 \}
\]
is a closed symmetric relation with $n_+(S_\omega)=n_-(S_\omega)=1$
and its adjoint is
\[ \begin{split} S^*_\omega &= \{\{f ,f' - \omega c_f \} \in \sH\times \sH: \{f,f'\} \in A_\infty, \, c_f \in \dC \}. \end{split}\]
A boundary triplet for $S^*_\omega$ is given by $\{\dC,\Gamma_0^{\omega,\eta_\infty},\Gamma_1^{\omega,\eta_\infty}\}$, where
\[ 
\Gamma_0^{\omega,\eta_\infty} \widehat{f} = c_f \quad \textrm{and}
\quad \Gamma_1^{\omega,\eta_\infty}\widehat{f} = \eta_\infty c_f +
(f,\omega)_\infty
\]
for $\widehat{f}=\{f, f' - \omega c_f \}\in S^*_\omega$ and
$\eta_\infty \in \dR$. The corresponding $\gamma$-field and Weyl
function are given by
\[ 
 \gamma_\lambda^{\omega} = (A_\infty - \lambda)^{-1}\omega \quad \textrm{and}
 \quad M_\omega(\lambda) =  \eta_\infty +(\gamma_\lambda^\omega ,\omega)_\infty, \quad \lambda \in
\rho(A).
\]
In particular, the Weyl function $M_\omega$ belongs to $\cN(\infty,1)$ and
$\eta_\infty = \lim_{\lambda \widehat{\to} \infty} M_\omega(\lambda)$.
\end{lemma}

The Weyl function associated to the boundary triplet in
Lemma~\ref{realizationprepinfty} has, see \cite[(2.7)]{HKS97},
the integral representation
\begin{equation}\label{fullevptinfty}
M_\omega(\lambda) = \eta_\infty + \int_{\dR}\frac{(1 + |t|)^2}{t-\lambda}d\left(E_tP_\infty V_\infty\omega,V_\infty\omega\right).
\end{equation}

In Lemma~\ref{realizationprepinfty} the formulas for $S_\omega$,
$S^*_\omega$, and the boundary mappings
$\Gamma_0^{\omega,\eta_\infty}$, $\Gamma_1^{\omega,\eta_\infty}$
remain the same when $\omega \in \sH_{-1}(A,\xi)$ is replaced by
$\omega' \in \sH_{-1}(A,\xi)$ with $\omega' - \omega \in \mul A$;
cf. Remark~\ref{rem5.7}. If $\omega \in \mul A$, then $S_\omega = A =
S_\omega^*$, $\Gamma = \Gamma_0 \times \Gamma_1$
becomes a so-called boundary relation and its Weyl function is $\eta
\in \dR$, see \cite[Example~6.1]{DHMS06} or Section~\ref{secext}
below.

The next result is the analog of Lemma~\ref{realizationprepinfty}
for $\cN(\xi,1)$, $\xi\in\dR$. Recall, that
$(A-\xi)_\xi:=((A-\xi)_\xi^{-1})^{-1}$ denotes the closure of
$(A-\xi)$ in $\sH_{-1}(A,\xi)\times\sH_{+1}(A,\xi)$.

\begin{lemma}\label{realizationprepN1}
Let $A$ be a selfadjoint relation in the Hilbert space
$\{\sH,(\cdot,\cdot)\}$ and let $\omega \in \sH_{-1}(A,\xi)$, $\xi
\in \dR$, $\omega \notin \ker (A-\xi)$. Then
\[
S_\omega = \{\{f,f'\} \in A: (f'-\xi f,\omega)_\xi=0 \}
\]
is a closed symmetric relation with $n_+(S_\omega)=n_-(S_\omega)=1$
and its adjoint is
\[ \begin{split}
S^*_\omega &= \{\{f + \omega c_f,f'+ \xi f  + \xi \omega c_f \} \in \sH^2:
 \{f,f'\} \in (A-\xi)_\xi, \, c_f \in \dC \}.
 \end{split}
\]
A boundary triplet for $S^*_\omega$ is given by
$\{\dC,\Gamma_0^{\omega,\eta_\xi},\Gamma_1^{\omega,\eta_\xi}\}$,
where
\[
\Gamma_0^{\omega,\eta_\xi} \widehat{f} = c_f \quad \textrm{and} \quad \Gamma_1^{\omega,\eta_\xi}\widehat{f}
= \eta_\xi c_f + (f',\omega)_\xi
\]
for $\widehat{f}=\{f + \omega c_f,f'+ \xi f  + \xi \omega c_f \}\in
S^*_\omega$ and $\eta_\xi \in \dR$. The corresponding $\gamma$-field
and Weyl function are given by
\[ 
\gamma_\lambda^{\omega} = (I+(\xi - \lambda)(A- \xi)^{-1}_\xi )^{-1}\omega,
 \quad
M_\omega(\lambda) = \eta_\xi + (\lambda-\xi)(\gamma_\lambda^\omega,\omega)_\xi, \quad \lambda \in
\rho(A).
\]
In particular, the Weyl function $M_\omega$ belongs to $\cN(\xi,1)$ and $\eta_\xi = \lim_{\lambda
\widehat{\to} \xi} M_\omega(\lambda)$.
\end{lemma}

\begin{proof}
The formulas for $S_\omega$ and $S^*_\omega$ are obtained from
Lemma~\ref{realizationprepinfty} by means of the transformation
$H \mapsto (H-\xi)^{-1}$. Next observe that the mapping $U_\xi:\sH^2\to \sH^2$
\[
U_\xi \{f,f'\} = \{f'-\xi f, f\}, \quad f,f' \in \sH,
\]
is $J_\sH$-unitary where $J_\sH\{f,f'\} = \{-if',if\}$. Therefore the composition
$\Gamma^{\omega,\eta_\infty}\circ U_\xi$, where
$\Gamma^{\omega,\eta_\infty}=\{\Gamma^{\omega,\eta_\infty}_0,\Gamma^{\omega,\eta_\infty}_1\}$ is the boundary
triplet from Lemma~\ref{realizationprepinfty}, defines a
boundary triplet for $S^*_\omega$, cf. \cite[Theorem~2.10]{DHMS09}, it coincides with $\{\Gamma^{\omega,\eta_\xi}_0,\Gamma^{\omega,\eta_\xi}_1\}$.
By Proposition~\ref{charclass} $\gamma_\lambda^\omega$ spans the
defect space $\ker(S_\omega^*-\lambda)$. A direct calculation shows
that $\{\gamma_\lambda^\omega,\lambda \gamma_\lambda^\omega \}$ can be written as
\[  \{(\lambda-\xi)P^-_\xi(A-\xi)^{-1}_\xi \gamma_\lambda^\omega + P^-_\xi\omega,
(\lambda-\xi)\gamma_\lambda^\omega + \xi(\lambda-\xi)P^-_\xi(A-\xi)^{-1}_\xi\gamma_\lambda^\omega + \xi
P^-_\xi\omega\}.
\]
Since
$\{(I-P^-_\xi)\omega,\xi(I-P^-_\xi)\omega\}\in\ker\Gamma^{\omega,\eta_\infty}$,
applying the boundary mappings to
$\{\gamma_\lambda^\omega,\lambda \gamma_\lambda^\omega \}$ yields the stated
$\gamma$-field and Weyl function.
\end{proof}
Note that the resolvents of $S_\omega$ and its adjoint $S_\omega^*$
at the point $\xi$ are given by
\begin{equation}\label{extendev}
\begin{split}
 (S_\omega-\xi)^{-1} &= \{\{f',f\} \in (A-\xi)_\xi :\, (f',\omega)_\xi=0 \},\\
 (S_\omega-\xi)^{-*} &= \{\{f',f + \omega c_f \} \in \sH^2:
 \{f,f'\} \in (A-\xi)_\xi, \, c_f \in \dC \}.
 \end{split}
\end{equation}
These formulas are analogous to those appearing in
Lemma~\ref{realizationprepinfty} and this gives an alternative
method to derive various facts appearing e.g. in
Lemma~\ref{realizationprepinfty} from the riggings of $A$ at
$\infty$ to the riggings of $A$ at a finite point $\xi\in\dR$.

As a consequence of \eqref{spectralQ} the Weyl function associated
to the boundary triplet in Lemma~\ref{realizationprepN1} has the
integral representation
\begin{equation}\label{fullevpt}
\begin{split}
 M_\omega(\lambda) = &\,\, \eta_\xi + (\lambda-\xi)((I-P^-_\infty)\omega,(I-P^-_\infty)\omega) \\
 &+
(\lambda-\xi)\int_{\dR\setminus \{\xi\}}\frac{\left(1+|t-\xi|\right)^2d\left(E_tP_\infty V_\xi\omega,V_\xi\omega\right)}{(t-\lambda)(t-\xi)}.
\end{split}
\end{equation}

\begin{remark}\label{int2}
To see the connection to the integral representation for $M_\omega$
in \eqref{Nevfullr}, recall from \cite[Proposition~2.1]{HLS95} that
$d\sigma(t)=(t^2+1)d(E_tP_\infty \gamma_i,P_\infty \gamma_i)$. Incorporating the representation of $\gamma_\lambda$ in
Lemmas~\ref{realizationprepinfty} or \ref{realizationprepN1} at
$\lambda=i$ and using functional calculus it is easy to check that
\[ 
d\sigma(t)= \left\{ \begin{array}{rl}
  (1+|t-\xi|)^2 d\left(E_tP_\xi P_\infty V_\xi\omega,V_\xi\omega\right),
    & \xi \in \dR, \\
  (1+|t|)^2 d\left(E_tP_\infty V_\infty\omega,V_\infty\omega\right),
    & \xi = \infty. \end{array} \right.
\]
\end{remark}
Combining Lemmas~\ref{realizationprepinfty} and
\ref{realizationprepN1} with Proposition~\ref{charclass} the
following realization result for $\cN(\xi,1)$ is obtained. Note that
a realization for the case $\xi = \infty$ was given in
\cite[Theorem~4.4 $\&$ 6.2]{DHS99}; in fact, there also the
Pontryagin space case was allowed.

\begin{theorem}\label{realizationN1}
For each nonconstant $Q \in \cN(\xi,1)$, $\xi \in \dR \cup\{\infty\}$, there exist a closed symmetric
relation $S$ in a Hilbert space $\{\sH,(\cdot,\cdot)\}$ and a
boundary triplet $\{\dC,\Gamma_0,\Gamma_1\}$ for
$S^*$ such that $Q$ is its associated Weyl function.

In fact, there exist a selfadjoint relation $A$ in $\{\sH,(\cdot,\cdot)\}$ and an element $\omega \in \sH_{-1}(A,\xi)$ such
that the symmetry and boundary triplet can be taken to be $S_\omega$ and $\{\dC,\Gamma_0^{\omega,\eta_\xi},\Gamma_1^{\omega,\eta_\xi}\}$ as in Lemma~\ref{realizationprepinfty} or \ref{realizationprepN1}, respectively.
Here $\eta_\xi = \lim_{\lambda \widehat{\to} \xi} Q(\lambda)$.
\end{theorem}

\subsection{Realizations for intersections of Kac-Donoghue classes}\label{secext}
Let $Q$ be a Nevanlinna function from the Kac-Donoghue class
$\cN(\xi,1)$. By Theorem~\ref{realizationN1} $Q$ can be realized
with a selfadjoint relation $A$ in a Hilbert space
$\{\sH,(\cdot,\cdot)\}$ and an element $\omega \in \sH_{-1}(A,\xi)$.
If $Q$ is also contained in another Kac-Donoghue class, say
$\cN(\xi')$, then the following results shows how this is reflected
by $\omega$.

\begin{lemma}\label{prepnod}
Let $Q \in \cN(\xi,1)$, $\xi \in \dR \cup\{\infty\}$, and let $\omega
\in \sH_{-1}(A,\xi)$ for a selfadjoint relation $A$ in the Hilbert
space $\{\sH,(\cdot,\cdot)\}$ be such that $Q$ is realized via
Lemma~\ref{realizationprepinfty} or \ref{realizationprepN1}. Then $Q
\in \cN(\xi,1) \cap \cN(\xi',1)$, $\xi \neq \xi' \in
\dR \cup\{\infty\}$, if and only if $(I-P^{-}_{\xi'}) \omega = 0$ and
$V_\xi \omega \in \sH_{+1}(A,\xi) \cap \sH_{+1}(A,\xi')$.
\end{lemma}
\begin{proof}
The assumption $Q \in \cN(\xi,1)$ implies the integral
representation \eqref{fullevpt} for $Q$. On the other hand, by
Proposition~\ref{charclassspectral} $Q\in\cN(\xi',1)$ if and only if
the measure $d\sigma(t)$ satisfies the integrability condition in
Proposition~\ref{charclassspectral}~(ii) with $\xi'$, where
$\beta=0$ if $\xi'=\infty$. In view of Remark~\ref{int2} this means
that $P_\infty V_\xi\omega\in\ran |A-\xi'|^\frac{1}{2}$ if
$\xi'\in\dR$, and $P_\infty V_\xi\omega\in\dom |A|^\frac{1}{2}$ and
$(I-P^-_\infty)\omega=0$ (see \eqref{fullevpt}) if $\xi'=\infty$.
Due to \eqref{Rdec02} this is equivalent to
$V_\xi\omega\in\sH_{+1}(A,\xi')\ominus\ker(A-\xi')$ if $\xi'\in\dR$,
and to $V_\xi\omega\in\sH_{+1}(A,\infty)\ominus\mul A$ if
$\xi'=\infty$. This implies the statement of the lemma.
\end{proof}

Let $Q$ be a Nevanlinna function which is in the intersection of two
Kac-Donoghue classes $\cN(\xi,1)$ and $\cN(\xi',1)$, and let
$\{A,\gamma_\lambda\}$ be a realization for $Q$, see
Remark~\ref{Qfunctions}. Then Proposition~\ref{charclass} implies
that $\gamma_\lambda \in \sH_{+1}(A,\xi) \cap \sH_{+1}(A,\xi')$.
Hence there exists a $\omega \in \sH_{-1}(A,\xi)$ and a $\omega' \in
\sH_{-1}(A,\xi')$ such that $Q$ is realized by the model associated
to $A$ and $\omega$, and $A$ and $\omega'$ as in
Lemma~\ref{realizationprepinfty} or \ref{realizationprepN1}.

To connect these two realization of $Q$ a connection between $\omega$ and $\omega'$ is needed. Therefore observe that the $\gamma$-field, which is unique for a Nevanlinna function, can be expressed by means of $\omega$ and $\omega'$:
\[ \gamma_\lambda  = \left(I+(\xi-\lambda)(A-\xi)^{-1}_\xi \right)^{-1}\omega = \left(I+(\xi'-\lambda)(A'-\xi')^{-1}_{\xi'} \right)^{-1}\omega',\]
if $\xi,\xi' \in \dR$, see Lemma~\ref{realizationprepN1}. This
discussion motivates the following lemma in which the notation
$\gamma_{\lambda,\xi} (A)$, $\lambda \in \rho(A)$, is used for the
mapping on $\sH_{-1}(A,\xi)$ defined as
\begin{equation}\label{gAmap}
 \gamma_{\lambda,\xi} (A)f
 = \left\{ \begin{array}{rl}(I+(\xi-\lambda)(A-\xi)_\xi^{-1})^{-1}f,& \xi \in \dR;\\
 (A_\infty-\lambda)^{-1}f,& \xi = \infty, \end{array} \right.
\end{equation}
i.e. $\gamma_{\lambda,\xi}f$ is the $\gamma$-field of a function in $\cN(\xi,1)$. Note that $\ker\gamma_{\lambda,\xi}(A)=\ker (A-\xi)$.
\begin{lemma}\label{corgamma}
Let $\gamma_{\lambda,\xi} (A)$ and $\gamma_{\lambda,\xi'}(A)$,
$\xi,\xi'\in\dR \cup\{\infty\}$, $\xi\neq\xi'$, be given by
\eqref{gAmap} and let
\begin{equation*}\label{defrhoab}
\rho_{\xi,\xi'}(A)= P^-_{\xi'}\left(\gamma_{\lambda,\xi'}(A)\right)^{-1}\gamma_{\lambda,\xi}(A),
\quad \lambda \in \rho(A).
\end{equation*}
Then $\rho_{\xi,\xi'}(A)$ defines a bounded operator from
$\sH_{-1}(A,\xi)\cap P^-_{\xi'}V_\xi^{-1}\sH_{+1}(A,\xi')$ onto
$P^-_\xi V_{\xi'}^{-1}\sH_{+1}(A,\xi) \cap
P^-_{\xi'}\sH_{-1}(A,\xi')$, which does not depend on
$\lambda\in\rho(A)$. Moreover, $\rho_{\xi,\xi'}$ is a bounded
extension of the mapping $P^-_{\xi'}(A-\xi')^{-1}_{\xi'}(A-\xi)_\xi$, where $(A-\xi)$, $(A-\xi')$ is to be
interpreted as $I$ if $\xi=\infty$ or $\xi'=\infty$, respectively.
\end{lemma}
\begin{proof}
From Section~\ref{apprig2} and Proposition~\ref{charclass} it is
known that $\gamma_{\lambda,\xi}(A)$ with $\lambda \in \rho(A)$ is a
bounded operator, which maps $\sH_{-1}(A,\xi)$ onto $P_{\xi}
\sH_{+1}(A,\xi)$ with $\ker\gamma_{\lambda,\xi}(A)=\ker P^-_\xi$.
Furthermore, it follows from Lemma~\ref{prepnod} that
$\gamma_{\lambda,\xi}(A)$ maps $\sH_{-1}(A,\xi)\cap
P^-_{\xi'}V_{\xi}^{-1}\sH_{+1}(A,\xi')$ onto
$P_{\xi}^+\sH_{+1}(A,\xi)\cap P_{\xi'}^+\sH_{+1}(A,\xi')$. With
$\lambda\in\rho(A)$ this implies that $\rho_{\xi,\xi'}(A)$ is a bounded operator from $\sH_{-1}(A,\xi)\cap
P^-_{\xi'}V_\xi^{-1}\sH_{+1}(A,\xi')$ onto $P^-_\xi V_{\xi'}^{-1}\sH_{+1}(A,\xi) \cap P^-_{\xi'}\sH_{-1}(A,\xi')$.
Finally, by means of functional calculus it is straightforward to
check that $\rho_{\xi,\xi'}(A)$ is independent of
$\lambda\in\rho(A)$ and extends $P^-_{\xi'}(A-\xi')^{-1}_{\xi'}(A-\xi)_\xi$.
\end{proof}

\begin{proposition}\label{inductionstepprop}
Let $Q \in \cN(\xi,1) \cap \cN(\xi',1)$, $\xi,\xi' \in
\dR \cup\{\infty\}$ with $\xi \neq \xi'$ and let $A$ be a selfadjoint
relation in a Hilbert space $\{\sH,(\cdot,\cdot)\}$ associated to a
realization of $Q$. Moreover, let $\omega \in \sH_{-1}(A,\xi)$ and
$\eta_\xi$ be such that the boundary triplet
$\{\dC,\Gamma_0^{\omega,\eta_\xi},\Gamma_1^{\omega,\eta_\xi}\}$
associated to $S_\omega^*$ as in Lemma~\ref{realizationprepinfty} or
\ref{realizationprepN1} realizes $Q$.

Then with $\omega' := \rho_{\xi,\xi'}(A)\omega \in
\sH_{-1}(A,\xi')$ and $\eta_{\xi'}:=\lim_{\lambda \wh \to
\xi'}Q(\lambda)$ also the boundary triplet $\{\dC,\Gamma_0^{
\omega',\eta_\xi'},\Gamma_1^{\omega',\eta_\xi'}\}$
associated to $S_{\omega'}^*$ as in
Lemma~\ref{realizationprepinfty} or \ref{realizationprepN1} realizes
$Q$. In particular, $\Gamma_j^{ \omega, \eta_\xi}\wh f =
\Gamma_j^{\omega',\eta_\xi'}\wh f$ for $\wh f \in S_\omega^* =
S_{\omega'}^*$ and $j=1,2$.
\end{proposition}

\begin{proof}
Lemma~\ref{prepnod} shows that $\omega\in\sH_{-1}(A,\xi)\cap
P^-_{\xi'}V_\xi^{-1}\sH_{+1}(A,\xi')$. Thus, by Lemma~\ref{corgamma}
$\omega' \in P^-_{\xi'}\sH_{-1}(A,\xi')$. Moreover, by definition of
$\rho_{\xi,\xi'}(A)$ the $\gamma$-fields $\gamma_\lambda^{\omega}$
and $\gamma_\lambda^{\omega'}$ associated to
$\{\dC,\Gamma_0^{\omega,\eta_\xi},\Gamma_1^{ \omega,\eta_\xi}\}$ and
$\{\dC,\Gamma_0^{\omega',
\eta_\xi'},\Gamma_1^{\omega',\eta_\xi'}\}$, respectively, satisfy
$\gamma_\lambda^{\omega}=\gamma_\lambda^{\omega'}$ for all
$\lambda\in\rho(A)$. This fact together with the observation that $A
\subset S_\omega^*$ and $A \subset S_{\omega'}^*$ shows that
$S_\omega^*=S_{\omega'}^*$, see \eqref{vonneu2}.

Furthermore, if $M(\lambda)$ is the Weyl function associated to
$\{\dC,\Gamma_0^{\omega',\eta_\xi'},\Gamma_1^{\omega',
\eta_\xi'}\}$, then $M(\lambda)-Q(\lambda) \in \dR$, $\lambda \in
\cmr$, because both functions have the same $\gamma$-field. Since
both function have by definition the same non-tangential limit
$n_{\xi'}$ at $\xi'$ they must coincide and, hence, so do the
boundary mappings.
\end{proof}

Proposition~\ref{inductionstepprop} contains, in operator language,
the information about how to rewrite the integral representation of
the form \eqref{fullevptinfty} or \eqref{fullevpt} for a function $Q
\in \cN(\xi,1) \cap \cN(\xi',1)$ from a point $\xi\in\dR
\cup\{\infty\}$ to the point $\xi' \in \dR \cup\{\infty\}$.

\begin{remark}\label{rem5.16}
If $\omega,\wt \omega \in \sH_{-1}(A,\xi)$
are such that the boundary triplets associated to them (and
$\eta_{\xi}$) via Lemma~\ref{realizationprepinfty} or
\ref{realizationprepN1} have the same Weyl function, then $\wt \omega -
\omega \in \ker (A-\xi)$. Namely, in this case also the
corresponding $\gamma$-fields $\gamma_{\lambda,\xi}(A)\omega$ and
$\gamma_{\lambda,\xi}(A)\wt \omega$ coincide, and hence
$\wt \omega-\omega\in\ker\gamma_{\lambda,\xi}(A)=\ker (A-\xi)$; see
\eqref{gAmap}.
\end{remark}

\subsection{Boundary triplets in rigged spaces}
The relations $S_\omega$ and $S_\omega^*$ in Lemma~\ref{realizationprepinfty} can be extended to relations from
$\sH_{+1}(A,\infty)$ to $\sH_{-1}(A,\infty)$ by the formulas:
\begin{equation}\label{rigSS*}
\begin{split}
\wt S_\omega &= \{\{f,f'\} \in A_\infty: (f,\omega)_\infty=0 \},\\
\widetilde{S^*_\omega} &= \{\{f ,f' - \omega c_f \} \in
\sH_{+1}(A,\xi)\times \sH_{-1}(A,\infty): \{f,f'\} \in A_\infty, \,
c_f \in \dC \},
\end{split}
\end{equation}
where $\wt S_\omega=A_\infty=\widetilde{S^*_\omega}$ if $\omega\in\mul A$. It follows from Proposition~\ref{connect+*cor} that $\wt S_\omega^+ = \widetilde{S^*_\omega}$, since $A_\infty^+=A_\infty$; cf. \eqref{dmap1},~\eqref{B+*}. Also the boundary mappings in Lemma~\ref{realizationprepinfty} can be extended by the same formulas to mappings on $\wt S_\omega^+$. For this reason the definition of boundary triplets as given in
\eqref{defbndscalles} is extended to the situation of riggings of
Hilbert spaces. Here this definition is stated just in the case of
the rigging $\sH_{+1}(A,\infty)\times\sH_{-1}(A,\infty)$.

\begin{definition}\label{DualBT}
Let $S$ be a symmetric relation in
$\sH_{+1}(A,\infty)\times\sH_{-1}(A,\infty)$ with the dual relation
$S^+$. Then $\{\cH,\wt \Gamma_0,\wt \Gamma_1\}$ is a boundary triplet for
$S^+$ if
\begin{enumerate}
\def\labelenumi{\rm (\roman{enumi})}
\item for all $\wh f=\{f,f'\}, \wh g=\{g,g'\}\in S^+$ the following Green's identity holds:
\[
(f',g)_\infty-(f,g')_\infty =
 (\wt \Gamma_1\wh f, \wt \Gamma_0\wh g)-(\wt \Gamma_0\wh f,\wt \Gamma_1\wh g);
\]
\item the mapping $\wt\Gamma: S^+ \to \cH\times\cH$,
$\wh f \mapsto \{\wt \Gamma_0 \wh f, \wt \Gamma_1 \wh f\}$, is
surjective.
\end{enumerate}
\end{definition}

As a consequence of the above definition
$\wt \Gamma(\wt A_\Theta^+)=(\wt \Gamma(\wt A_\Theta))^*$; see \eqref{dmap1}. Hence
similar to the case of ordinary boundary triplets the formula
\begin{equation}\label{rigAtau}
 \wt A_\Theta=\bigl\{\wh f\in \wt S_\omega^+:\, \wt\Gamma\wh f\in\Theta\,\}
 =\ker(\wt \Gamma_1-\Theta\wt \Gamma_0)
\end{equation}
establishes a bijective correspondence between all self-dual
extensions $\wt A_\Theta$ of $S$ in
$\sH_{+1}(A,\infty)\times\sH_{-1}(A,\infty)$ and all selfadjoint
relations $\Theta$ on $\cH$.

It is also possible to extend the notion of boundary relation
introduced in \cite{DHMS06} to the present setting;  cf.
\cite[Definition~3.1]{DHMS06}. With a boundary triplet (or relation)
for $S^+$ one can associate the $\gamma$-field and the Weyl function
as in Definition~\ref{defbndscalles} or as in
\cite[Definition~3.3,~Section~4.2]{DHMS06}. The only difference in
these definitions is the change of the state space: $\sH\times\sH$
is replaced by $\sH_{+1}(A,\infty)\times\sH_{-1}(A,\infty)$. This
offers a different, and also simpler, realization for functions
belonging to the classes  $\cN(\xi,1)$, $\xi\in\dR \cup\{\infty\}$,
since in rigged spaces the self-dual extensions $\wt A_\Theta$ of
$S$ appear as perturbations. This is made explicit in the next
proposition for the class $\cN(\infty,1)$.

\begin{proposition}
Let $A$ be a selfadjoint relation in the Hilbert space
$\{\sH,(\cdot,\cdot)\}$, let $\omega \in \sH_{-1}(A,\infty)$, and
let $\wt S_\omega$ be given by \eqref{rigSS*}. Then
$\{\dC,\wt\Gamma_0^{\omega,\eta_\infty},\wt\Gamma_1^{\omega,\eta_\infty}\}$
with
\begin{equation}\label{rigBTinfty}
 \wt\Gamma_0^{\omega,\eta_\infty} \widehat{f} = c_f \quad \textrm{and} \quad
 \wt\Gamma_1^{\omega,\eta_\infty}\widehat{f} = \eta_\infty c_f +
(f,\omega)_\infty,
\end{equation}
$\wh f=\{f, f' - \omega c_f \}, \wh g=\{g, g' - \omega c_g \}\in \wt
S_\omega^+$, is a boundary triplet for $\wt S_\omega^+$, if $\omega
\not\in \mul A$, and a boundary relation for $\wt S_\omega^+$ with
$\mul\wt\Gamma^{\omega,\eta_\infty}=
\ran\wt\Gamma^{\omega,\eta_\infty}= \{\{c,\eta_\infty
c\}:\,c\in\dC\}$, if $\omega\in\mul A$. Moreover, \eqref{rigAtau}
with $\omega\not\in\mul A$ gives a bijective correspondence between
$\tau\in{\dR}\cup\{\infty\}$ and the self-dual extensions $\wt
A_\tau$ of $\wt S_\omega$ of the form:
\begin{equation}\label{rigAtau2}
\begin{array}{ll}
& \wt A_\tau f=
 \left\{\left\{f,f'+\dfrac{(f,\omega)_\infty}{\eta_\infty-\tau}f\right\}: \{f,f'\}\in A_\infty\right\},
 \quad \tau\neq\eta_\infty; \\[3mm]
& \wt A_{\eta_\infty}=\wt S_\omega\wh{+} \,{\rm span\,}(\{0\}\times\{\omega\}).
\end{array}
\end{equation}
The corresponding $\gamma$-field and Weyl function coincide with the
$\gamma$-field and Weyl function given in
Lemma~\ref{realizationprepinfty}.
\end{proposition}
\begin{proof}
A straightforward calculation using $(A_\infty)^+=A_\infty$ (see
Proposition~\ref{connect+*cor}) shows that the Green's identity (i)
in Definition~\ref{DualBT} with the mappings in \eqref{rigBTinfty}
holds for all $\wh f, \wh g\in\wt S_\omega^+$:
\[ 
(f' - \omega c_f,g)_\infty-(f,g' - \omega c_g)_\infty =
 (\wt\Gamma_1^{\omega,\eta_\infty}\wh f, \wt\Gamma_0^{\omega,\eta_\infty}\wh g)
  -(\wt\Gamma_0^{\omega,\eta_\infty}\wh f,\wt\Gamma_1^{\omega,\eta_\infty}\wh g).
\]
Surjectivity with $\omega\not\in\mul A$ is clear, since by
Lemma~\ref{realizationprepinfty} $\Gamma^{\omega,\eta_\infty}$ is
surjective on $S_\omega^*$ and $S_\omega^*\subset\wt S_\omega^+$.
Since $(f,\omega)_\infty=0$ when $\omega\in\mul A$, the
stated formula for $\mul\wt\Gamma^{\omega,\eta_\infty}$, as well as
for $\ran\wt\Gamma^{\omega,\eta_\infty}$, is obtained from
\eqref{rigBTinfty}. In particular, if $\omega\not\in\mul A$ then the
mapping $\wt\Gamma^{\omega,\eta_\infty}:\wt S_\omega^+ \to
\dC\times\dC$ is bounded, and if $\omega\in\mul A$, then $\wt
S_\omega=A_\infty=\wt S_\omega^+$ and
$\wt\Gamma^{\omega,\eta_\infty}$ is a boundary relation of the form
discussed in \cite[Example~6.1]{DHMS06}.

The formulas for $\wt A_\tau$ in \eqref{rigAtau2} are obtained by incorporating the definitions
\eqref{rigBTinfty} into \eqref{rigAtau}. Finally, by taking closure
of $S_\omega^*=A\wh{+}\{\gamma_\lambda,\lambda\gamma_\lambda\}$,
$\lambda\in\rho(A)$, in $\sH_{+1}(A,\infty)\times\sH_{-1}(A,\infty)$
gives $\wt
S_\omega^+=A_\infty\wh{+}\{\gamma_\lambda,\lambda\gamma_\lambda\}$.
Hence, by applying the usual definitions of the $\gamma$-field and
Weyl function (see Definition~\ref{defbndscalles}) with the mappings
$\wt\Gamma^{\omega,\eta_\infty}_0$ and
$\wt\Gamma^{\omega,\eta_\infty}_1$ gives the formulas for $\wt\gamma_\lambda^\omega$ and $\wt M_\omega$ as in Lemma~\ref{realizationprepinfty}.
\end{proof}

Note that $\wt A_\tau$ is the closure (in
$\sH_{+1}(A,\infty)\times\sH_{-1}(A,\infty)$) of the corresponding
selfadjoint extensions $A_\tau$ of $S_\omega$, in particular, $\wt
A_\infty=A_\infty$. Equivalently, $\wt A_\tau \cap \sH^2=A_\tau$ and
similarly $\wt S_\omega \cap \sH^2 = S_\omega$ and $\wt S_\omega^+
\cap \sH^2 = S_\omega^*$.

\section{Realization of $\cN_\kappa^{\wt \kappa}(r)$-functions}\label{sec6}
Let $Q$ be a generalized Nevanlinna function in the class
$\cN_\kappa^{\wt \kappa}(r)$. In this section the realizations for
the functions $Q$ and $rQ$ as Weyl function are studied and, in
particular, and the connection between their realizations is
established. Recall that if $\phi Q_0$ and $\wt \phi \wt Q_0$ are
the canonical factorizations of $Q$ and $rQ$, then $Q_0 \in
\cN_0^0(r \phi/\wt \phi)$; see \eqref{cnkkcn00}. Therefore, first
the realizations and their connection is established for the class
$\cN_0^0(r)$. From these results the general case can then be
derived.

To establish the connection between the realizations of $Q_0$ and
$rQ_0$ in the case that $Q_0 \in \cN_0^0(r)$ some further facts are
needed. An important observation here is the result stated in
Proposition~\ref{vorige}: if $Q_0 \in \cN_0^0(r)$ then $Q_0 \in
\bigcap_{b_i} \cN(b_i,1)$, while $rQ_0 \in \bigcap_{a_i}
\cN(a_i,1)$, where $b_i \in \dR \cup\{\infty\}$ and $a_i \in \dR
\cup\{\infty\}$ are the poles and zeros of $r$. Hence the
realization of the subclasses $\cN(\xi,1)$, $\xi \in \dR
\cup\{\infty\}$, given in Section~\ref{sec52}, will naturally appear
here.

\subsection{Rational transformations of selfadjoint relations}
For $a,b \in \dR \cup\{\infty\}$, $a \neq b$, and $\gamma \in
\dR\setminus \{0\}$, let $r_{a,b}^\gamma$ be the symmetric rational
function given by
\begin{equation}\label{defrabgamma}
r_{a,b}^\gamma(\lambda)=\gamma \frac{\lambda-a}{\lambda-b},
\end{equation}
where, for notational convenience, $\lambda-\infty$ should be
interpreted as being $1$. With the notations introduced in
\eqref{defnot} and \eqref{defnot2} associate with the selfadjoint
relation $A$ the operator $r_{a,b}^\gamma(A): \sH \to \sH$ by
setting
\begin{equation}\label{defRa}
r_{a,b}^\gamma(A)f = P_{\cK_a + \cK_b}f + \left\{ \begin{array}{rl}  \gamma(I + P_b (b-a)(A-b)^{-1})P_bf,& a,b \in \dR; \\
\gamma P_b (A-a) P_bf,& a \in \dR, b=\infty;\\
\gamma P_b(A-b)^{-1}P_b f,& a =\infty, b \in \dR.
\end{array} \right.
\end{equation}
Furthermore, define its extension
$\wt r_{a,b}^\gamma(A):\sH_{+1}(A,b) \to \sH_{-1}(A,b)$ by
\begin{equation}\label{rhoRa}
\wt r_{a,b}^\gamma(A)f = P_{\cK_a + \cK_b}f + \left\{ \begin{array}{rl}  \gamma(I + P_b^- (b-a)(A-b)^{-1}_b)P_bf,& a,b \in \dR; \\
\gamma P_b^- (A_\infty-a) P_b f,& a \in \dR, b=\infty;\\
\gamma P_b^-(A-b)^{-1}_bP_b f,& a =\infty, b \in \dR,
\end{array} \right.
\end{equation}
where $A_\infty$ and $(A-b)^{-1}_b$ are as defined in
\eqref{Rdec01b} and \eqref{Rdec02b}. Observe, that
\begin{equation}\label{projprop1}
\begin{array}{ll}
(I-P_a)r_{a,b}^\gamma(A) = I-P_a,& \quad (I-P_b)r_{a,b}^\gamma(A) = I-P_b,\\
(I-P_a^-)\wt r_{a,b}^\gamma(A) = I-P_a^-,& \quad (I-P_b^-)\wt r_{a,b}^\gamma(A) = I-P_b^-.
\end{array}
\end{equation}
If $r_{a,b}^\gamma(A)$ is a nonnegative operator, then the square
root of $r_{a,b}^\gamma(A)$, denoted by
$(r_{a,b}^\gamma)^\frac{1}{2}(A)$, is an everywhere defined bounded
operator from $\sH_{+1}(A,b)$ onto $\sH_{+1}(A,a)$. Moreover, in
this case $(I+ \wt r_{a,b}^\gamma(A))$ is also an everywhere defined
bounded operator from $\sH_{+1}(A,b)$ onto $\sH_{-1}(A,b)$. The next
lemma identifies an isomorphism between the rigged spaces
$\sH_{-1}(A,b)$ and $\sH_{-1}(A,a)$.

\begin{lemma}\label{trsqlemma}
Assume that the operator $r_{a,b}^\gamma(A)$ defined in
\eqref{defRa} is nonnegative. Then
$(\rho_{b,a}^\gamma)^{\frac{1}{2}}(A):\sH_{-1}(A,b) \to
\sH_{-1}(A,a)$, defined by
\begin{equation}\label{trsqroot}
(\rho_{b,a}^\gamma)^{\frac{1}{2}}(A)f = |\gamma| \left(I+\wt r_{b,a}^{1/\gamma}(A)\right)
 (r_{a,b}^\gamma)^\frac{1}{2}(A)\left(I+\wt r_{a,b}^\gamma(A)\right)^{-1}f,
\end{equation}
the dual mapping of $(r_{b,a}^{\gamma})^\frac{1}{2}(A)$, which is an
isomorphism. Furthermore, the following formula holds
\begin{equation}\label{connectselfadext}
P_a^-(A-a)_a^{-1}P_aP_b =
 \sgn(\gamma)(\rho_{b,a}^\gamma)^\frac{1}{2}(A)P_b^-(A-b)_b^{-1}
  P_bP_a\left((r_{a,b}^{\gamma})^\frac{1}{2}(A)\right)^{-1}.
\end{equation}
\end{lemma}
\begin{proof}
Let the operator $r_{a,b}^\gamma(A)$ be nonnegative. Note in
particular that, if $a,b\in\dR$ and $\mul A\neq \{0\}$ then
necessarily $\gamma>0$ (and the whole spectrum of $A$ lies outside
the open interval with endpoints $a$ and $b$); see \eqref{defRa}.
Since $(r_{a,b}^\gamma(A))^{-1}=r_{b,a}^{1/\gamma}(A)$, the
operators $r_{b,a}^{\gamma}(A)=\gamma^2 r_{b,a}^{1/\gamma}(A)$ and
$(r_{b,a}^\gamma)^{\frac{1}{2}}(A)$ are also nonnegative.

Next observe that for all $f\in\sH_{+1}(A,a)$ and $g\in\ran(A-a)$
one has
\[
 \left(g,\wt r_{b,a}^{1/\gamma}(A)f\right)_a
=\left(r_{b,a}^{1/\gamma}(A)g,f\right),
\]
since $\wt r_{b,a}^{1/\gamma}(A)$ is obtained as a continuous
extension of the selfadjoint operator $r_{b,a}^{1/\gamma}(A)$ and
the above identity clearly holds for all $f,g\in\ran(A-a)$; see
Remark~\ref{inprred}. Now for all $h\in\sH$ and $g\in\ran(A-a)$ one
obtains
\[
\begin{split}
 \left(g,(\rho_{b,a}^\gamma)^{\frac{1}{2}}(A)h\right)_a
 &=|\gamma| \left(\left(I+r_{b,a}^{1/\gamma}(A)\right)g,
  (r_{a,b}^\gamma)^\frac{1}{2}(A)\left(I+r_{a,b}^\gamma(A)\right)^{-1}h\right) \\
 &=\left((r_{b,a}^\gamma)^\frac{1}{2}(A)g,h\right)
  =\left((r_{b,a}^\gamma)^\frac{1}{2}(A)g,h\right)_b,
\end{split}
\]
where the second equality holds by functional calculus. The first
assertion of the lemma now follows from the above identity by
boundedness of the mappings
$(r_{b,a}^\gamma)^\frac{1}{2}(A):\sH_{+1}(A,a)\to\sH_{+1}(A,b)$ and
$(\rho_{b,a}^\gamma)^{\frac{1}{2}}(A):\sH_{-1}(A,b) \to
\sH_{-1}(A,a)$; see also Remark~\ref{inprred}.

The identity \eqref{connectselfadext} follows from the fact that the
left- and righthand side are continuous operators on $\sH_{+1}(A,a)$
which coincide on the dense subset $\ran (A-a)\ominus \ker (A-b)$ of
$\sH_{+1}(A,a)\ominus (\ker (A-a)+ \ker (A-b))$, cf.
Proposition~\ref{iner1}; see also \eqref{projprop1}.
\end{proof}

\subsection{Realizations in the case of a rational function of degree one}
Let the rational function $r^\gamma_{a,b}$ be given by
\eqref{defrabgamma} and assume that $Q \in \cN_0^0(r^\gamma_{a,b})$.
Then Theorem~\ref{productinNg} implies that $Q$ is holomorphic at
$x\in\dR$ if $r(x)<0$ and, moreover, Proposition~\ref{vorige} shows
that $Q \in \cN(b,1)$. Hence, there exist a selfadjoint relation $A$
in a Hilbert space $\{\sH,(\cdot,\cdot)\}$ and an element $\omega
\in \sH_{-1}(A,b)$ such that $Q$ can be realized as the Weyl
function associated to $A$ and $\omega$; see
Theorem~\ref{realizationN1}.

If this realization is minimal, then $\rho(A)=\rho(Q)$ and in this
case the transform $r^\gamma_{a,b}(A)$ is automatically a
nonnegative operator. As noted in Remark~\ref{rem5.7} it is always
possible to add to the selfadjoint relation $A$ a nontrivial point
spectrum at $b$ and replace $\omega$ by $\omega'=\omega+e_b$, where
$e_b\in\ker (A-b)$, to produce another (necessarily non-minimal)
realization for $Q$. The transform $r_{a,b}^\gamma(A)$ of the
extended $A$ still remains a nonnegative operator; cf.
\eqref{projprop1}. The addition of a nontrivial eigenspace to $A$ at
$b$ simplifies the expressions connecting the realizations of $Q$
and $r^\gamma_{a,b}Q$; see Remark~\ref{rem6.3trsbymeansofrhor}
below.

\begin{theorem}\label{eerstestap}
Let $r^\gamma_{a,b}$ be given by \eqref{defrabgamma} and assume that
$Q_{b} \in \cN_0^0(r^\gamma_{a,b})$. Let $A$ be a selfadjoint
relation in the Hilbert space $\{\sH,(\cdot,\cdot)\}$ such that
$b\in\sigma_p(A)$\footnote{For a selfadjoint relation $A$ with operator part $A_0$ (see \eqref{oppartred}) $\sigma(A)=\sigma(A_0)$ and $\sigma_p(A)=\sigma(A_0)$ if $\mul A = \{0\}$, and $\sigma(A)=\sigma(A_0) \cup \{\infty\}$ and $\sigma_p(A)=\sigma(A_0)\cup\{\infty\}$ if $\mul A \neq \{0\}$.} and $r^\gamma_{a,b}(A)\ge 0$. Moreover, let
$\omega_b \in \sH_{-1}(A,b)$ be such that the model in
Lemma~\ref{realizationprepinfty} or \ref{realizationprepN1} realizes
$Q_{b}$ with
\begin{equation}\label{pos2}
Q_{b}(b) = \left\{ \begin{array}{rl}
\gamma(a-b)((I-P_b^-)\omega_b,(I-P_b^-)\omega_b),& a,b \in \dR;\\
\gamma((I-P_b^-)\omega_b,(I-P_b^-)\omega_b) ,& a \in \dR, b =\infty;\\
-\gamma ((I-P_b^-)\omega_b,(I-P_b^-)\omega_b),& a=\infty,b \in \dR.
\end{array} \right.
\end{equation}
Then $\omega_{a} := (\rho_{b,a}^\gamma)^\frac{1}{2}(A)\omega_b \in
\sH_{-1}(A,a)$ and the Weyl function $Q_{a}$ associated
with $A$, $\omega_{a}$, and
\[ \eta_{a} = \left\{ \begin{array}{rl}
\frac{b-a}{\gamma}((I-P_a^-)\omega_{a},(I-P_a^-)\omega_{a}),& a,b \in \dR;\\
-\frac{1}{\gamma}((I-P_a^-)\omega_{a},(I-P_a^-)\omega_{a}) ,& a \in \dR, b =\infty;\\
\frac{1}{\gamma}((I-P_a^-)\omega_{a},(I-P_a^-)\omega_{a}),& a=\infty,b \in \dR,
\end{array} \right.\]
via Lemma~\ref{realizationprepinfty} or \ref{realizationprepN1}
coincides with  $r_{a,b}^\gamma Q_{b}$. In particular,
$Q_{a}=r_{a,b}^\gamma Q_{b}\in \cN(a,1)$.
\end{theorem}

\begin{proof}
Only the case that $a,b \in \dR$ is proved in detail; the other
cases can be treated by similar arguments. Using the assumption
$b\in\sigma_p(A)$ and Remark~\ref{rem5.7} one can assume that
$\omega_b$ is such that \eqref{pos2} holds, because
\[ \frac{Q_{b}(b)}{\gamma(a-b)} = \frac{\lim_{z \wh \to b} (b-z)r^\gamma_{a,b}(z)Q_{b}(z)}{\gamma^2(a-b)^2} \ge 0, \]
see Theorem~\ref{productinNg} (iv)(b). Since $r_{a,b}^\gamma(A)$ is a nonnegative
operator, the transformation \eqref{trsqroot} in
Lemma~\ref{trsqlemma} is well defined and it maps $\omega_b\in
\sH_{-1}(A,b)$ to an element
$\omega_{a}=(\rho_{b,a}^\gamma)^\frac{1}{2}(A)\omega_b \in
\sH_{-1}(A,a)$. The equation \eqref{fullevpt} with $\xi=a$ implies
that $Q_{a}$ can be written as
\begin{equation}\label{fullevpt2pr}
\begin{split} Q_{a}(\lambda)
&= \frac{b-a}{\gamma}((I-P_a^-)\omega_{a},(I-P_a^-)\omega_{a}) \\
&+ (\lambda-a)((I-P_\infty^-)\omega_{a},(I-P_\infty^-)\omega_{a}) \\
&+ (\lambda-a)\int_{\{b\}}\frac{\left(1+|t-a|\right)^2d\left(E(t) P_\infty V_a\omega_{a},V_a\omega_{a}\right)}{(t-\lambda)(t-a)} \\
&+ (\lambda-a)\int_{\dR\setminus\{a,b\}}\frac{\left(1+|t-a|\right)^2d\left(E(t) P_\infty V_a\omega_{a},V_a\omega_{a}\right)}{(t-\lambda)(t-a)}. \\
\end{split}
\end{equation}
The formulas \eqref{projprop1} and \eqref{trsqroot} imply the
identities $(I-P_a^-)\omega_{a} = |\gamma| (I-P_a^-)\omega_b$ and
$(I-P_b^-)\omega_{a} = |\gamma| (I-P_b^-)\omega_b$. This gives
\[
\begin{split}
Q_{a}(a) &= \frac{b-a}{\gamma}((I-P_a^-)\omega_{a},(I-P_a^-)\omega_{a}) = \gamma(b-a)((I-P_a^-)\omega_b,(I-P_a^-)\omega_b) \\ &=
r_{a,b}^\gamma(\lambda) (\lambda - b)\int_{\{a\}} \frac{t-b}{t-\lambda}d(E(t)P_\infty^-\omega_b,P_\infty^-\omega_b)\\ &=r_{a,b}^\gamma(\lambda) (\lambda-b)\int_{\{a\}}\frac{\left(1+|t-b|\right)^2d\left(E(t) P_\infty V_b\omega_b,V_b\omega_b\right)}{(t-\lambda)(t-b)}
\end{split}
\]
and by a similar calculation one obtains
\[ \begin{split} (\lambda-a)\int_{\{b\}}\frac{\left(1+|t-a|\right)^2d\left(E(t)
P_\infty V_a\omega_{a},V_a\omega_{a}\right)}{(t-\lambda)(t-a)}= &
r_{a,b}^\gamma(\lambda) Q_{b}(b).
\end{split}
\]
Moreover, if $\mul A\neq \{0\}$ then $r_{a,b}^\gamma(A)\ge 0$
implies that $\gamma>0$ (see \eqref{defRa}) and a direct calculation
shows that $(I-P_\infty^-)\omega_{a} = \gamma^\frac{1}{2}(I-P_\infty^-)\omega_b$ and
\[
(\lambda-a)((I-P_\infty^-)\omega_{a},(I-P_\infty^-)\omega_{a})
= r_{a,b}^\gamma (\lambda)(\lambda-b)((I-P_\infty^-)\omega_b,(I-P_\infty^-)\omega_b).
\]
Finally, using extended functional calculus one obtains
\[
\begin{split}
&\int_{\dR \setminus \{a,b\}}\frac{t-a}{t-\lambda}\frac{d\left(E_t P_\infty V_a(\rho_{b,a}^\gamma)^\frac{1}{2}(A)V_b^{-1}V_b \omega_b,V_a(\rho_{b,a}^\gamma)^\frac{1}{2}(A)V_b^{-1}V_b \omega_b\right)}{\left(1+|t-a|^{-1}\right)^{-2}} \\
&=\int_{\dR \setminus \{a,b\}}\frac{t-a}{t-\lambda}\phi(t)
\left(1+|t-b|^{-1}\right)^2 d\left(E_t P_\infty V_b \omega_b,V_b \omega_b\right),
\end{split}
\]
where for $t \neq a,b$
\[ \begin{split}
\phi(t) &= \gamma^2\left(1+ \frac{1}{\gamma} \frac{t-b}{t-a}\right)^2\left(\gamma \frac{t-a}{t-b}\right)\left( 1 + \gamma \frac{t-a}{t-b}\right)^{-2} = \gamma \frac{t-b}{t-a},
\end{split}
\]
cf. Lemma~\ref{trsqlemma}. Hence, the last term in
\eqref{fullevpt2pr} can be rewritten in the form
\[ \begin{split}
&(\lambda -a) \int_{\dR \setminus \{a,b\}}
 \frac{t-a}{t-\lambda}\phi(t) \left(1+|t-b|^{-1}\right)^2d\left(E_t P_\infty V_b \omega_b,V_b \omega_b\right) \\
& =  r_{a,b}^\gamma(\lambda)(\lambda-b) \int_{\dR \setminus \{a,b\}}
 \frac{(1+|t-b|)^2\, d\left(E_t P_\infty V_b \omega_b,V_b \omega_b\right)}{(t-\lambda)(t-b)}. \end{split}
\]
Combining all the above calculations yields the desired identity
$Q_{a}=r_{a,b}^\gamma Q_{b}$; see \eqref{fullevpt} with $\xi=b$.
Since $\omega_{a}\in \sH_{-1}(A,a)$, the inclusion $Q_{a}\in
\cN(a,1)$ is clear; cf. Proposition~\ref{vorige}.
\end{proof}

The formulas for the limit value $Q_{b}(b)$ in
Theorem~\ref{eerstestap} indicate the use of a (non-minimal)
selfadjoint relation $A$ with $b\in\sigma_p(A)$; note also the
similar formulas for $Q_{a}(a)$. Since $Q_{b}\in
\cN(b,1)$ this function does not have a pole at $b$, see Lemma~\ref{limvalue}.
Thus in a minimal realization of $Q_{b}$ the selfadjoint relation,
say $A_{min}(Q_{b})$ (see Remark~\ref{rem6.3} below), does not have $b$
in it's point spectrum. However, in general the product
$r_{a,b}^\gamma Q_{b}$ has a pole at $b$, unless $\lim_{z\wh\to
b} Q_{b}(z)=0$. Therefore, in a minimal realization for
$r_{a,b}^\gamma Q_{b}$ one typically has $b\in
\sigma_p(A_{min}(r_{a,b}^\gamma Q_{b}))$. By symmetry, similar
situation holds for $a$: $a\not\in \sigma_p(A_{min}(r_{a,b}^\gamma
Q_{b}))$, but it is possible that
$a\in\sigma_p(A_{min}(Q_{b}))$. Minimal realizations for
$Q_{b}$ and $r_{a,b}^\gamma Q_{b}$ include the information
about the possible point masses at $b$ and $a$; consequently the
explicit formulas connecting minimal realizations would get longer
and less readable, especially when the degree of $r$ increases. On
the other hand, it is easy to produce minimal realization for
$Q_{b}$ and $r_{a,b}^\gamma Q_{b}$ from the general
connection established in Theorem~\ref{eerstestap}; this is
explained in the next remark.

Furthermore, in Section~\ref{secL2} this connection between the minimal
realizations will be made explicit in specific model spaces; see
Theorem~\ref{l2finfin}.

\begin{remark}\label{rem6.3}
For $Q_{b}$ the minimal realization is obtained via the
corresponding $\gamma$-field using $\sH_{min}:=\cspan
\{\gamma^{\omega_b}_\lambda:\,\lambda\in\cmr \}$, cf.
\eqref{minimal}. The subspace $\sH_{min}$ reduces $A$ and
consequently, the rigged spaces $\sH_{\pm 1}(A,b)$ and $\sH_{\pm
1}(A,b)$ get decomposed accordingly. Automatically, $\omega_b\in
\sH_{min,-1}(A_{min},b)$. The transforms $r_{b,a}^\gamma(A)$ and
$(\rho_{b,a}^\gamma)^\frac{1}{2}(A)$ decompose accordingly.
\end{remark}

\begin{remark}\label{rem6.3trsbymeansofrhor}
Let $S_{b}$ and $S_{a}$ be the closed symmetric
relations associated to $A$ and $\omega_b$, and for $A$ and
$\omega_{a}:= (\rho_{b,a}^\gamma)^\frac{1}{2}(A)\omega_b$,
respectively, via Lemma~\ref{realizationprepinfty} and
\ref{realizationprepN1}. Then these symmetries can be extended to
$\wt S_{b}$ and $\wt S_{{a}}$ in the corresponding
rigged spaces by means of the transforms in \eqref{extendev}: for instance, if $a,b \in \dR$, then
\[
\begin{split}
(\wt S_{b}-b)^{-1} &= \{\{f,f'\} \in (A-b)^{-1}_b : (f,\omega_b)_b=0 \};\\
(\wt S_{{a}}-a)^{-1} &= \{\{f,f'\} \in (A-a)^{-1}_a : (f,\omega_{a})_a=0 \}.
\end{split}
\]
Taking into account the connection \eqref{connectselfadext} between
$(A-a)^{-1}_a$ and $(A-b)^{-1}_b$ in Lemma~\ref{trsqlemma} it
follows that
\[\begin{split}
P_a^-(\wt S_{{a}}-a)^{-1}P_aP_b &= \sgn(\gamma)(\rho_{b,a}^\gamma)^\frac{1}{2}(A)P_b^-(\wt S_{b}-b)^{-1}P_aP_b \left((r_{a,b}^{\gamma})^\frac{1}{2}(A)\right)^{-1};\\
P_a^-(\wt S_{{a}}^*-a)^{-1}P_aP_b &= \sgn(\gamma)(\rho_{b,a}^\gamma)^\frac{1}{2}(A)P_b^-(\wt S_{b}^*-b)^{-1}P_aP_b \left((r_{a,b}^{\gamma})^\frac{1}{2}(A)\right)^{-1}.
\end{split}\]
\end{remark}

\subsection{Realizations in the case of symmetric rational functions}
Assume that $Q \in \cN_0^0(r)$. Then according to
Proposition~\ref{CorN00a} the symmetric rational function $r$ can be
factorized as $r = \prod_{i=1}^n r_{a_i,b_i}^{\gamma_i}$, see
\eqref{defrabgamma}, where
\begin{equation}\label{condonfac3}
\left(\prod_{i=1}^j r_{a_i,b_i}^{\gamma_i}\right)Q  \in \cN, \quad j=1,\dots,n.
\end{equation}
The realizations for the functions $Q$ and $rQ$ are constructed in
this section by applying Theorem~\ref{eerstestap} combined with
Proposition~\ref{inductionstepprop} inductively to the product
functions $\left(\prod_{i=1}^j r_{a_i,b_i}^{\gamma_i}\right)Q  \in
\cN$ for $j=1,\dots,n$.

Recall that $\rho_{a,c}(A)$ with $a,c\in\dR
\cup\{\infty\}$, $a\neq c$, defined in Lemma~\ref{corgamma} maps
$\sH_{-1}(A,a)$ into $\sH_{-1}(A,c)\ominus\ker(A-c)$. In particular,
if $Q_{{a}}\in \cN(a,1)$ corresponds to $\omega_{a}\in
\sH_{-1}(A,a)$ (cf. Theorem~\ref{eerstestap}) and if, in addition,
$Q_{a}\in \cN(c,1)$ then the limit value $Q_{a}(c)\in\dR$ exists
and, as noted in Remark~\ref{rem5.16}, one can replace the vector
$\omega_{c}=\rho_{a,c}(A)\omega_{a}\in\sH_{-1}(A,c)$ by $\wt
\omega_c=\omega_c+e_c$, where $e_c\in\ker(A-c)$, in the realization
of $Q_{a}\in \cN(c,1)$. In the main theorem of this section the
vector $\wt \omega_c$ is selected such that the limit value
$Q_{a}(c)$ satisfies an analog of \eqref{pos2} with $c$ instead of
$b$; such a selection of $\wt\omega_{c}$ is expressed shortly by
using the notation
\begin{equation}\label{wtrho}
 \wt\omega_{c}=\wt\rho_{a,c}\omega_{a}\in\sH_{-1}(A,c),
\end{equation}
cf. Lemma~\ref{corgamma}.

\begin{theorem}\label{laatstestap}
Assume that $Q\in \cN_0^0(r)$ and let $\prod_{i=1}^n
r_{a_i,b_i}^{\gamma_i}$ be a factorization of $r$ such that
\eqref{condonfac3} holds. Let $A$ be a selfadjoint relation in the
Hilbert space $\{\sH,(\cdot,\cdot)\}$ such that
$\sigma(A) = \sigma(Q) \cup \{b_i \in \dR \cup \{\infty\}:
b_i\in \sigma_p(A), i=1,\dots, n\}$ and let $\omega \in
\sH_{-1}(A,b_1)$ be such that the model in
Lemma~\ref{realizationprepinfty} or \ref{realizationprepN1} realizes
$Q$.

Then with a vector $\omega_e$ of the form
\begin{equation}\label{omega-e}
 \omega_e = (\rho_{b_n,a_n}^{\gamma_n})^\frac{1}{2}(A)
 \prod_{i=1}^{n-1}\left(\wt\rho_{a_{i},b_{i+1}}(A)(\rho_{b_i,a_i}^{\gamma_i})^\frac{1}{2}(A)\right)\omega,
\end{equation}
where $\wt\rho_{a_{i},b_{i+1}}(A)$ and
$(\rho_{b_i,a_i}^{\gamma_i})^\frac{1}{2}(A)$ are defined as in
\eqref{wtrho} and \eqref{trsqroot}, the model in
Lemma~\ref{realizationprepinfty} or \ref{realizationprepN1}
associated to $A$, $\omega_e \in \sH_{-1}(A,a_n)$, and $\eta_e =
\lim_{\lambda \wh \to a_n} r(\lambda)Q(\lambda)$ realizes the
function $rQ$. In particular, $rQ\in\cap_{i=1}^n\cN(a_i,1)$.
\end{theorem}
\begin{proof}
By the assumption on $A$ all the products $\prod_{i=1}^j
r_{a_i,b_i}^{\gamma_i}(A)$ are nonnegative operators. An application
of Theorem~\ref{eerstestap} with
$\omega_{a_1}=(\rho_{b_1,a_1}^{\gamma_i})^\frac{1}{2}(A)\omega$
gives a realization for the function $r_{a_1,b_1}^{\gamma_1}Q\in
\cN(a_1,1)$ and, by the selection of the factorization for $r$, one
actually has $r_{a_1,b_1}^{\gamma_1}Q\in
\cN_0^0(r_{a_2,b_2}^{\gamma_2})$. In particular,
$r_{a_1,b_1}^{\gamma_1}Q\in \cN(1,b_2)$ and one can define
$\omega_{b_2}=\wt\rho_{a_1,b_2}\omega_{a_1}\in\sH_{-1}(A,b_2)$, such
that $r_{a_1,b_1}(b_2)Q(b_2)$ satisfies the formula \eqref{pos2}
with $r_{a,b}^\gamma$ replaced by $r_{a_2,b_2}^{\gamma_2}$; see
\eqref{wtrho}. Now, one can apply Theorem~\ref{eerstestap} to the
product $r_{a_2,b_2}^{\gamma_2}\left(r_{a_1,b_1}^{\gamma_1}Q\right)$
to get a desired realization for this product. Now by proceeding
inductively one gets the stated realization for the product function
$rQ\in\cN(a_n,1)$ with $\omega_e \in \sH_{-1}(A,a_n)$ of the form
\eqref{omega-e} and $\eta_e = \lim_{\lambda \wh \to a_n}
r(\lambda)Q(\lambda)\in\dR$. Finally, the fact that
$rQ\in\cap_{i=1}^n\cN(a_i,1)$ is a direct consequence of
Proposition~\ref{vorige}.
\end{proof}

The integral representation of the functions $Q$
and $rQ$ can be obtained by means of $\omega$ and $\omega_e$
due to their connection to the underlying spectral measures
$d\sigma(t)$ and $d\sigma_{e}(t)$ stated in
Remark~\ref{int2}. Therefore note that in case that $b_i$ and $a_i$ are finite the spectral measure
$d\sigma_e$ of $Q_{e}$ is given by
\[ d\sigma_e(t) = (1+|t-a_n|^2)d(E_tP_{a_n} P_\infty V_{a_n}\omega_e,V_{a_n}\omega_e),\]
see Remark~\ref{int2}. Using extended functional calculus this can be rewritten as
\[ \begin{split}
d\sigma_e(t) &= |r(t)|(1+|t-b_1|^2)d(E_tP_{b_1}\cdots P_{b_n} P_{a_n} V_{b_1}\omega,V_{b_1}\omega) + \sum_{i=1}^n \zeta_i \delta_{b_i}(t)\\
&= r(t)d\sigma(t) + \sum_{i=1}^n \zeta_i \delta_{b_i}(t),
\end{split} \]
where $d\sigma$ is the spectral measure of $Q$ and $\zeta_i \in
\dR_+$, $1\leq i\leq n$. See also Section~\ref{secL2}
where $L^2(d\sigma)$-realizations of the functions $Q$ and $rQ$ are considered.

Theorem~\ref{laatstestap} contains the main result for describing
the realization for the functions belonging to the class
$\cN_0^{0}(r)$ with some symmetric rational function $r$. This
result can be used to describe for $Q \in \cN_\kappa^{\wt
\kappa}(r)$ the connection between the realizations for the
functions $Q$ and $rQ$. Therefore let the canonical factorizations
of $Q\in\cN_\kappa$ and $rQ\in\cN_{\wt\kappa}$ be written in the
form
\begin{equation}\label{cfact}
 Q=\varphi \varphi^\sharp Q_0,\quad rQ=\wt\varphi\wt\varphi^\sharp \wt
 Q_0;
 \qquad \varphi^\sharp(\lambda)=\overline{\varphi(\bar{\lambda})},
 \quad \wt\varphi^\sharp(\lambda)=\overline{\wt\varphi(\bar{\lambda})}.
\end{equation}
Then
\begin{equation}\label{QQ0}
 \wt Q_0=\wt r Q_0,\quad \wt r=\frac{r\varphi
 \varphi^\sharp}{\wt\varphi\wt\varphi^\sharp},
\end{equation}
and consequently $Q_0\in \cN_0^0(\wt r)$; cf. \eqref{cnkkcn00}. Now
Theorem~\ref{laatstestap} guaranties that the functions $Q_0$ and
$\wt Q_0=\wt r Q_0$ can be realized by means of the elements
$\omega\in\sH_{-1}(A,b)$ and $\wt \omega\in\sH_{-1}(A,a)$ for some
$a,b\in\dR\cup\{\infty\}$ using the same selfadjoint relation $A$ in
a Hilbert space $\{\sH,(\cdot,\cdot)\}$. The realizations for the
original functions $Q$ and $rQ$ can be obtained using the model
based on their canonical factorizations in \eqref{cfact}. To
describe this consider the matrix functions
\begin{equation}\label{RR}
  \Phi(\lambda)=\begin{pmatrix} 0 & \varphi(\lambda) \\ \varphi^\sharp(\lambda) & 0 \end{pmatrix},
  \quad
  \wt \Phi(\lambda)=\begin{pmatrix} 0 & \wt\varphi(\lambda) \\ \wt\varphi^\sharp(\lambda) & 0
  \end{pmatrix}.
\end{equation}
Then $\Phi\in\cN_\kappa$ with $\kappa=\deg \varphi$ and $\wt
\Phi\in\cN_{\wt\kappa}$ with $\wt\kappa=\deg \wt\varphi$. Let
$A_\Phi$ and $A_{\wt \Phi}$ be selfadjoint relations in the
reproducing kernel Pontryagin spaces $\sH_\Phi$ and $\sH_{\wt \Phi}$,
which realize the functions $\Phi$ and $\wt \Phi$ with the $\gamma$-fields
$\gamma_\lambda^\Phi$ and $\wt\gamma_\lambda^{\wt \Phi}$,
respectively. Now the realization for the functions $Q$ and $\wt
Q=rQ$ can be obtained via the next result, which is formulated by
means of $\{A,\gamma_\lambda\}$-realizations, see Remark~\ref{Qfunctions}, using extensions of the
orthogonal sums $S_0\oplus S_\Phi$ and $\wt S_0\oplus \wt S_{\wt
\Phi}$ of the corresponding symmetric restrictions determined by
\eqref{recoversym}.

\begin{proposition}\label{DHmoldel}
Let $Q \in \cN_\kappa^{\wt\kappa}(r)$, let the canonical
factorizations of $Q$ and $rQ$ be given by \eqref{cfact}, and let
$Q_0$ and $\wt Q_0=\wt r Q_0$ be realized as in
Theorem~\ref{laatstestap} with the $\gamma$-fields
$\gamma^0_\lambda$ and $\wt\gamma^0_\lambda$, respectively.
Moreover, let $\Phi$ and $\wt \Phi$ be realized via the pairs
$\{A_\Phi,\gamma_\lambda^\Phi\}$ and $\{A_{\wt
\Phi},\wt\gamma_\lambda^{\wt \Phi}\}$ in the Pontryagin spaces
$\sH_\Phi$ and $\sH_{\wt \Phi}$, respectively.

Then the functions $Q$ and $\wt Q=rQ$ can be realized in the Pontryagin spaces $\sH\oplus
\sH_\Phi$ and $\sH\oplus \sH_{\wt \Phi}$ via the pairs $\{A_0,\gamma_\lambda\}$
and $\{\wt A_0,\wt\gamma_\lambda\}$, where the $\gamma$-fields
$\gamma_\lambda$ and $\wt\gamma_\lambda$ are given by
\begin{equation}\label{gg0}
\gamma_\lambda
 =\begin{pmatrix} \gamma^0_\lambda \varphi \\
 \gamma_\lambda^\Phi \binom{Q_0\varphi}{1}
 \end{pmatrix}
 \quad \text{and} \quad
\wt\gamma_\lambda
 =\begin{pmatrix} \wt\gamma^0_\lambda \wt\varphi \\
  \wt\gamma_\lambda^{\wt \Phi} \binom{\wt Q_0\wt \varphi}{1}
  \end{pmatrix},
\end{equation}
and where $A_0$ and $\wt A_0$ are selfadjoint extensions of
$S_0\oplus S_\Phi$ and $\wt S_0\oplus \wt S_{\wt \Phi}$ satisfying
the associated $\gamma$-field formulas in \eqref{chargamma},
respectively.
\end{proposition}
\begin{proof}
By symmetry it suffices to prove the statement for $Q$. It follows
from the orthogonal sum construction that $\gamma_\lambda$ satisfies
the formula \eqref{chargamma} with some selfadjoint relation $\wt
A_0$, which is a selfadjoint extension of $S_0\oplus S_\Phi$; see
\cite[Theorem~3.3]{DH03}. On the other hand, a straightforward
calculation using the matrix formula for $\gamma_\lambda$ in
\eqref{gg0} and the definition of $\Phi(\lambda)$ in \eqref{RR}
yields the identity
\[
\begin{split}
 (\gamma_\lambda,\gamma_\mu) &=  \overline{\varphi(\mu)}\,\frac{Q_0(\lambda)-Q_0(\overline{\mu})}{\lambda -\overline{\mu}}\, \varphi(\lambda) +
 \binom{Q_0(\mu)\varphi(\mu)}{1}
 \frac{\Phi(\lambda)-\Phi(\overline{\mu})}{\lambda -\overline{\mu}}
 \binom{Q_0(\lambda) \varphi(\lambda)}{1}
 \\ &=\frac{\varphi^\sharp(\lambda)Q_0(\lambda) \varphi(\lambda)-\overline{\varphi^\sharp(\mu)Q_0(\mu)
 \varphi(\mu)}}
 {\lambda-\bar\mu}
 \end{split}
\]
for $\lambda,\mu\in \rho(A\oplus A_\Phi)$. Therefore, $\{A_\Phi,\gamma_\lambda^\Phi\}$ realizes the function
$Q=\varphi^\sharp Q_0 \varphi$.
\end{proof}
The realizations used in Proposition~\ref{DHmoldel} rely on the
canonical factorizations in \eqref{cfact}. The idea of factorization
models and the above orthogonal sum approach via matrix functions as
in \eqref{RR} goes back to \cite{DH03}. In particular, the
construction of the $\gamma$-fields in \eqref{gg0} is based on
\cite[Theorem~3.3]{DH03}. Also the selfadjoint realizations $A_0$
and $\wt A_0$ can be made explicit by means of boundary triplets:
If, for instance, $\{\dC,\Gamma^0_0,\Gamma^0_1\}$ and
$\{\dC^2,\Gamma^\Phi_0,\Gamma^\Phi_1\}$ are boundary triplets for
$S_0^*$ and $S_\Phi^*$ with Weyl-functions $Q_0$ and $\Phi$, then
$A_0$ is given by the following boundary conditions:
\[
 A_0=\{ \wh f_0\oplus \wt F\in S_0^*\oplus S_\Phi^*:\,
 \Gamma^0_0\wh f_0=(\Gamma^\Phi_1\wh F)_1,\,
 \Gamma^0_1\wh f_0=(\Gamma^\Phi_0\wh F)_1,\,
 (\Gamma^\Phi_0\wh F)_2=0 \},
\]
where e.g. $(\Gamma^\Phi_0\wh F)_1$ and $(\Gamma^\Phi_0\wh F)_2$
stand for the components of $\Gamma^\Phi_0\wh F\in\dC^2$; cf.
\cite[Theorem~3.3]{DH03}.

The non-minimality of the realizations in
Proposition~\ref{DHmoldel} can only be due to a finite number of
possible point masses in the underlying spectral measure of the
functions (or, equivalently, eigenvalues of the associate symmetric
relations determined by \eqref{recoversym}), which are canceled due
to multiplications by rational functions appearing in various
factorizations. This means that minimality of models can be easily
obtained by checking if the functions $Q$ and $rQ$ actually have
poles or zeros at points, where the rational factors admit poles and
zeros.

\subsection{$L^2(d\sigma)$-realizations}\label{secL2}
The abstract realizations from the previous section will be
supplemented by explicit (and minimal) realizations in the form of
$L^2(d\sigma)$-models for Nevanlinna functions. With the help of
these models the connection between the realizations for $Q$ and
$rQ$, $Q \in \cN_0^{0}(r)$, as given in Theorem~\ref{laatstestap} is
made more accessible. The reason is that the associated riggings as
well as the mappings $(\rho_{\xi,\xi'})(A)$ and
$(\rho_{b,a}^\gamma)^\frac{1}{2}(A)$ appearing in
Lemma~\ref{corgamma} and \ref{trsqlemma}, are easily expressed in
$L^2(d\sigma)$-spaces (without regularizations). For instance, the
formula for $(\rho_{\xi,\xi'})(A)$ ($\xi, \xi' \in \dR
\cup\{\infty\}$, $\xi \neq \xi'$) in the $L^2(d\sigma)$ setting is
given by the expression
\[ 
 \rho_{\xi,\xi'}(t)f(t) = \mathbf{1}^c_{\{\xi'\}} \frac{t-\xi}{t-\xi'}f(t),
\]
with domain being described now by proper integrability conditions.
Here, and in what follows, for a subset $E$ of $\dR \cup \{\infty\}$
the notations $\mathbf{1}_E$ and $\mathbf{1}^c_E$ stand for the characteristic functions
of the sets $E\cap \dR$ and $\dR \setminus \{E \cap \dR\}$, respectively.

Let $\beta \geq 0 $ and let $d\sigma$ be a nonnegative measure
satisfying $\int_\dR d\sigma(t)/(1+t^2) < \infty$. Associate with
$\beta$ and $\sigma$ the Hilbert space $\sH_{\sigma,\beta}:=
(L^2(d\sigma)\times \dC)/(\cdot,\cdot)$, where
\[ 
(f\times f_\infty,g\times g_\infty) = g_\infty^*\beta f_\infty + \int_\dR g(t)^*f(t)d\sigma(t), \quad f\times f_\infty, g\times g_\infty \in L^2(d\sigma)\times \dC.
\]
Thus $\sH_{\{\sigma,\beta\}}$ is $L^2(d\sigma) \times \dC$ if $\beta
\neq 0$ and $\sH_{\{\sigma,\beta\}}$ is equivalent to $L^2(d\sigma)$
if $\beta=0$. In this space the multiplication operator with the
independent variable $A_{\sigma,\beta}$,
\begin{equation}\label{Aasmult}
A_{\sigma,\beta} = \{ \{f(t)\times 0,tf(t)\times f_\infty\} \in \sH^2_{\sigma,\beta}\},
\end{equation}
is a selfadjoint relation. Note that
$\sigma(A_{\sigma,\beta}) = \supp (\sigma)$ if $\beta =
0$ and $\sigma(A_{\sigma,\beta}) = \supp (\sigma) +
\{\infty\}$ if $\beta \neq 0$. Let $\sH_{+1}(A_{\sigma,\beta},\xi)
\subset \sH_{\sigma,\beta} \subset \sH_{-1}(A_{\sigma,\beta},\xi)$
be the rigging of $\sH_{\sigma,\beta}$ with respect to
$A_{\sigma,\beta}$ at $\xi$ as defined in Definition~\ref{rigdef},
see also \cite{HS97} ($\xi = \infty$) and \cite{HKS98} ($\xi=0$).
Then the spaces $\sH_{+1}(A_{\sigma,\beta},\xi)$ and
$\sH_{-1}(A_{\sigma,\beta},\xi)$ are weighted $L^2(d\sigma)$ spaces
and the duality $(\cdot,\cdot)_\xi$ associated with the above
rigging takes for $f\times f_\infty \in
\sH_{+1}(A_{\sigma,\beta},\xi)$ and $g \times g_\infty \in
\sH_{-1}(A_{\sigma,\beta},\xi)$ the value
\begin{equation}\label{l2duality}
(f\times f_\infty, g\times g_\infty)_\xi = g_\infty^*\beta f_\infty + \int_\dR g(t)^*f(t)d\sigma(t).
\end{equation}
The introduced rigging can be used to give explicit realizations for
the subclass $\cN(\xi,1)$, $\xi \in \dR \cup \{\infty\}$ by
interpreting Lemma~\ref{realizationprepN1} and
\ref{realizationprepinfty}. In particular, if for a Nevanlinna
function $d\sigma$ is taken to be its spectral measure, $\beta$ its
non-tangential limit at $\infty$ as in Lemma~\ref{pointlimit} and
$\omega$ is taken to be $1/(t-\xi)\times 1$, if $\xi \in \dR$, and
$1\times 1$, if $\xi =\infty$, then the corresponding realization is
minimal. Note that the case $\xi=\infty$ has been studied in the
$L^2(d\sigma)$ setting in \cite{HKS97} and \cite{HS97}, and that in
\cite{DM95,HLS95} a realization by means of $L^2(d\sigma)$-models
for all Nevanlinna functions is given.

\begin{theorem}\label{l2finfin}
Let $Q \in \cN_0^0(r)$ and denote the zeros and poles of $r$ by
$a_i$ and $b_i$, $1\leq i \leq n$, where $a_n=\infty$ or
$b_n=\infty$ if $\infty$ is a zero or pole of $r$.

Let $Q$ be realized by the boundary triplet
$\{\dC,\Gamma_0^{\omega,\eta_{b_1}},\Gamma_1^{\omega,\eta_{b_i}}\}$
associated with $S_\omega^*$ in the Hilbert space
$\sH_{\sigma,\beta}$, where $\omega \times 1 = 1/(t-b_1) \times 1\in
\sH_{-1}(A_{\sigma,\beta}, b_1)$ if $b_1\in\dR$ and $\omega \times 1
= 1 \times 1\in \sH_{-1}(A_{\sigma,\beta}, \infty)$ if $b_1=\infty$;
cf. Theorem~\ref{realizationN1}. With $\zeta_{b_i}= \lim_{\lambda
\wh \to b_i} (b_i-\lambda)r(\lambda)Q(\lambda)$ and
$\zeta_\infty=\lim_{\lambda \wh \to \infty}
r(\lambda)Q(\lambda)/\lambda$, define $d\sigma_e$, $\beta_e$ and
$\omega_e$ as
\[ \left\{ \begin{array}{lll}
d\sigma_e(t)= \mathbf{1}^c_{\{a_1,\ldots,a_n\}}d\sigma(t) + \sum_{i=1}^n  \delta_{b_i}, & \beta_e = \gamma \beta,& b_n,a_n \in \dR;\\
d\sigma_e(t)= \mathbf{1}^c_{\{a_1,\ldots, a_n\}}d\sigma(t) + \sum_{i=1}^{n-1}  \delta_{b_i}, &\beta_e = \zeta_{b_n} ,& b_n = \infty, a_n \in \dR;\\
d\sigma_e(t)= \mathbf{1}^c_{\{a_{1},\ldots, a_{n-1}\}}d\sigma(t) + \sum_{i=1}^n  \delta_{b_i},& \beta_e = 0,& b_n \in \dR, a_n =\infty;
\end{array} \right. \]
and
\[ \omega_e(t) = \left\{ \begin{array}{rl}
\dfrac{\mathbf{1}^c_{\{b_1,\ldots, b_n\}}}{t-a_n}\sqrt{|r(t)|} + \sum_{i=1}^n \dfrac{\mathbf{1}_{\{b_i\}}\sqrt{\zeta_{b_i}}}{b_i-a_n},& b_n,a_n \in \dR;\\
\dfrac{\mathbf{1}^c_{\{b_1,\ldots, b_{n-1}\}}}{t-a_n}\sqrt{|r(t)|}+ \sum_{i=1}^{n-1} \dfrac{\mathbf{1}_{\{b_i\}}\sqrt{\zeta_{b_i}}}{b_i-a_n},& b_n = \infty,a_n \in \dR;\\
\mathbf{1}^c_{\{b_1,\ldots,b_n\}}\sqrt{|r(t)|}  + \sum_{i=1}^n \sqrt{\mathbf{1}_{\{b_i\}}\zeta_{b_i}},& b_n \in \dR, a_n =\infty.
\end{array} \right. \]
Then $\omega_e \times 1 \in \sH_{-1}(A_{\sigma_e,\beta_e}, a_n)$
and, with $\eta_{a_n} = \lim_{\lambda \wh \to a_n}
r(\lambda)Q(\lambda)$, the Nevanlinna function $rQ$ is realized by
the boundary triplet
$\{\dC,\Gamma_0^{\omega_e,\eta_{a_n}},\Gamma_1^{\omega_e,\eta_{a_n}}\}$
associated with $S_{\omega_e}^*$ in the Hilbert space
$\sH_{\sigma_e,\beta_e}$, cf. Theorem~\ref{realizationN1}
\end{theorem}

\begin{proof}[Sketch of the proof]
Let $d\sigma$ be a nonnegative measure satisfying $\int_\dR
d\sigma(t)/(1+t^2) < \infty$, let $\beta \geq 0$ and let $\omega
\times 1 \in \sH_{-1}(A_{\sigma,\beta}, \xi)$, where
$A_{\sigma,\beta}$ is as in \eqref{Aasmult} and $\xi \in \dR \cup
\{\infty\}$. Then the spectral measure $d\sigma_\omega$ of the Weyl
function realized by $A_{\sigma,\beta}$ and $\omega \times 1$ via
Lemma~\ref{realizationprepinfty} ($\xi =\infty$) or
Lemma~\ref{realizationprepN1} ($\xi \in \dR$) is given by
\[d\sigma_\omega(t)= \left\{ \begin{array}{rl}  (t-\xi)^2|\omega(t)|^2d\sigma(t),& \xi \in \dR, \\  |\omega(t)|^2 d\sigma(t),& \xi = \infty,
  \end{array} \right. \quad (t\neq \xi).
 \]
By means of this observation, it can be easily seen that the
spectral measure of the Weyl function associated with the boundary
triplet or boundary relation corresponding to $\omega_e \times 1$
and $A_{\sigma_e,\beta_e}$ (at $a_n$) is the spectral measure of
$rQ$, which outside the poles of $r$ is given by
$r(t)d\sigma(t)=|r(t)|d\sigma(t)$; see \eqref{genStieltjesinv}.
\end{proof}

Note that the zeros $a_i$ and the poles $b_i$ of $r$ in
Theorem~\ref{l2finfin} can be either of order one or of order two,
cf. Theorem~\ref{Newth}.

The models in Theorem~\ref{l2finfin} are minimal, see the
discussion following \eqref{l2duality}, and, hence, different from
Theorem~\ref{laatstestap} the spaces used in the realizations of $Q$
and $rQ$ will, in general, differ from each other. The equivalent of
the above minimal model could also be constructed in the abstract case.
In that case the coupling method should be used to
properly treat point masses; here this is reflected by the fact that
the measure changes due to the point masses as indicated in
Theorem~\ref{l2finfin}.


\begin{thebibliography}{10}
\newcommand{\enquote}[1]{``#1''}

\bibitem{Berezanskii}
Ju.~M.~Berezanskii, {Expansions in eigenfunctions of selfadjoint
operators
  (Russian)}, Naukova Dymka, Kiev, 1965. English translation: Transl. Math. Monographs,
  vol. 17, Amer. Math. Soc., 1968.

\bibitem{Bruk}
V.M.~Bruk, \enquote{A certain class of boundary value problems with
a spectral parameter in the boundary condition (Russian)}, {Mat. Sb.
(N.S.)}, \text{100 (142), no. 2} (1976), 210--216.

\bibitem{Der95}
V.A.~Derkach, \enquote{On Weyl function and generalized resolvents of a
  Hermitian operator in a Kre\u{\i}n space}, {Integr. Eq. Oper. Th.},
  \text{23} (1995), 387--415;

\bibitem{DH03}
V.A.~Derkach and S.~Hassi, \enquote{A reproducing kernel space model for
  $\mathcal{N}_\kappa$-functions}, {Proc. Amer. Math. Soc.}, \text{131}
  (2003), 3795--3806.

\bibitem{DHMS06}
V.A. Derkach, S.~Hassi, M.M. Malamud and H.S.V. de~Snoo, \enquote{Boundary
  relations and their Weyl families}, {Trans. Amer. Math. Soc.},
  \text{358} (2006), 5351--5400.

\bibitem{DHMS09}
V.A. Derkach, S.~Hassi, M.M. Malamud,  and H.S.V. de~Snoo, \enquote{Boundary
  relations and generalized resolvents of symmetric operators}, {Russian
  Journal of Mathematical Physics}, \text{16 no. 1} (2009), 17--60.

\bibitem{DHS99}
V.A.~Derkach, S.~Hassi and H.S.V. de~Snoo, \enquote{Operator models associated
  with Kac subclasses of generalized Nevanlinna functions}, {Methods of
  Funct. Anal. and Top.}, \text{5} (1999), 65--87.



\bibitem{DM91}
V.A.~Derkach and M.M.~Malamud, \enquote{Generalized resolvents and the boundary
  value problems for Hermitian operators with gaps}, {J. Funct. Anal.},
  \text{95} (1991), 1--95.

\bibitem{DM95}
V.A.~Derkach and M.M.~Malamud, \enquote{The extension theory of
Hermitian operators and the moment problem}, {J. Math. Sci.},
  \text{73,  no. 2} (1995) , 141--242.

\bibitem{DM97}
V.A.~Derkach and M.M.~Malamud, \enquote{On some classes of
holomorphic operator functions with nonnegative imaginary part},
{Operator theory, operator algebras and related topics
(Timi\c{s}oara, 1996)}, Theta Found., Bucharest (1997) , 113--147.


\bibitem{DLLS00}
A.~Dijksma, H.~Langer, A.~Luger and Yu~Shondin, \enquote{A factorization result
  for generalized Nevanlinna function of the class $\textbf{N}_\kappa$},
  {Integral Equations Operator Theory}, \text{36} (2000), 121--125.

\bibitem{Don74}
W.F.~Donoghue, {Monotone matrix functions and analytic
continuation}, Springer-Verlag, Berlin, Heidelberg, New York, 1974.

\bibitem{HKS97}
S.~Hassi, M.~Kaltenb{\"a}ck and H.S.V. de~Snoo, \enquote{Triplets of Hilbert
  spaces and Friedrichs extensions associated with the subclass $\textbf{N}_1$
  of Nevanlinna functions}, {J. Operator Theory}, \text{37} (1997),
  155--181.

\bibitem{HKS98}
S.~Hassi, M.~Kaltenb{\"a}ck and H.S.V. de~Snoo, \enquote{Generalized
Kre\u{\i}n-von Neumann extensions and associated operator models},
{Acta Sci. Math. (Szeged)}, \text{64}, (1998), 627--655.

\bibitem{HLS95}
S.~Hassi, H.~Langer and H.S.V. de~Snoo, \enquote{Selfadjoint extensions for a
  class of symmetric operators with defect numbers (1,1)}, {15th OT
  Conference Proceedings}, 1995, pp. 115--145.

\bibitem{HL06}
S.~Hassi and A.~Luger, \enquote{Generalized zeros and poles of
  $\mathcal{N}_\kappa$-functions: on the underlying spectral structure},
  {Methods Funct. Anal. Topology}, \text{12, no. 2} (2006), 131--150.


\bibitem{HSSW08}
S. Hassi, A. Sandovici, H.S.V. de Snoo, and H. Winkler,
\enquote{One-dimensional perturbations, asymptotic expansions, and
spectral gaps}, {Oper. Theory Adv. Appl.}, \text{188} (2008),
149--173.

\bibitem{HS97}
S.~Hassi and H.S.V.~ de Snoo, \enquote{One-dimensional graph
perturbations of selfadjoint relations}, {Ann. Acad. Sci. Fenn. A.I.
Math.}, \text{22} (1997), 123--164.

\bibitem{Ka56}
I.S.~Kac, \enquote{On integral representations of analytic functions mapping
  the upper half-plane onto a part of itself}, {Uspekhi Mat. Nauk},
  \text{11} (1956), 139--144.

\bibitem{KK74}
I.S.~Kac and M.G.~Kre\u{\i}n, \enquote{$R$-functions - analytic functions
  mapping the upper halfplane into itself}, {Amer. Math. Soc. Transl.},
  \text{103} (1974), 1--18.

\bibitem{KWW06}
M.~Kaltenb{\"a}ck, H.~Winkler and H.~Woracek, \enquote{Generalized Nevanlinna
  functions with essentially positive spectrum}, {J. Operator Theory},
  \text{55} (2006), 17--48.



\bibitem{KL71}
M.G.~Kre\u{\i}n and H.~Langer, \enquote{The defect subspaces and
generalized resolvents of Hermitian operator in Pontryagin space},
{Funkts. Anal. i Prilozhen}, \text{5, no. 2} (1971), 59--71; ibid.
\text{5, no. 3} (1971), 54--69 (Russian). English translation:
{Funct. Anal. Appl.} \text{5} (1971), 136--146; ibid. \text{5}
(1971), 217--228.

\bibitem{KL73}
M.G.~Kre\u{\i}n and H.~Langer, \enquote{\"Uber die $Q$-Funktion
eines $\pi$-Hermiteschen Operators im Raume $\Pi _{\kappa }$
(German)}, {Acta Sci. Math. (Szeged)}, \text{34} (1973), 191--230.

\bibitem{KL75}
M.G. Kre\u{\i}n and H. Langer, \enquote{\"Uber einige
Fortsetzungsprobleme, die eng mit der Theorie hermitescher
Operatoren im Raume $\Pi_\kappa$ zusammenh\"angen. I. Einige
Funktionklassen und ihre Dahrstellungen}, {Math. Nachr.},
 \text{77} (1977), 187--236.

\bibitem{KNapp}
M.G. Kre\u{\i}n and A.A. Nudelman, {The Markov moment problem and
extremal problems}, ''Nauka'', Moscow, 1973 (Russian). English
translation: Translation of Mathematical monographs AMS, \text{vol.
40}, 1977.

\bibitem{Kochu}
A.N.~Kochubei, \enquote{Characteristic functions of symmetric
operators and their extensions}, (Russian) Izv. Akad. Nauk Armyan.
SSR Ser. Mat. 15, no. 3 (1980), 219-–232, 247.

\bibitem{LW98}
H. Langer and H. Winkler, \enquote{Direct and inverse spectral
problems for generalized strings}, {Integr. equ. oper. theory},
\text{30} (1998), 409--431.

\bibitem{Wa64}
B.~van der Waerden, {Modern Algebra}, Fredrick Ungar, NY, 1964.
\end{thebibliography}
\end{document}